\documentclass[11pt]{amsart}

\usepackage{amsmath, amsthm, amssymb, amsfonts, enumerate,stmaryrd}
\usepackage[colorlinks=true,linkcolor=blue,urlcolor=blue]{hyperref}
\usepackage{dsfont}
\usepackage{color}
\usepackage{geometry}
\usepackage{todonotes}
\usepackage{soul}
\usepackage{graphicx}
\usepackage{caption}
\usepackage{mathtools}

\usepackage{appendix}

\usepackage{tikz}
\tikzset{every node/.style={align=center}}

\mathtoolsset{showonlyrefs}

\makeatletter
\@namedef{subjclassname@2020}{%
  \textup{2020} Mathematics Subject Classification}
\makeatother

\newcommand\cA{\mathcal{A}}
\newcommand\cB{\mathcal{B}}

\newcommand\cE{\mathcal{E}}
\newcommand\cF{\mathcal{F}}
\newcommand\cG{\mathcal{G}}
\newcommand\cH{\mathcal{H}}

\newcommand\cK{\mathcal{K}}
\newcommand\cL{\mathcal{L}}

\newcommand\cN{\mathcal{N}}
\newcommand\cO{\mathcal{O}}
\newcommand\cP{\mathcal{P}}
\newcommand\cQ{\mathcal{Q}}

\newcommand\cS{\mathcal{S}}
\newcommand\cT{\mathcal{T}}

\newcommand\cV{\mathcal{V}}

\newcommand\cX{\mathcal{X}}

\newcommand\bE{\mathbb{E}}

\newcommand\bN{\mathbb{N}}

\newcommand\bP{\mathbb{P}}
\newcommand\bQ{\mathbb{Q}}
\newcommand\bR{\mathbb{R}}

\def \ud{\mathrm{d}}
\def \e{\mathrm{e}}
\def \AA{\mathrm{Adm}}

\newcommand{\eps}{\varepsilon}
\newcommand{\ind}{\mathbf{1}}

\newtheorem{theorem}{Theorem}[section]
\newtheorem{lemma}[theorem]{Lemma}

\newtheorem{proposition}[theorem]{Proposition}

\newtheorem{remark}[theorem]{Remark}
\newtheorem{assumptions}[theorem]{Assumption}

\title[DPP for classical and singular control with discretionary stopping]{Dynamic programming principle \\ for classical and singular stochastic control \\ with discretionary stopping}
\author[T.\ De Angelis and A.\ Milazzo]{Tiziano De Angelis \and Alessandro Milazzo}
\subjclass[2020]{60G07, 60G40, 93E20, 49L20}
\keywords{dynamic programming principle, stochastic control, singular control, discretionary stopping}
\address{T.\ De Angelis: School of Management and Economics, Dept.\ ESOMAS, University of Torino, C.so Unione Sovietica 218bis, 10134, Torino, ITALY; Collegio Carlo Alberto, P.za Arbarello 8, 10122, Torino, ITALY.}
\email{\href{mailto:tiziano.deangelis@unito.it}{tiziano.deangelis@unito.it}}
\address{A.\ Milazzo: Department of Mathematics, Uppsala University, Box 480, 75106, Uppsala, SWEDEN.}
\email{\href{mailto:alessandro.milazzo@math.uu.se}{alessandro.milazzo@math.uu.se}}

\numberwithin{equation}{section}

\begin{document}
	
\begin{abstract}
We prove the dynamic programming principle (DPP) in a class of problems where an agent controls a $d$-dimensional diffusive dynamics via both classical and singular controls and, moreover, is able to terminate the optimisation at a time of her choosing, prior to a given maturity. The time-horizon of the problem is random and it is the smallest between a fixed terminal time and the first exit time of the state dynamics from a Borel set. We consider both the cases in which the total available fuel for the singular control is either bounded or unbounded. 

We build upon existing proofs of DPP and extend results available in the traditional literature on singular control (Haussmann and Suo, SIAM J.\ Control Optim., 33, 1995) by relaxing some key assumptions and including the discretionary stopping feature. We also connect with more general versions of the DPP (e.g., Bouchard and Touzi, SIAM J.\ Control Optim., 49, 2011) by showing in detail how our class of problems meets the abstract requirements therein.  
\end{abstract}
	
\maketitle
	
\section{Introduction}

In this paper, we prove the dynamic programming principle (DPP) for problems of the form:
\[
\sup_{\tau,\alpha,\xi}J(\tau,\alpha,\xi)
\]
where $\tau$ is a stopping time and $\alpha$ and $\xi$ are, respectively, a classical control and a singular control. The objective function $J(\tau,\alpha,\xi)$ is in the form of an expected reward that depends on the path of the stochastic dynamics of a controlled diffusion process $X^{\alpha,\xi}$ and on the amount of control exerted. The optimisation runs over a random time horizon which is determined as the smallest between a deterministic time $T$ and the first time the state-dynamics leaves a given Borel set $\cO$. Furthermore, we treat both the so-called finite- and infinite-fuel problem, meaning that the total amount of singular control that can be exerted is either capped or uncapped (see Karatzas \cite{karatzas1981monotone,karatzas1985probabilistic}).

The DPP is easy to state, and rather intuitive, and it quite literally forms the backbone of the whole stochastic (and deterministic) control theory. That is why it is important to develop a full understanding of all the mechanisms underpinning DPP and its rigorous mathematical proof across different classes of stochastic control problems. Moreover, DPP goes hand in hand with the theory of viscosity solutions for partial differential equations (see Crandall, Ishii and Lions \cite{crandall1992user}). In that respect, the DPP is the first step to prove that the value function of a stochastic control problem is a viscosity solution to a suitable (problem-specific) Hamilton Jacobi Bellman  (HJB) equation. The link to viscosity theory opens the door to the study of stochastic control problems via PDE methods that are often more versatile than the classical techniques based on Sobolev-type solutions (strong and weak). Notably, Bayraktar and Sirbu (see, e.g., \cite{bayraktar2012stochastic,bayraktar2013stochastic,sirbu2014stochastic,sirbu2015asymptotic}) developed an alternative approach to showing that the value functions of certain stochastic control problems/games are viscosity solutions of the corresponding HJB equations. Their approach does not require the DPP, which is instead obtained as a by-product. It does not seem to us that a direct application of those results is immediate in our set-up.  

Over the course of the past three decades mathematicians have increasingly come to realise that there are numerous subtleties hidden in a rigorous mathematical proof of the DPP. Perhaps the best known of such subtleties concerns the so-called `measurable selection' (see, e.g., Soner and Touzi \cite{soner2002dynamic} and references therein) which is needed in order to concatenate $\eps$-optimal controls starting from random initial conditions. That becomes problematic when the value function is only known to be measurable but it is not an issue when, for example, the value function is known to be continuous. Work by Bouchard and Touzi \cite{bouchard2011weak} and Bouchard and Nutz \cite{bouchard2012weak} develop a notion of {\em weak} DPP that overcomes the measurable selection problem without even requiring continuity of the value function (an extension of those ideas to the case of non-linear expectations is provided by Dumitrescu et al.\ in \cite{dumitrescu2016weak}). It is also worth mentioning an approach based on optimisation over families of probability measures, associated to controlled dynamics, on the space of c\`adl\`ag paths as in, e.g., El Karoui and Tan \cite{karoui2013capacities,el2013capacities} or Zitkovic \cite{zitkovic2014dynamic}. Further delicate technical problems arise from the use of {\em regular conditional probabilities} and the role played by null sets when changing the so-called reference probability system while following the trajectories of the controlled dynamics. Those difficulties are indicated in the monograph by Fabbri, Gozzi and Swiech \cite{fabbri2017stochastic} and in the work by Claisse, Talay and Tan \cite{claisse2016pseudo}, which we take as the main building blocks for our study. 

For an overview of classical results on the DPP we refer to traditional monographs on stochastic control (e.g., Krylov \cite{krylov2008controlled}, Fleming and Soner \cite{fleming2006controlled}, Yong and Zhou \cite{yong1999stochastic}) and to the references therein. Due to the singular control feature, our work is closely related to work by Haussmann and Suo \cite{haussmann1995singularExistence,haussmann1995singularDPP} who originally developed the DPP for singular control problems and their connection with viscosity solutions of suitable HJB equations. Ma and Yong \cite{ma1993regular, ma1999dynamic} extended the results by Haussmann and Suo to a more general set-up and gave sufficient conditions under which the value function of the problem is the {\em unique} continuous viscosity solution of a suitable HJB equation. In a one-dimensional setting Chiarolla \cite{chiarolla1997singular} obtained analogous results when the diffusion coefficient of the singularly controlled dynamics is not Lipschitz, which poses additional technical difficulties in proving continuity of the value function. Atar and Budhiraja \cite{atar2006singular} study DPP for state-constrained singular stochastic control problems and obtain that the value is the {\em unique} viscosity solution of the corresponding HJB equation. In particular, their controlled state process is a Brownian motion constrained to evolve inside a cone.

Our proof of the DPP encompasses a framework that is more general than in the papers from the paragraph above (except that we do not have a constraint on the controlled dynamics as in \cite{atar2006singular}). We include the discretionary stopping feature, the exit time from the domain $\cO$ and we allow the cost of exerting singular controls to be state-dependent. We also avoid making specific assumptions on the problem data, which other papers normally introduce as sufficient conditions for growth estimates and continuity of the value function. Instead we shift the focus to mild regularity of the objective function $J$ and pathwise uniqueness of the controlled dynamics (Assumptions \ref{ass:Xind} and \ref{ass:obj}). Of course our assumptions are implied by the more specific ones made, e.g., in \cite{ma1999dynamic}. Continuity of the value function is also not required but it is replaced by a weaker condition on the convergence of the expected values of suitable stochastic processes (Assumption \ref{ass:DCT}). Again, that requirement is satisfied if, for example, the value function can be shown to be continuous. Finally, we do not impose that the gain functions and costs appearing in the objective function $J$ have a sign.

Works in \cite{bouchard2012weak, bouchard2011weak,karoui2013capacities,zitkovic2014dynamic} provide the DPP under great generality. It seems reasonable to expect that a suitable adaptation and combination of the results and techniques in those papers would allow to devise a DPP in our setup. However, the task is highly non-trivial. The generality attained in \cite{bouchard2012weak, bouchard2011weak,karoui2013capacities,zitkovic2014dynamic} relies in part on abstract overarching assumptions, including concatenability of controls and stability under conditioning (in the language of \cite{karoui2013capacities}) that need to be verified on a case-by-case basis. Our work is self-contained and complements those results: we present a constructive approach based on probabilistic concepts and tools from the general theory of stochastic processes, under assumptions for which we also provide simple sufficient conditions with wide applicability. The overall philosophy in our paper is certainly inspired by \cite{fabbri2017stochastic} but the actual derivation of key technical results requires a different line of argument, due to the structural differences between our set-up and that in \cite{fabbri2017stochastic}.

The class of problems we consider encompasses modelling features that have already attracted sustained attention from the scientific community. For applications in singular control, it is interesting to consider exit times from a domain $\cO$ as they do appear, for example, in the famous optimal dividend problem (see, e.g., Jeanblanc and Shiryaev \cite{jeanblanc1995optimization}); in that context, $\cO$ is the solvency region. Discretionary stopping is also a desirable feature in several models addressed in the literature. In the case with only classical controls and without exit time from a domain we find work by Karatzas and Zamfirescu \cite{karatzas2006martingale} characterising martingale properties of the value process. In a similar setting but adopting relaxed controls (a notion close to that of randomised controls) we find work by Bassan and Ceci \cite{ceci2004mixed}, who prove that the value function is a viscosity solution of a suitable HJB equation. Explicit solutions in some particular problems of singular control with discretionary stopping over an infinite-time horizon are obtained by Davis and Zervos in \cite{davis1994problem} for one dimensional controlled dynamics. In a similar set-up, Morimoto \cite{morimoto2010singular} adds also an exit time of the controlled dynamics (a controlled geometric Brownian motion) from an interval of the real line. Using variational methods and penalisation techniques \cite{morimoto2010singular} proves that the value function solves a suitable HJB equation. A finite-fuel singular control problem with discretionary stopping and infinite-time horizon is solved by Karatzas et al.\ in \cite{karatzas2000finite} in closed form using free boundary methods applied to a parametric family of ordinary differential equations. Chen and Yi \cite{chen2012problem} use (parabolic) PDE methods and free boundary theory to solve a finite-time horizon problem of singular control with discretionary stopping with one dimensional controlled dynamics. Our contribution to this stream of the literature is to provide a rigorous derivation of the DPP, which was missing so far, in sufficient generality to cover all the models mentioned above and more.

The paper is organised as follows. In Section \ref{sec:notation} we collate some notation that will be used throughout the paper. In Section \ref{sec:problem} we set-up the problem and state standing assumptions. The main results of the paper are presented in Section \ref{sec:results} but their proofs are given in subsequent sections. In particular, in Section \ref{sec:ind} we prove independence of the value function of our problem from the reference probability system adopted; this leads to the equivalence of weak and strong formulation of the problem and to several useful equivalences in law of the controlled dynamics under different reference probability systems. In Section \ref{Sect:DPP} we use the technical results from Section \ref{sec:ind} to finally prove the DPP and the other results stated in Section \ref{sec:results}. The paper is completed by a technical appendix gathering useful results (largely known) on regular conditional probabilities.

\section{Notation and terminology}\label{sec:notation}

In this section we summarise notations and terminology adopted throughout the paper. While this is a useful compendium of symbols, it can be skipped at a first read as all concepts will be introduced in the paper when they first appear.
\begin{itemize}
\item[\bf(a)] $T>0$ is the time horizon, $d, d', l\in\bN$ denote dimensions of various state processes.\medskip
\item[\bf(b)] For a vector $x\in\bR^d$, we denote by $|x|$ its Euclidean norm. For $x,y\in\bR^d$, we denote by $\langle x,y\rangle=\sum_{i=1}^d x^iy^i$ the inner product. Given a set $A\in\bR^d$, we denote $A^c=\bR^d\setminus A$. For $z\in\bR$, we denote $(z)^\pm=\max\{0,\pm z\}$.
\medskip
\item[\bf(c)] For a Polish space $\Gamma$, we denote by $\cB(\Gamma)$ the Borel $\sigma$-algebra on $\Gamma$.\medskip
\item[\bf(d)] For a random variable $X:(\Omega,\cF)\to(\Gamma,\cS)$, defined on a probability space $(\Omega,\cF,\bP)$ and with values in a measurable space $(\Gamma,\cS)$, we denote by $\cL_\bP(X)$ its law. Notice that, in particular, $X$ could be a stochastic process (see, e.g., \cite[Page 24]{kallenberg2006foundations}).\medskip
\item[\bf(e)] For a bounded variation function $f:[0,T]\to\bR^d$ with $f=(f^1,\ldots,f^d)$, we denote by $f^\pm=(f^{1,\pm},\ldots,f^{d,\pm})$ the two components of its Jordan decomposition. Namely, for every $i=1,\ldots,d$, we have
\begin{align}
f^i(s)&=f^{i,+}(s)-f^{i,-}(s), \quad\text{with $f^{i,\pm}$ non-decreasing.}\label{JordanDecomp} 
\end{align}		
For each $i$, we denote by $V_{[t,s]}(f^i)$ the variation of $f^i$ on $[t,s]$:		 
\begin{align}
V_{[t,s]}(f^i)&=f^{i,+}(s)-f^{i,+}(t)+f^{i,-}(s)-f^{i,-}(t), \label{Variation} \qquad s\in[t,T].
\end{align}
Then, the variation of $f$ reads $V_{[t,s]}(f):=\sum_{i=1}^d V_{[t,s]}(f^i)$. \medskip
\item[\bf(f)] $M_{d,d'}$ is the collection of the $d\times d'$-dimensional real matrices.\medskip
\item[\bf(g)] $C([0,T];\bR^{d})$ is the collection of continuous functions $\varphi:[0,T]\mapsto\bR^d$.\medskip
\item[\bf(h)] For $t\in[0,T]$, a reference probability system starting at time $t$ is a 5-tuple 
\begin{equation*}
\nu:=(\Omega, \cF,\bP,\{\cF^t_s\}_{s\in[t,T]},W),
\end{equation*}
where:
\begin{itemize} 
\item[(i)] $(\Omega,\cF,\bP)$ is a complete probability space; 
\item[(ii)] $W=(W_s)_{s\in[t,T]}$ is a $d'$-dimensional Brownian motion on $(\Omega,\cF,\bP)$ starting at time $t$, i.e., $\bP(W_t=0)=1$; 
\item[(iii)] $\cF^{t,0}_s:=\sigma(W_u: u\in[t,s])$ and $\cF_s^t$ is the augmentation of $\cF^{t,0}_s$ with the $\bP$-null sets. 
\end{itemize}
The class $\cV_t$ contains all reference probability systems starting at time $t\in[0,T]$.
\medskip
\item[\bf(i)] We say that a reference probability system $\nu\in\cV_t$ is standard if there exists a $\sigma$-algebra $\cF^0$ such that $\cF^{t,0}_T\subseteq\cF^0\subseteq\cF$, where $\cF$ is the completion of $\cF^0$ with the $\bP$-null sets and $(\Omega,\cF^0)$ is a standard measurable space. Recall that a measurable space is standard if it is Borel isomorphic to, e.g., $(\bN,\cB(\bN))$ with $\cB(\bN)$ the Borel $\sigma$-algebra for the discrete topology.\medskip
\item[\bf(j)] For $t\in[0,T]$, the {\em canonical reference probability system} (starting at time $t$) is the 5-tuple 
\begin{equation*}
\nu^*:=(\Omega^*,\cF^*,\bP^*,\{\cB^t_s\}_{s\in[t,T]},W^*),
\end{equation*}
where:
\begin{itemize} 
\item[(i)] $\Omega^*:=\{\omega\in C([t,T];\bR^{d'}): \omega(t)=0\}$; 
\item[(ii)] $\bP^*$ is the Wiener measure on $(\Omega^*,\cB(\Omega^*))$ that makes the canonical process $(s,\omega)\mapsto W^*_s(\omega)=\omega(s)$ a Brownian motion starting at time $t$; 
\item[(iii)] $\cF^*$ is the completion of $\cB(\Omega^*)$ with the $\bP^*$-null sets; 
\item[(iv)] $\cB^{t,0}_s:=\sigma(W^*_u: u\in [t,s])$ and $\cB^t_s$ is the augmentation of $\cB^{t,0}_s$ with the $\bP^*$-null sets.
\end{itemize}\medskip
\item[\bf(k)] For a fixed $\nu\in\cV_t$, we denote by $\cA^\nu_t$ the collection of $\{\cF^t_s \}$-progressively measurable processes $\alpha=(\alpha_s)_{s\in[t,T]}$, taking values in a (possibly compact) subset $\cK\subseteq\bR^l$ ($l\in\bN$ as in (a)).\medskip
\item[\bf(l)] For a fixed $\nu\in\cV_t$ and a given $u\in[t,T]$, we denote by $\cX^\nu_u$ the collection of processes $\xi=(\xi_s)_{s\in[t,T]}$ such that
\begin{enumerate}[(i)]
\item $\xi$ is $\bR^d$-valued and $\{\cF^t_s\}$-adapted.
\item $\xi$ is left-continuous and of bounded variation $\bP$-a.s.
\item $\xi_{s}=0$ for every $s\in[t,u]$, $\bP$-a.s.
\item $\bE\big[|V_{[u,T]}(\xi)|^{p}\big]<\infty$, for some fixed $p>0$ (depending on the problem).
\end{enumerate}\medskip
Similarly, for a given random variable $\zeta\in[0,\infty)$, $\bP$-a.s., we denote by $\cX^\nu_u(\zeta)$ the class of finite-fuel controls, i.e., those for which condition (iv) above is replaced by \medskip
\begin{itemize}
	\item[(iv')] $\bP\big(V_{[u,T]}(\xi)\leq \zeta\big)=1$.
\end{itemize}\medskip
For $\xi$ belonging to either $\cX^{\nu}_u$ or $\cX^{\nu}_u(\zeta)$ its Jordan decomposition reads $\xi=\xi^+-\xi^-$, $\bP$-a.s.\medskip 
\item[\bf(m)] Fix $u\in[t,T]$ and $0\le z\le \bar z<\infty$. Given a control $\xi\in\cX^\nu_t(\bar z-z)$ and a $\cF^t_u$-measurable random variable $Z\ge z$, $\bP$-a.s., we set
\[
\sigma_Z:=\inf\{s\geq u: V_{[u,s]}(\xi)\geq \bar z- Z\}\wedge T
\] 
and define the truncation of $\xi^\pm$ at $Z$ (after time $u$) by $(\xi^\pm_{s\wedge\sigma_Z})_{s\in[t,T]}$. The increments after time $u$ of the truncated process are denoted by
\begin{equation*}
[\xi^\pm]^{u,Z}_s:=\xi^\pm_{s\wedge\sigma_Z}-\xi^\pm_u,\quad s\in[u,T].
\end{equation*}
\medskip
\item[\bf(n)] For a fixed $\nu\in\cV_t$ and a given $u\in[t,T]$, we denote by $\cT^\nu_u$ the collection of $\{\cF^t_s\}$-stopping times such that $u\leq\tau\leq T$, $\bP$-a.s.\medskip
\item[\bf(o)] For a fixed $\nu\in\cV_t$ and a given $u\in[t,T]$ and $z\in[ 0,\infty)$, we denote 
\[
	\AA^\nu_t=\cT^\nu_t\times\cA^\nu_t\times\cX^\nu_t\quad\text{and}\quad \AA^\nu_{t,z}=\cT^\nu_t\times\cA^\nu_t\times\cX^\nu_t(z).
\]
\item[\bf(p)] For a fixed $\nu\in\cV_t$ and a given $u\in[t,T]$, we denote by $\cE^\nu_u$ the collection of processes $\eta=(\eta_s)_{s\in[t,T]}$ such that
\begin{enumerate}[(i)]
\item $\eta$ is $\{\cF^t_s\}$-adapted.
\item $\eta$ is left-continuous and non-decreasing $\bP$-a.s.
\item $\eta_s\in\{0,1\}$ for every $s\in[t,T]$ and $\eta_s=0$ for every $s\in[t,u]$, $\bP$-a.s.
\end{enumerate}\medskip
\item[\bf (q)] For a fixed $\nu\in\cV_t$ we denote by $\cP_\Omega$ the $\sigma$-algebra of $\{\cF^{t,0}_s\}$-predictable sets. Recall that $\cP_\Omega$ is a $\sigma$-algebra on $[t,T]\times\Omega$ and it is generated by the sets of the form 
\[
(s,u]\times A\quad\text{with $t\leq s< u\leq T$ and $A\in\cF^{t,0}_{s-}$}
\] 
and of the form
\[
\{t\}\times A, \quad\text{with $A\in\cF^{t,0}_{t}$}.
\] 
Notice that by left-continuity of the raw Brownian filtration, we have $\cF^{t,0}_{s-}=\cF^{t,0}_{s}$ for every $s\in(t,T]$ (and by convention $\cF^{t,0}_{t-}:=\cF^{t,0}_t$). A process $Y:[t,T]\times\Omega\to\bR^d$ that is $\cP_\Omega/\cB(\bR^d)$-measurable is called $\{\cF^{t,0}_s\}$-predictable.
\end{itemize}

\section{Problem formulation}\label{sec:problem}
Let $p>0$, $d,d',l\in\bN$. Fix $t\in[0,T]$ and let $\cV_t$ be the collection of reference probability systems starting at time $t$. That is, $\nu\in\cV_t$ means:
\begin{equation*}
\nu=(\Omega, \cF,\bP,\{\cF^t_s\}_{s\in[t,T]},W),
\end{equation*}
where $(\Omega,\cF,\bP)$ is a complete probability space equipped with a $d'$-dimensional Brownian motion $W=(W_s)_{s\in[t,T]}$ starting at time $t$, i.e., $\bP(W_t=0)=1$, and $\cF_s^t$ is the augmentation of $\cF^{t,0}_s:=\sigma(W_u: u\in[t,s])$ with the $\bP$-null sets. 
With no loss of generality, we assume that $t\mapsto W_t(\omega)$ is continuous for all $\omega\in\Omega$. For completeness we should use the notation $W^t$ for the Brownian motion, in order to keep track of the starting condition $W^t_t=0$. We omit this notation here as the time $t$ will be fixed (but arbitrary) throughout the paper.

An admissible control-stopping treble $(\tau,\alpha,\xi)\in\cT^\nu_t\times\cA^\nu_t\times\cX^\nu_t$ in the reference probability system $\nu$ includes a stopping time $\tau$ and a pair of processes $(\alpha,\xi)$ such that:
\begin{itemize}
\item[(i)] $\tau$ is a $\{\cF^t_s \}$-stopping time such that $t\leq \tau\leq T$, $\bP$-a.s.
\item[(ii)] $\alpha=(\alpha_s)_{s\in[t,T]}$ is $\{\cF^t_s \}$-progressively measurable and taking values in a (possibly compact) subset $\cK\subseteq\bR^l$;
\item[(iii)] $\xi=(\xi_s)_{s\in[t,T]}$ is $\bR^d$-valued, $\{\cF^t_s\}$-adapted, left-continuous and of bounded variation $\bP$-a.s.~with $\xi_t=0$, $\bP$-a.s.~and such that $\bE\big[|V_{[t,T]}(\xi)|^p\big]<\infty$.
\end{itemize}
Here $V_{[t,T]}(\xi)$ is the (random) total variation of the process $\xi$ over the time interval $[t,T]$ defined as the sum of the variations of each coordinate $\xi^i$, $i=1,\ldots d$; that is,
\begin{align}\label{eq:var}
V_{[t,s]}(\xi):=\sum_{i=1}^d\big(\xi^{i,+}_s+\xi^{i,-}_s\big)=\sum_{i=1}^d \sup_{t\le r\le s}\big(\xi^{i,+}_r+\xi^{i,-}_r\big),
\end{align}
where $\xi^i=\xi^{i,+}-\xi^{i,-}$ is the Jordan decomposition of the $i$-th entry of the vector $\xi$ and we use monotonicity of $\xi^{i,\pm}$.

Alternatively, we can also consider control-stopping trebles $(\tau,\alpha,\xi)\in\cT^\nu_t\times\cA^\nu_t\times\cX^\nu_t(\zeta)$ where for a given random variable $\zeta\in[ 0,\infty)$, $\bP$-a.s., we denote by $\cX^\nu_t(\zeta)$ the class of finite-fuel controls, i.e., those satisfying (iii) above but when the condition $\bE\big[|V_{[t,T]}(\xi)|^p\big]<\infty$ is replaced by the stricter condition $\bP\big(V_{[t,T]}(\xi)\leq \zeta\big)=1$. Notice that, formally, $\cX^\nu_t\subset \cX^\nu_t(\infty)$. 

The results in this paper hold for both types of control pairs (finite and infinite fuel) and we will explicitly refer to small differences in the arguments of proof as needed. For future reference, given $u\in[t,T]$, we also introduce the subsets $\cX^\nu_u\subset\cX^\nu_t$ and $\cX^\nu_u(z)\subset\cX^\nu_t(z)$ of processes from (iii) above but  
such that $\bP(\text{$\xi_s=0$ for $s\in[t,u]$})=1$.

Given an admissible pair $(\alpha,\xi)$ the controlled state process for our problem follows the stochastic differential equation (SDE)
\begin{equation}\label{StateProcess}
X_s=x+\int_t^s\mu(r,X_r,\alpha_r)\ud r+\int_t^s \sigma(r,X_r,\alpha_r)\ud W_r+\xi_s,\:\: s\in[t,T], \quad x\in\bR^d,
\end{equation}
where $\mu:[0,T]\times\bR^d\times\cK\to\bR^d$, $\sigma:[0,T]\times\bR^d\times\cK\to M_{d,d'}$. We will assume that the SDE \eqref{StateProcess} admits a unique, caglad, $\{\cF^t_s\}$-adapted solution for any admissible pair $(\alpha,\xi)$, up to indistiguishibility (see Assumption \ref{ass:Xind}).

Sometimes it is convenient to denote the solution of \eqref{StateProcess} by $X^{t,x;\nu;\alpha,\xi}$ to highlight the dependence on the reference probability system $\nu$, the admissible controls $(\alpha,\xi)$ and the initial condition $(t,x)$. However, when no confusion shall arise we may also use the notations $X^{t,x;\alpha,\xi}$, $X^{t,x}$, $X^{\alpha,\xi}$ or simply $X$ depending on the circumstances. Similarly, we denote by $\rho^{t,x;\nu;\alpha,\xi}_\cO$ the first exit time of the process $X^{t,x;\nu;\alpha,\xi}$ from a domain $\cO\in\cB(\bR^d)$, i.e.,
\begin{equation}\label{eq:rho}
\rho^{t,x;\nu;\alpha,\xi}_\cO:=\inf\{s\geq t: X^{t,x;\nu;\alpha,\xi}_s\notin \cO \}\wedge T.
\end{equation}
When no confusion shall arise we may also use the simpler notations $\rho^{t,x;\alpha,\xi}_\cO$, $\rho^{\alpha,\xi}_\cO$ and $\rho_\cO$.
	
Given an initial condition $(t,x)\in[0,T]\times\bR^d$, a reference probability system $\nu\in\cV_t$ and an admissible treble $(\tau,\alpha,\xi)\in\cT^\nu_t\times\cA^\nu_t\times\cX^\nu_t$, the objective function in our problem reads 
\begin{align}\label{eq:J}
J^\nu_{t,x}(\tau,\alpha,\xi):=\bE\bigg[&\int_t^{\tau\wedge\rho_\cO} f(s,X_s,\alpha_s)\ud s-\int_{[t,\tau\wedge\rho_\cO)}\langle c_+(s,X_s),\ud\xi^+_s\rangle\nonumber \\
	&-\int_{[t,\tau\wedge\rho_\cO)}\langle c_-(s,X_s),\ud\xi^-_s\rangle +g_1(\rho_\cO,X_{\rho_\cO})\ind_{\{\rho_\cO\leq\tau \}}+g_2(\tau,X_\tau)\ind_{\{\rho_\cO>\tau \}} \bigg],
\end{align}
where $f:[0,T]\times\bR^d\times\cK\to\bR$ is a running gain, $g_1:[0,T]\times\bR^d\to\bR$ the terminal gain when the controlled process leaves the set $\cO$ prior to $\tau$, $g_2:[0,T]\times\bR^d\to\bR$ the terminal gain when $\tau$ occurs prior to $\rho_\cO$ and $c_\pm:[0,T]\times\bR^d\to\bR^d$ are the vectors of cost per unit of singular control exerted. The integrals with respect to the controls $\xi^\pm$ are Lebesgue-Stieltjes integrals and are understood as
\begin{equation*}
\int_{[t,\tau)}\langle c_\pm(s,X_s),\ud\xi^\pm_s\rangle:=\sum_{i=1}^d\int_{[t,\tau)} c^{i}_\pm(s,X_s)\ud\xi^{i,\pm}_s.
\end{equation*}
We allow for different costs $c_+$ and $c_-$ associated to the two increasing processes $\xi^+$ and $\xi^-$, respectively. For the sake of generality, $c_\pm$ take values in $\bR^d$ so that negative costs may be allowed too. 
\begin{remark}
It is worth emphasising that while in the case of state-independent costs, $c_\pm(t)$, there is a unique interpretation of the integrals above, in the state-dependent case, $c_{\pm}(t,x)$, the classical Lebesgue-Stieltjes integral is known to cause some technical problems with the use of Hamilton-Jacobi-Bellman (HJB) equations. This fact is well-illustrated in \cite{alvarez2000singular} where the absence of an admissible optimal control is demonstrated in a one-dimensional problem. This misalignment can be resolved, at least in the case when the cost of control is the same in all directions, i.e., $c^i_\pm=c^j_\pm=c$ for all $i, j$, by taking a different type of integral. Namely, for the singular control one should use the representation $\xi_t=\int_{[0,t)}n_s \ud |\xi|_s$, where $n_s\in\bR^d$ is a unitary vector, $(s,\omega)\mapsto n_s(\omega)$ is progressively measurable and $|\xi|_s(\omega)$ denotes the total variation of $\xi(\omega)$ on $[0,s]$. Then, the cost per unit of control exerted is defined as
\begin{equation}\label{eq:zhu}
\int_{[t,\tau)}c(s,X_s)\circ \ud\xi_s:=\int_{[t,\tau)} c(s,X_s)\ud|\xi^c|_s+\sum_{0\le s< \tau}\int^{|\xi|_{s+}}_{|\xi|_{s}}c(s,X_s+\lambda n_s)\ud \lambda,
\end{equation}
where $\xi^c$ is the continuous part of $\xi$. With this formulation, it is normally possible to connect the singular stochastic control problem with a HJB equation with gradient constraint (see, e.g., \cite{zhu1992generalized}). From the point of view of our analysis, the specific choice of the integral is irrelevant, as long as it is a measurable function of the paths of the controlled state process. So we avoid delving further into this matter as our results continue to hold under, for example, the specification in \eqref{eq:zhu}.
\end{remark}

Naturally, we assume that the functions $(s,y,a)\mapsto f(s,y,a)$, $(s,y)\mapsto c_\pm(s,y)$, $(s,y)\mapsto g_1(s,y)$ and $(s,y)\mapsto g_2(s,y)$ are Borel-measurable. 
	
The controller-stopper aims at maximising the objective function $J^\nu_{t,x}(\tau,\alpha,\xi)$ over all admissible trebles $(\tau,\alpha,\xi)$. To simplify the notation, we set
\[
\mathrm{Adm}^\nu_t=\cT^\nu_t\times\cA^\nu_t\times\cX^\nu_t.
\]
A priori there are two formulations of the problem: \medskip

{\em Strong formulation}. For fixed $(t,x)\in[0,T]\times\bR^d$ and $\nu\in\cV_t$, we define
\begin{equation}\label{StrongFormulation}
w^\nu(t,x):=\sup_{(\tau,\alpha,\xi)\in\AA^\nu_t}J^\nu_{t,x}(\tau,\alpha,\xi).
\end{equation}

{\em Weak formulation}. For fixed $(t,x)\in[0,T]\times\bR^d$, we define
\begin{equation}\label{WeakFormulation}
v(t,x):=\sup_{\nu\in\cV_t} \: \sup_{(\tau,\alpha,\xi)\in\AA^\nu_t}J^\nu_{t,x}(\tau,\alpha,\xi)=\sup_{\nu\in\cV_t}w^\nu(t,x). 
\end{equation}

In the case of {\em finite fuel}, we must fix the total fuel $\bar z\in[0,\infty)$ and add one state variable to the problem that accounts for the remaining fuel at each moment in time. That means that we will consider the state process $(s,X^{\alpha,\xi}_s,Z^\xi_s)_{s\in[t,T]}$ with the additional dynamics
\begin{equation}\label{eq:Z}
Z^\xi_s=Z^{t,z;\xi}_s:=z+V_{[t,s]}(\xi), \quad s\in[t,T], \: \xi\in\cX^\nu_t(\bar{z}-z),
\end{equation}
where $z\in[0,\bar z]$. Analogously to $X$, we may use the simpler notations $Z^{t,z;\xi}$ and $Z$ when no confusion would arise.
The objective function may be taken as in \eqref{eq:J} or it may also depend on the total amount of fuel exerted. In the latter case, we denote it by 
\begin{align}\label{eq:JZ}
J^\nu_{t,x,z}(\tau,\alpha,\xi):=\bE\bigg[&\int_t^{\tau\wedge\rho_\cO} f(s,X_s,Z_s,\alpha_s)\ud s-\int_{[t,\tau\wedge\rho_\cO)}\langle c_+(s,X_s),\ud\xi^+_s\rangle\nonumber \\
&-\int_{[t,\tau\wedge\rho_\cO)}\langle c_-(s,X_s),\ud\xi^-_s\rangle +g_1(\rho_\cO,X_{\rho_\cO},Z_{\rho_\cO})\ind_{\{\rho_\cO\leq\tau \}}+g_2(\tau,X_\tau,Z_{\tau})\ind_{\{\rho_\cO>\tau \}} \bigg],
\end{align}
where now $\rho_\cO=\inf\{s\ge t: (X^{\alpha,\xi}_s,Z^\xi_s)\notin\cO\}\wedge T$ for some $\cO\subset \bR^d\times[0,\bar z]$ and
with obvious changes to the domains of the functions $f$, $g_1$ and $g_2$. It will be completely clear from our analysis below that our results hold if we take an even more general form of the cost per unit control exerted, i.e., $c_{\pm}(s,X_s,Z_s)$. We refrain from adding that extension to avoid further notational complexity. Instead, to simplify the notation, we set
\[
\AA^\nu_{t,\bar z-z}=\cT^\nu_t\times\cA^\nu_t\times\cX^\nu_t(\bar z-z).
\]
The value functions in the two formulations with finite fuel read as follows:\medskip

{\em Strong formulation}. For fixed $(t,x,z)\in[0,T]\times\bR^d\times[0,\bar z]$ and $\nu\in\cV_t$, we define
	\begin{equation}\label{SFff}
	w^\nu(t,x,z):=\sup_{(\tau,\alpha,\xi)\in\AA^\nu_{t,\bar z-z}}J^\nu_{t,x,z}(\tau,\alpha,\xi).
	\end{equation}

{\em Weak formulation}. For fixed $(t,x,z)\in[0,T]\times\bR^d\times[0,\bar z]$, we define
\begin{equation}\label{WFff}
v(t,x,z):=\sup_{\nu\in\cV_t} \: \sup_{(\tau,\alpha,\xi)\in\AA^\nu_{t,\bar z-z}}J^\nu_{t,x,z}(\tau,\alpha,\xi)=\sup_{\nu\in\cV_t}w^\nu(t,x,z). 
\end{equation}
Notice that $\xi\in \cX^\nu_t(\bar z-z)$ implies $\bP(V_{[t,T]}(\xi)\le \bar z-z)=\bP(Z^{t,z;\xi}_T\le \bar z)=1$.\medskip

As it will be shown, weak and strong formulations (both for finite and infinite fuel) are in fact identical. The dual approach is, however, essential in order to prove the DPP (see Remark \ref{Rmk:PseudoMarkov}).
\begin{remark}
We adopt the same terminology as in \cite[Section 2.1]{fabbri2017stochastic}. The term ``{\em weak}'' in \eqref{WeakFormulation} and \eqref{WFff} refers to the fact that the reference probability system can vary together with the controls whereas in the ``{\em strong}'' formulation $\nu$ is fixed (see \eqref{StrongFormulation} and \eqref{SFff}). 
\end{remark}
	
\begin{remark}
Our setup is well suited to cover also the following situations:
\begin{enumerate}[(a)]
\item Controls that do not operate in all directions, i.e., $\xi^i\equiv 0$ for some $i=1,2,\ldots,d$.
\item Monotone controls, i.e., either $\xi^+\equiv 0$ or $\xi^-\equiv 0$.
\end{enumerate}
We can accommodate fully degenerate controlled diffusions. In particular, this allows to consider a state-dependent discount factor in the definition of the objective function $J_{t,x}^\nu$ and $J_{t,x,z}^\nu$ by taking, e.g., 
\[
X^d_s=\int_t^s r(u,X^1_u,\ldots, X^{d-1}_u)\ud u
\]
and
\[
f(t,x,\alpha)=\e^{-x^d}\bar f(t,x^1,\ldots , x^{d-1},\alpha),\:\:g_i(t,x,\alpha)=\e^{-x^d}\bar g_i(t,x^1,\ldots , x^{d-1},\alpha),
\]
for some functions $\bar f$, $\bar g_i$, $i=1,2$ (and analogously for the costs $c_\pm$).
\end{remark}
	
Throughout the paper we make a number of standing assumptions that simplify the exposition. Such assumptions concern mainly the controlled dynamics and the objective function and can be checked on a case-by-case basis in practical applications. In particular, ideas contained in \cite{ma1993regular} for singular control problems can be easily adapted to our setting and, more specifically, \cite[Ch.\ III, Sec.\ 9]{milazzo2021} contains mild sufficient conditions for problems of singular control with discretionary stopping. Assumptions \ref{ass:Xind}--\ref{ass:value} below hold throughout the paper, whereas Assumption \ref{ass:DCT} is only needed in Theorems \ref{thm:DPPst} and \ref{thm:DPPsttff}.
 
\begin{assumptions}[{\bf Indistinguishability}]\label{ass:Xind}
For every admissible pair $(\alpha,\xi)$, the SDE \eqref{StateProcess} admits a unique $\{\cF^t_s \}$-adapted solution, up to indistiguishibility.
\end{assumptions}

When $V_[t,T](\xi)$ is $p$-integrable with $p\geq 2$ and $\mu$ and $\sigma$ are Lipschitz-continuou, Assumption \ref{ass:Xind} holds by standard SDE techniques. For a more general case, when no integrability on $\xi$ is assumed, one can use results from \cite{doleans1976existence}. Assumption \ref{ass:Xind} also yields the following simple lemma.

\begin{lemma}\label{lemma:Xind}
	Fix $(t,x,z)\in[0,T]\times\bR^d\times[0,\bar z]$ and $\nu\in\cV_t$. Let $\xi^1,\xi^2\in\cX^\nu_t$ (or $\xi^1,\xi^2\in\cX^\nu_t(\bar z)$) and $\alpha^1,\alpha^2\in\cA^\nu_t$ be such that $\xi^1$ and $\xi^2$ are indistinguishable and $\alpha^1=\alpha^2$, $\ud s\times\bP$-a.e. Then,
	\begin{equation}
	\bP\big(X^{\alpha^1,\xi^1}_s=X^{\alpha^2,\xi^2}_s, \: \forall \: s\in[t,T]\big)=1.
	\end{equation}
\end{lemma}

\begin{proof}
	With probability one
	\begin{align*}
	X^{\alpha^1,\xi^1}_s&=x+\int_t^s\mu(r,X^{\alpha^1,\xi^1}_r,\alpha^1_r)\ud r+\int_t^s\sigma(r,X^{\alpha^1,\xi^1},\alpha^1_r)\ud W_r+\xi^1_s\\
	&=x+\int_t^s\mu(r,X^{\alpha^1,\xi^1}_r,\alpha^2_r)\ud r+\int_t^s\sigma(r,X^{\alpha^1,\xi^1},\alpha^2_r)\ud W_r+\xi^2_s, 
	\end{align*}
	for all $s\in[t,T]$, where the second equality holds because $\xi^1$ and $\xi^2$ are indistinguishable and $\alpha^1=\alpha^2$, $\ud s\times\bP$-a.e. Thus, $X^{\alpha^1,\xi^1}$ verifies the same SDE as $X^{\alpha^2,\xi^2}$ on $\nu$ and so, by Assumption \ref{ass:Xind}, $X^{\alpha^1,\xi^1}$ and $X^{\alpha^2,\xi^2}$ are indistinguishable.
\end{proof}

For concreteness, we assume some integrability of the reward/cost functions appearing in the objective of our optimisation. Given an admissible treble $(\tau,\alpha,\xi)$ on a reference probability system $\nu$, we introduce the process
\begin{align}\label{eq:N}
N^{\nu;\alpha,\xi}_{u\wedge\tau}:=&\int_t^{u\wedge\tau\wedge\rho_\cO}\!\! f(s,X_s,Z_s,\alpha_s)\ud s-\int_{[t,u\wedge\tau\wedge\rho_\cO)}\langle c_+(s,X_s),\ud\xi^+_s\rangle-\int_{[t,u\wedge\tau\wedge\rho_\cO)}\langle c_-(s,X_s),\ud\xi^-_s\rangle\nonumber \\
&+g_1(\rho_\cO,X_{\rho_\cO},Z_{\rho_\cO})\ind_{\{\rho_\cO\leq\tau \}\cap\{\rho_\cO<u\}}+g_2(\tau,X_\tau,Z_{\tau})\ind_{\{\tau<u\wedge\rho_\cO\}}, 
\end{align}
with $(X,Z)=(X^{t,x;\alpha,\xi},Z^{t,z;\xi})$. An analogous process is clearly defined for the infinite-fuel problem, by dropping the dependence on $Z$. We also denote $(x)^\pm:=\max\{0,\pm x\}$.

\begin{assumptions}[{\bf Objective function I}]\label{ass:obj0} 
There is a constant $\bar g>0$ such that for $i=1,2$ it holds $g_i(t,x,z)\ge -\bar g$ for all $(t,x,z)$. Moreover, for any $(t,x,z)$ and any admissible treble $(\tau,\alpha,\xi)$ in a reference probability system $\nu$, we have $\bE[(N^{\nu;\alpha,\xi}_{u\wedge\tau})^-]<\infty$ for any $u\in[t,T]$.
\end{assumptions}
Integrability of $(N_u)^-$ is a mild requirement. It is immediately satisfied when, for example, the total variation of $\xi$ is at least $p$-integrable with $p\geq 1$, $f$ is bounded from below and $ c_\pm$ are bounded from above. If the coefficients of the SDE \eqref{StateProcess} are Lipschitz, then when $p\geq 2$ it is enough to assume that $|f(t,x,z,\alpha)|+|g_1(t,x,z)|+|g_2(t,x,z)|\le c(1+|x|^2+|z|^2)$ for some $c>0$ and $c_\pm$ bounded from above.

We also impose some continuity on the objective function. This is easy to state for infinite-fuel problems but it requires an additional notion of truncated controls in the case of finite-fuel problems. Fix $u\in[t,T]$ and $0\le z\le \bar z<\infty$. Given a control $\xi\in\cX^\nu_t(\bar z-z)$ and a $\cF^t_u$-measurable random variable $Z\in [z,\bar{z}]$, $\bP$-a.s., we set
\[
\sigma_Z:=\inf\{s\geq u: V_{[u,s]}(\xi)\geq \bar z- Z\}\wedge T
\] 
and define the truncation of $\xi^\pm$ at $Z$ (after time $u$) by $(\xi^\pm_{s\wedge\sigma_Z})_{s\in[u,T]}$.
Increments of the truncated controls read $[\xi^\pm]^{u,Z}_s:=\xi^\pm_{s\wedge\sigma_Z}-\xi^\pm_u$ for $s\in[u,T]$ and setting $[\xi]^{u,Z}_s:=[\xi^+]^{u,Z}_s-[\xi^-]^{u,Z}_s$, we have $[\xi]^{u,Z}\in\cX^\nu_u(\bar z-Z)$ by construction. By convention, to simplify the notation, we drop the superscript $u$ in $[\xi]^{u,Z}$ and we write $[\xi]^Z$ if $\xi\in\cX^\nu_u(\tilde{z})$ for some $\tilde{z}\in(0,\infty)$.

\begin{assumptions}[{\bf Objective function II}]\label{ass:obj}
Let $u\in[0,T]$ be arbitrarily fixed.
\begin{itemize}
\item {\em (Infinite fuel)} The mapping $y\mapsto J^{\nu}_{u,y}(\tau,\alpha,\xi)$ is continuous on $\bR^d$ uniformly in $(\tau,\alpha,\xi)\in\AA^{\nu}_u$ and uniformly with respect to $\nu\in\cV_u$. 
\item 
{\em (Finite fuel)} For every $(x_1,z_1),(x_2,z_2)\in\bR^d\times[0,\bar z]$ with $z_2\geq z_1$ and every $\eps>0$ there exists $\delta>0$ such that, if $|(x_1,z_1)-(x_2,z_2)|<\delta$, then 
\begin{equation}\label{ContFinFuel}
|J^{\nu}_{u,x_1,z_1}(\tau,\alpha,\xi)-J^{\nu}_{u,x_2,z_2}(\tau,\alpha,[\xi]^{z_2})|<\eps,
\end{equation}
for every $(\tau,\alpha,\xi)\in\AA^{\nu}_{u, \bar z-z_1}$ and all $\nu\in\cV_u$.
\end{itemize}
\end{assumptions}

The continuity requirements for the objective function can be easily verified in the infinite-fuel case when the coefficients in the SDE \eqref{StateProcess} are Lipschitz continuous, the functions $f$, $g_1$ and $g_2$ are, for example, H\"older continuous and $c_\pm$ are functions of time only (see \cite[Proposition III.9.5]{milazzo2021}). In the finite-fuel case, the uniform bound on the total variation of the singular controls allows to prove these continuity requirements (by a similar argument) also if $c_\pm$ depend on the space variable and are, e.g., H\"older continuous. 
The next assumption is a minimal technical assumption on measurability and finiteness of the problem's value function in its weak formulation, that guarantees well-posedness of the optimisation problem.
\begin{assumptions}[{\bf Value function}]\label{ass:value} 
The function $(t,x)\mapsto v(t,x)$ is Borel-measurable with $v(t,x)<\infty$ for every $(t,x)\in[0,T]\times\bR^d$. Analogously, for the finite fuel set-up, $(t,x,z)\mapsto v(t,x,z)$ is Borel-measurable with $v(t,x,z)<\infty$ for every $(t,x,z)\in[0,T]\times\bR^d\times[0,\bar z]$. 
\end{assumptions}		

It is not difficult to check that in the infinite-fuel case Assumptions \ref{ass:obj} and \ref{ass:value} imply that for each $u\in[t,T]$, given any $\eps>0$ there exists $\delta>0$ such that
\begin{align}\label{eq:vc1}
|v(u,x)-v(u,y)|<\eps,\quad\text{for $|x-y|<\delta$.} 
\end{align}
Analogously for the finite-fuel case, for each $u\in[t,T]$ and any $\eps>0$, there exists $\delta>0$ such that 
\begin{align}\label{eq:vc2}
|v(u,x_1,z_1)-v(u,x_2,z_2)|<\eps,\quad\text{for $|(x_1,z_1)-(x_2,z_2)|<\delta$, with $z_2\ge z_1$.} 
\end{align}
Here, notice that the inequality $v(u,x_1,x_2)\leq v(u,x_2,z_2)+\eps$ follows easily from Assumption \ref{ass:obj} since $z_2\geq z_1$, whereas the opposite inequality is slightly more delicate and can be shown as follows. By Assumption \ref{ass:obj}, we have that if $|(x_1,z_1)-(x_2,z_2)|<\delta$ then
	$$J^\nu_{u,x_2,z_2}(\tau,\alpha,[\xi]^{z_2})\leq J^\nu_{u,x_1,z_1}(\tau,\alpha,\xi)+\eps\leq v(u,x_1,z_1)+\eps,$$
	for every $(\tau,\alpha,\xi)\in\AA^\nu_{u,\bar{z}-z_1}$. Since $z_1\leq z_2$, then $\AA^\nu_{u,\bar{z}-z_2}\subseteq\AA^\nu_{u,\bar{z}-z_1}$. Moreover, if $\xi\in\cX^\nu_{\bar{z}-z_2}$ then $[\xi]^{z_2}=\xi$ and thus we obtain
	$$J^\nu_{u,x_2,z_2}(\tau,\alpha,\xi)\leq v(u,x_1,z_1)+\eps,$$
	for every $(\tau,\alpha,\xi)\in\AA^\nu_{u,\bar{z}-z_2}$. Hence, by taking the supremum,
	$$v(u,x_2,z_2)\leq v(u,x_1,z_1)+\eps.$$

For our last assumption, given an admissible treble $(\tau,\alpha,\xi)$ let the process $M\!=\!(M^{\nu;\alpha,\xi}_{u\wedge \tau})_{u\in[t,T]}$ be defined as
\begin{align}\label{eq:M}
M_{u\wedge\tau}&:=\!\int_t^{u\wedge\tau\wedge\rho_\cO}\!\! f(s,X_s,Z_s,\alpha_s)\ud s\!-\!\int_{[t,u\wedge\tau\wedge\rho_\cO)}\!\!\langle c_+(s,X_s),\ud\xi^+_s\rangle\!-\! \int_{[t,u\wedge\tau\wedge\rho_\cO)}\!\langle c_-(s,X_s),\ud\xi^-_s\rangle\nonumber\\
&\hspace{17pt}+g_1(\rho_\cO,X_{\rho_\cO},Z_{\rho_\cO})\ind_{\{\rho_\cO\leq\tau\}\cap\{ \rho_\cO<u \}}+g_2(\tau,X_\tau,Z_\tau)\ind_{\{\tau< u\wedge\rho_\cO \}}+v(u,X_u,Z_u)\ind_{\{\tau\wedge\rho_\cO\geq u\}},
\end{align}
with $(X,Z)=(X^{\nu;\alpha,\xi},Z^{\nu;\xi})$. Notice that $M_u=N_u+v(u,X_u,Z_u)$ on the event $\{\tau\wedge\rho_\cO\geq u\}$.
An analogous definition holds in the infinite-fuel problem if we drop the dependence on $Z$.

\begin{assumptions}[{\bf Convergence of} $M$]\label{ass:DCT} 
	Let $(t,x,z)\in[0,T]\times\bR^d\times[0,\bar{z}]$, $\nu\in\cV_t$ be an arbitrary reference probability system and $(\tau,\alpha,\xi)\in\AA^\nu_{t,\bar{z}-z}$. Let $\sigma\in\cT_t^\nu$ be arbitrary and $\{\sigma_n \}_{n\in\bN}\subseteq \cT^\nu_t$ be a sequence of stopping times such that $\sigma_n \uparrow \sigma$ as $n\to\infty$, $\bP$-a.s.\ on $\{\sigma>t \}$. Then,
	\begin{equation}\label{DCTcondition}
	\lim_{n\to\infty}\bE\big[M^{\nu;\alpha,\xi}_{\sigma_n\wedge\tau}\big]=\bE\big[M^{\nu;\alpha,\xi}_{\sigma\wedge\tau}\big],
	\end{equation}
	and an analogous condition holds for the infinite-fuel case.
\end{assumptions}

Recalling that $s\mapsto(X_s,Z_s)$ is left-continuous $\bP$-a.s., the convergence in \eqref{DCTcondition} can be obtained if, e.g., $s\mapsto v(s,X_s,Z_s)$ is also left-continuous and if the dominated convergence theorem applies. Continuity of the value function and suitable growth estimates are therefore sufficient. 

Continuity of $v$ is generally satisfied when, for example, the coefficients in the SDE \eqref{StateProcess} are Lipschitz continuous, the functions $f$, $g_1$ and $g_2$ are H\"older continuous and the costs $c_\pm$ depend continuously on time only (see Corollary III.9.8 and Proposition III.9.10 in \cite{milazzo2021} for the infinite-fuel case when $p\geq 2$; the finite-fuel case is analogous. See also \cite[Theorem 3.3]{ma1993regular} for singular control problems with both finite- and infinite-fuel). 

As for the growth estimates, dominated convergence theorem can be applied in the finite-fuel case if, e.g., $f$, $g_1$, $g_2$ have linear growth and $c_\pm$ are bounded (thanks also to standard SDE estimates). In the infinite-fuel case, when the total variation of $\xi$ is $p$-integrable with $p\geq 1$, the functions $f$, $g_1$, $g_2$ are bounded from above and $c_\pm$ are bounded from below (as assumed in \cite{ma1993regular}) the resulting value function is bounded.

\section{Main results}\label{sec:results}

Under the standing Assumptions \ref{ass:Xind}--\ref{ass:value} we obtain the following versions of the dynamic programming principle (DPP), which are the main results in this paper. The proofs are distilled in Section \ref{Sect:DPP} and build upon a series of technical lemmas and propositions. 

The first two theorems state the DPP for deterministic times in both the infinite- and finite-fuel setting, respectively.
\begin{theorem}[{\bf Infinite fuel}]\label{thm:DPPdet}
Fix $(t,x)\in[0,T]\times\bR^d$ and $\nu\in\cV_t$. For any $u\in[t,T]$, we have
\begin{align*}
v(t,x)=\sup_{(\tau,\alpha,\xi)\in\AA^\nu_t} \bE\bigg[&\!\int_t^{u\wedge\tau\wedge\rho_\cO}\!\!\! f(s,X_s,\alpha_s)\ud s\!-\!\int_{[t,u\wedge\tau\wedge\rho_\cO)}\!\Big(\langle c_+(s,X_s),\ud\xi^+_s\rangle\! +\!\langle c_-(s,X_s),\ud\xi^-_s\rangle\Big)\nonumber \\ 
&\!+\!g_1(\rho_\cO,X_{\rho_\cO})\ind_{\{\rho_\cO\leq\tau\}\cap\{ \rho_\cO<u \}}\!+\!g_2(\tau,X_\tau)\ind_{\{\tau< u\wedge\rho_\cO \}}\!+\!v(u,X_u)\ind_{\{\tau\wedge\rho_\cO\geq u\}} \bigg].
\end{align*}
\end{theorem}

\begin{theorem}[{\bf Finite fuel}]\label{thm:DPPdetff}
Fix $(t,x,z)\in[0,T]\times\bR^d\times[0,\bar z]$ and $\nu\in\cV_t$. For any $u\in[t,T]$, we have
\begin{align*}
v(t,x,z)=\sup_{(\tau,\alpha,\xi)\in\AA^\nu_{t,\bar z-z}} \bE\bigg[&\!\int_t^{u\wedge\tau\wedge\rho_\cO}\!\!\! f(s,X_s,Z_s,\alpha_s)\ud s\!-\!\int_{[t,u\wedge\tau\wedge\rho_\cO)}\!\langle c_+(s,X_s),\ud\xi^+_s\rangle\!\nonumber\\
&-\! \int_{[t,u\wedge\tau\wedge\rho_\cO)}\!\langle c_-(s,X_s),\ud\xi^-_s\rangle+g_1(\rho_\cO,X_{\rho_\cO},Z_{\rho_\cO})\ind_{\{\rho_\cO\leq\tau\}\cap\{ \rho_\cO<u \}} \\
&+g_2(\tau,X_\tau,Z_\tau)\ind_{\{\tau< u\wedge\rho_\cO \}}+v(u,X_u,Z_u)\ind_{\{\tau\wedge\rho_\cO\geq u\}} \bigg].\nonumber
\end{align*}
\end{theorem}

We also have a more probabilistic interpretation of the above results. We state it in the next proposition, where we recall the definition of the process $M$ in \eqref{eq:M} and $N$ in \eqref{eq:N}.
\begin{proposition}\label{Prop:Msupermatingale}
For any admissible treble $(\tau,\alpha,\xi)$ the process $(M^{\nu;\alpha,\xi}_{s\wedge \tau})_{s\in[t,T]}$ defined in \eqref{eq:M} is a supermartingale in the reference probability system $\nu$. 
Assume further that the treble $(\tau^*,\alpha^*,\xi^*)$ is optimal and that
\begin{equation}\label{MartingaleIntegrability}
\bE[|N^{\alpha^*,\xi^*}_{u\wedge\tau^*}|+|v(u\wedge\tau^*,X^{\alpha^*,\xi^*}_{u\wedge\tau^*},Z^{\xi^*}_{u\wedge\tau^*})|]<\infty,\quad\text{for all $u\in[t,T]$}. 
\end{equation}
Then, the associated process $(M^{\nu;\alpha^*,\xi^*}_{s\wedge \tau^*})_{s\in[t,T]}$ is a martingale. The results hold for both finite and infinite fuel.
\end{proposition}
Finally, under the additional Assumption \ref{ass:DCT}, we obtain the DPP for stopping times.
\begin{theorem}[{\bf Infinite fuel}]\label{thm:DPPst}
Fix $(t,x)\in[0,T]\times\bR^d$ and $\nu\in\cV_t$. Under Assumption \ref{ass:DCT}, for any $\sigma\in\cT^\nu_t$, we have
\begin{align}\label{DPPst}
v(t,x)=\sup_{(\tau,\alpha,\xi)\in\AA^\nu_t} \bE\bigg[&\!\int_t^{\sigma\wedge\tau\wedge\rho_\cO}\!\!\! f(s,X_s,\alpha_s)\ud s\!-\!\int_{[t,\sigma\wedge\tau\wedge\rho_\cO)}\!\!\Big(\langle c_+(s,X_s),\ud\xi^+_s\rangle\!+\!\langle c_-(s,X_s),\ud\xi^-_s\rangle\Big)\nonumber\\
&\!+g_1(\rho_\cO,X_{\rho_\cO})\ind_{\{\rho_\cO\leq\tau\}\cap\{ \rho_\cO<\sigma \}}\!+\!g_2(\tau,X_\tau)\ind_{\{\tau< \sigma\wedge\rho_\cO \}}\!+\!v(\sigma,X_\sigma)\ind_{\{\tau\wedge\rho_\cO\geq \sigma\}} \bigg].
\end{align}
\end{theorem}

\begin{theorem}[{\bf Finite fuel}]\label{thm:DPPsttff}
Fix $(t,x,z)\in[0,T]\times\bR^d\times[0,\bar z]$ and $\nu\in\cV_t$. Under Assumption \ref{ass:DCT}, for any $\sigma\in\cT^\nu_t$, we have
\begin{align*}
v(t,x,z)=\sup_{(\tau,\alpha,\xi)\in\AA^\nu_{t,\bar z-z}} \bE\bigg[&\!\int_t^{\sigma\wedge\tau\wedge\rho_\cO}\!\!\! f(s,X_s,Z_s,\alpha_s)\ud s\!-\!\int_{[t,\sigma\wedge\tau\wedge\rho_\cO)}\!\langle c_+(s,X_s),\ud\xi^+_s\rangle\!\nonumber\\
&-\! \int_{[t,\sigma\wedge\tau\wedge\rho_\cO)}\!\langle c_-(s,X_s),\ud\xi^-_s\rangle+g_1(\rho_\cO,X_{\rho_\cO},Z_{\rho_\cO})\ind_{\{\rho_\cO\leq\tau\}\cap\{ \rho_\cO<\sigma \}} \\
&+g_2(\tau,X_\tau,Z_\tau)\ind_{\{\tau< \sigma\wedge\rho_\cO \}}+v(\sigma,X_\sigma,Z_\sigma)\ind_{\{\tau\wedge\rho_\cO\geq \sigma\}} \bigg].\nonumber
\end{align*}
\end{theorem}

In the next sections we will develop the theoretical framework that allows to prove the main results stated above. The key steps are two: 
\begin{enumerate} 
\item Showing the equivalence of strong and weak formulation via the so-called {\em independence of the reference probability system} (Section \ref{sec:ind});
\item Combining the use of strong and weak formulation with the use of regular conditional probabilities to arrive at the DPP (Section \ref{Sect:DPP}).
\end{enumerate}
For our analysis we follow closely the approach and main ideas in \cite[Chapter 2]{fabbri2017stochastic}, where the DPP is obtained in an infinite-dimensional setting. In \cite{fabbri2017stochastic} only classical controls are considered and without discretionary stopping or exit times from a given domain. As stated in \cite[Remark 2.15]{fabbri2017stochastic}, the fine technical details of the proofs are extremely sensitive to variations in the problem setting  and particularly to conditions imposed on the class of admissible controls that go beyond their measurability (for instance, left-continuity, bounded variation and the integrability/finite-fuel condition for singular controls, in our case). Additional difficulties arise from the discretionary stopping and the exit time $\rho_\cO$. Thus, we develop specific arguments to address our needs.

\section{Independence of the reference probability system}\label{sec:ind}
	
In this section we show that the problem is independent of the choice of the reference probability system $\nu\in\cV_t$ and thus that the strong formulation \eqref{StrongFormulation} and the weak formulation \eqref{WeakFormulation} coincide (respectively, for finite-fuel, \eqref{SFff} and \eqref{WFff}). In particular, for every $(t,x)\in[0,T]\times\bR^d$ and $\nu\in\cV_t$ given and fixed, we are going to show that we have
\begin{equation}\label{IndepRefProbSyst}
w^\nu(t,x)=v(t,x),
\end{equation}
and analogously for the finite-fuel case. We develop all arguments in this section for the infinite-fuel problem for notational simplicity and, when necessary, we show that their analogue for the finite-fuel setting holds with obvious changes. From now on, let $(t,x)\in [0,T]\times\bR^d$ be fixed (analogously $(t,x,z)\in [0,T]\times\bR^d\times[0,\bar z]$ are fixed in the finite-fuel setting). Unless stated otherwise, $\nu\in \cV_t$ is an arbitrary reference probability system.
	
It is useful to note an equivalent representation of stopping times in terms of a non-decreasing process. Indeed, for $\tau\in\cT^\nu_t$ we can define $\eta_s^\tau=\ind_{\{s>\tau\}}$ for $s\in[t,T]$ so that $(\eta_s^\tau)_{s\in[t,T]}$ is non-decreasing, left-continuous, with a single jump at time $\tau$. Motivated by this simple observation, given $u\in[t,T]$, we denote by $\cE^\nu_u$ the collection of processes $\eta=(\eta_s)_{s\in[t,T]}$ such that
\begin{enumerate}[(i)]
\item $\eta$ is $\{\cF^t_s\}$-adapted.
\item $\eta$ is left-continuous and non-decreasing $\bP$-a.s.
\item $\eta_s\in\{0,1\}$ for every $s\in[t,T]$ and $\eta_s=0$ for every $s\in[t,u]$, $\bP$-a.s.
\end{enumerate}
Then, $\eta^\tau\in\cE^\nu_t$ for $\tau\in\cT^\nu_t$ and, conversely, given $\eta\in\cE^\nu_t$ we define a $\{\cF^t_s\}$-stopping time 
\begin{align}\label{eq:etatau}
\tau^\eta=\inf\{s\geq t: \eta_{s}=1 \}\wedge T\in\cT^\nu_t.
\end{align}
Clearly $\cE^\nu_t\subset\cX^\nu_t$, so that properties which we will prove below for elements of $\cX^\nu_t$ immediately hold for elements of $\cE^\nu_t$ as well. 

Noticing that $\ind_{\{s\le \tau\}}=1-\eta^\tau_s$ and $\ind_{\{s< \tau\}}=1-\eta^\tau_{s+}$ we can also rewrite the objective function in another convenient form. For every admissible treble $(\tau,\alpha,\xi)$ we have
\begin{align}\label{TauEtaEquivalence}
J^\nu_{t,x}(\tau,\alpha,\xi)
&=\bE\bigg[\!\int_t^{\rho_{\cO}}\!\!(1-\eta^\tau_{s})f(s,X_s,\alpha_s)\ud s\!-\!\int_{[t,\rho_{\cO})}\!(1\!-\!\eta^\tau_{s+})\Big(\langle c_+(s,X_s),\ud\xi^+_s\rangle\!+\!\langle c_-(s,X_s),\ud\xi^-_s\rangle\Big) \nonumber\\
&\hspace{30pt}+g_1(\rho_{\cO},X_{\rho_{\cO}})(1-\eta^\tau_{\rho_{\cO}})+\int_{[t,\rho_{\cO})}g_2(s,X_s)\ud\eta^\tau_s\bigg]=:I^\nu_{t,x}(\eta^\tau,\alpha,\xi).
\end{align}
Conversely, for every $(\eta,\alpha,\xi)\in\cE^\nu_t\times\cA^\nu_t\times\cX^\nu_t$, we have
\begin{equation}\label{EtaTauEquivalence}
I^\nu_{t,x}(\eta,\alpha,\xi)=J^\nu_{t,x}(\tau^\eta,\alpha,\xi).
\end{equation}
For the finite-fuel case $J^\nu_{t,x,z}(\tau,\alpha,\xi)=I^\nu_{t,x,z}(\eta^\tau,\alpha,\xi)$ and $I^\nu_{t,x,z}(\eta,\alpha,\xi)=J^\nu_{t,x,z}(\tau^\eta,\alpha,\xi)$ by the exact same argument. For simplicity, but with a slight abuse of notation, we sometimes say that $(\eta^\tau,\alpha,\xi)$ is an admissible treble belonging to either $\AA^\nu_{t}$ or $\AA^\nu_{t,\bar z-z}$, provided that $(\tau,\alpha,\xi)$ is such.

The first difficulty in establishing the equivalence in \eqref{IndepRefProbSyst} is that null-sets and probabilistic properties vary along with the underlying reference probability systems. We now show that while all processes are adapted to the $\bP$-augmented Brownian filtration $\{\cF^t_s\}$, it is however possible to select representatives (in a suitable sense) that are adapted to the raw Brownian filtration $\{\cF^{t,0}_s\}$.

For classical controls the result can be found directly in \cite{fabbri2017stochastic}, hence we do not prove it here. We refer the reader to the final item in our Section \ref{sec:notation} for a formal definition of $\{\cF^{t,0}_s \}$-predictable process used below.
\begin{lemma}\label{Lemma:AlphaAlpha0}
Given $\alpha\in\cA^\nu_t$ there exists a $\{\cF^{t,0}_s \}$-predictable process $\alpha^0\in\cA^{\nu}_t$ such that $\alpha^0=\alpha$, $\ud s\times\bP$-a.e.~on $[t,T]\times\Omega$.
\end{lemma}
\begin{proof}
See \cite[Lemma 1.99]{fabbri2017stochastic} or \cite[Lemma 2.4]{soner2011quasi}.
\end{proof}

A slightly stronger result, that we prove in the next lemma, holds for general left-continuous $\{\cF^t_s \}$-adapted processes and hence also for those in $\cX^\nu_t$ and $\cE^\nu_t$, for the state process $X$ and for the fuel process $Z$. An analogous result is stated without proof in \cite{dellacherie1978probabilities} after Theorem IV.78.

\begin{lemma}\label{Lemma:Predict0Indisting}
Let $\nu=(\Omega,\cF,\bP,\{\cF^t_s\},W)\in\cV_t$ and $\gamma$ be an $\bR^d$-valued, $\{\cF^t_s \}$-adapted process which is $\bP$-a.s.~left-continuous. Then, there exists an $\{\cF^{t,0}_s\}$-predictable process $\gamma^0$ which is indistinguishable from $\gamma$.
\end{lemma}

\begin{proof}
Let $\Omega_0\subseteq\Omega$ such that $\bP(\Omega_0)=1$ and $s\mapsto \gamma_s(\omega)$ is left-continuous for every $\omega\in\Omega_0$. Let $n\in\bN$ and let $\{t_i^n\}_{i=0}^{2^n}$ be the corresponding dyadic partition of $[t,T]$, i.e., $t_i^n=t+i/2^n(T-t)$. By defining the sequence of processes
\begin{equation}
\gamma^n_s:=\gamma_t\ind_{\{t \}}(s)+\sum_{i=0}^{2^n-1}\gamma_{t_i^n}\ind_{(t_i^n,t^n_{i+1}]}(s), \quad s\in[t,T],
\end{equation}
we obtain that $\gamma_s^n(\omega)\to\gamma_s(\omega)$ for every $(s,\omega)\in[t,T]\times\Omega_0$ as $n\to\infty$. For every $n\in\bN$ and $i=0,1,\ldots,2^n$ there exists (see, e.g., \cite[Lemma 1.25]{kallenberg2006foundations}) $\gamma_{t_i^n}^0$ that is $\cF^{t,0}_{t_i^n}$-measurable and such that $\bP(\gamma_{t_i^n}=\gamma^0_{t_i^n})=1$. Thus, for every $n\in\bN$, $i=0,1,\ldots,2^n$ there exists $\Omega^{n,i}\subseteq\Omega$ with $\bP(\Omega^{n,i})=1$ such that for every $\omega\in\Omega^{n,i}$ we have $\gamma_{t_i^n}(\omega)=\gamma^0_{t_i^n}(\omega)$. Let $\bar{\Omega}:=\Omega_0\cap(\cap_{n,i}\Omega^{n,i})$, then $\bP(\bar{\Omega})=1$ and for every $\omega\in\bar{\Omega}$ we have $\gamma_{t_i^n}(\omega)=\gamma^0_{t_i^n}(\omega)$ for every $n\in\bN$, $i=0,1,\ldots,2^n$. By defining, for $n\in\bN$, the new sequence of processes
\begin{equation}
\gamma^{0,n}_s:=\gamma^0_t\ind_{\{t \}}(s)+\sum_{i=0}^{2^n-1}\gamma^0_{t_i^n}\ind_{(t_i^n,t^n_{i+1}]}(s), \quad s\in[t,T],
\end{equation}
we obtain that $\gamma^{0,n}$ is $\{\cF^{t,0}_s \}$-predictable and $\gamma^{0,n}_s(\omega)=\gamma^n_s(\omega)$ for every $(s,\omega)\in[t,T]\times\bar{\Omega}$, $n\in\bN$. Let
\begin{equation}
\gamma^0_s:=\liminf_{n\to\infty}\gamma^{0,n}_s, \quad s\in[t,T].
\end{equation}
Then, $\gamma^0$ is $\{\cF^{t,0}_s \}$-predictable and if $\omega\in\bar{\Omega}$ we have
\begin{equation}
\gamma^0_s(\omega):=\liminf_{n\to\infty}\gamma^{0,n}_s(\omega)=\liminf_{n\to\infty}\gamma^n_s(\omega)=\gamma_s(\omega), \quad \forall \: s\in[t,T],
\end{equation}
i.e., $\gamma^0$ and $\gamma$ are indistinguishable.
\end{proof}

Our next goal is to show that any $\{\cF^{t,0}_s\}$-predictable process on a reference probability system $\nu$ can be expressed as a deterministic, measurable function of the Brownian paths. For that we must introduce the {\em canonical} reference probability system:  
\begin{equation}\label{CanRefProbSyst}
\nu^*:=(\Omega^*,\cF^*,\bP^*,\{\cB^t_s\}_{s\in[t,T]},W^*),
\end{equation}
where
\begin{itemize} 
\item[(i)] $\Omega^*:=\{\omega\in C([t,T];\bR^{d'}): \omega(t)=0\}$; 
\item[(ii)] $\bP^*$ is the Wiener measure on $(\Omega^*,\cB(\Omega^*))$ that makes the canonical process $(s,\omega)\mapsto W^*_s(\omega)=\omega(s)$ a Brownian motion starting at time $t$; 
\item[(iii)] $\cF^*$ is the completion of $\cB(\Omega^*)$ with the $\bP^*$-null sets; 
\item[(iv)] $\cB^{t,0}_s:=\sigma(W^*_u: u\in [t,s])$ and $\cB^t_s$ is the augmentation of $\cB^{t,0}_s$ with the $\bP^*$-null sets.
\end{itemize}\medskip

Since $\Omega^*$ with the usual supremum norm is a Polish space, then $(\Omega^*,\cB(\Omega^*))$ is a standard measurable space (see, e.g., \cite[Chapter V, Theorem 2.2]{parthasarathy2005probability}) and so $\nu^*$ is a standard reference probability system. We denote by $\cP_{\Omega^*}$ the $\sigma$-algebra of $\{\cB^{t,0}_s\}$-predictable sets (see the last item in Section \ref{sec:notation}).	
\begin{lemma}\label{Lemma:CanonicalControl}
Let $\nu=(\Omega,\cF,\bP,\{\cF^t_s\},W)\in\cV_t$ and $\gamma$ be an $\bR^d$-valued, $\{\cF^{t,0}_s\}$-predictable process. There exists a $\cP_{\Omega^*}/\cB(\bR^d)$-measurable function $\psi:[t,T]\times\Omega^*\to\bR^d$ such that
\begin{equation}\label{CanonicalControl}
\gamma_s(\omega)=\psi(s,W_\cdot(\omega)), \quad (s,\omega)\in[t,T]\times\Omega.
\end{equation}
\end{lemma}	
\begin{proof}
Define $\beta:[t,T]\times\Omega\to [t,T]\times\Omega^*$ by $\beta(s,\omega):=(s,W_\cdot(\omega))$. For $t\le r< s <u\le T$ and $B\in\cB(\bR^d)$, sets of the form
\begin{equation}\label{genset}
\begin{aligned}
B^r_{s,u}&=(s,u]\times\{\omega\in\Omega^*: W^*_r(\omega)\in B\},\\
B_t&=\{t \}\times \{\omega\in\Omega^*: W^*_t(\omega)\in B\},
\end{aligned}
\end{equation}
generate $\cP_{\Omega^*}$.
By definition of the map $\beta$, we have $\beta^{-1}(B^r_{s,u})=A^r_{s,u}$ and $\beta^{-1}(B_t)=A_t$ where
\begin{align*}
A^r_{s,u}&=(s,u]\times\{\omega\in\Omega: W_r(\omega)\in B\},\\
A_t&=\{t\}\times\{\omega\in\Omega: W_t(\omega)\in B\}.
\end{align*}
Since the sets of the form $A^r_{s,u}$ and $A_t$ generate $\cP_\Omega$, we obtain that $\cP_{\Omega}=\beta^{-1}(\cP_{\Omega^*})$.

By assumption, $\gamma$ is $\cP_{\Omega}/\cB(\bR^d)$-measurable, then it is $\sigma(\beta)/\cB(\bR^d)$-measurable. Thus, there must exist a $\cP_{\Omega^*}/\cB(\bR^d)$-measurable function $\psi:[t,T]\times\Omega^*\to\bR^d$ such that $\gamma=\psi\circ\beta$ as claimed in \eqref{CanonicalControl} (see, e.g, \cite[Lemma 1.13]{kallenberg2006foundations}). 
\end{proof}
	
Next we show how the map $\psi:[t,T]\times\Omega^*\to\bR^d$ in \eqref{CanonicalControl} can be used to connect admissible controls in different reference probability systems that have the same law. The first lemma below is a standard result on distributional properties induced by measurable maps.
\begin{lemma}\label{lem:lawpsi}
Fix $n\in\bN$ and let $\psi:[t,T]\times\Omega^*\to\bR^n$ be $\cP_{\Omega^*}/\cB(\bR^n)$-measurable. For any two  reference probability systems $\nu=(\Omega,\cF,\bP,\{\cF^t_s\},W)$ and $\tilde\nu=(\tilde\Omega,\tilde\cF,\tilde\bP,\{\tilde{\cF}^t_s\},\tilde W)$ set $\gamma_s=\psi(s,W_\cdot)$ and $\tilde\gamma_s=\psi(s,\tilde W_\cdot)$, $s\in[t,T]$. Then, we have
$\cL_\bP(\gamma)=\cL_{\tilde\bP}(\tilde\gamma)$.
\end{lemma}
\begin{proof}
We use a monotone class theorem. Let $\cQ^0$ be the collection of all $\cP_{\Omega^*}/\cB(\bR^n)$-measurable {\em bounded} functions $\psi$ and let $\cQ^1\subset\cQ^0$ be the subset of those $\psi$'s for which $\cL_\bP(\gamma)=\cL_{\tilde\bP}(\tilde\gamma)$. Clearly $\cQ^1$ contains indicators of the sets that generate $\cP_{\Omega^*}$ and functions that are finite products of such indicators. Let us denote the latter set by $\cQ^2$, so that $\cQ^2$ is closed under multiplication and $\cQ^2\subset\cQ^1\subset \cQ^0$. The set $\cQ^1$ is closed under pointwise convergence: if $\psi_k\to \psi$ pointwise in $[t,T]\times\Omega^*$ and $\cL_\bP(\gamma^k)=\cL_{\tilde\bP}(\tilde{\gamma}^k)$ for $\gamma^k_s=\psi_k(s,W_\cdot)$ and $\tilde{\gamma}^k_s=\psi_k(s,\tilde{W}_\cdot)$, then for any finite collection of times $\{s_\ell\}_{\ell=0}^m\subset[t,T]$ and any sequence of vectors $\{\lambda_\ell\}_{\ell=0}^m\subset\bR^{n}$ we have
\[
\bE[\e^{i\sum_{\ell=0}^m\langle\lambda_\ell,\gamma_{s_\ell}\rangle}]=\lim_{k\to\infty}\bE[\e^{i\sum_{\ell=0}^m\langle\lambda_\ell,\gamma^k_{s_\ell}\rangle}]=\lim_{k\to\infty}\bE[\e^{i\sum_{\ell=0}^m\langle\lambda_\ell,\tilde{\gamma}^k_{s_\ell}\rangle}]=\bE[\e^{i\sum_{\ell=0}^m\langle\lambda_\ell,\tilde\gamma_{s_\ell}\rangle}],
\]
by dominated convergence. Therefore, by \cite[Th.\ I.21, p.14]{dellacherie1978probabilities} $\cQ^1$ contains all bounded $\sigma(\cQ^2)$-measurable functions and since $\sigma(\cQ^2)=\cP_{\Omega^*}$ we obtain the claim for bounded and $\cP_{\Omega^*}/\cB(\bR^n)$-measurable functions. 

Given a generic $\cP_{\Omega^*}/\cB(\bR^n)$-measurable function $\psi$ which is not necessarily bounded, we can approximate it with bounded functions $\psi^M$ by truncation so that $\psi^M\to \psi$ pointwise as $M\to\infty$. Then,
\[
\bE[\e^{i\sum_{\ell=0}^m\langle\lambda_\ell,\gamma_{s_\ell}\rangle}]=\lim_{M\to\infty}\bE[\e^{i\sum_{\ell=0}^m\langle\lambda_\ell,\gamma^M_{s_\ell}\rangle}]=\lim_{M\to\infty}\bE[\e^{i\sum_{\ell=0}^m\langle\lambda_\ell,\tilde{\gamma}^M_{s_\ell}\rangle}]=\bE[\e^{i\sum_{\ell=0}^m\langle\lambda_\ell,\tilde\gamma_{s_\ell}\rangle}],
\]
where $\gamma^M_s=\psi^M(s,W_\cdot)$ and $\tilde{\gamma}^M_s=\psi^M(s,\tilde{W}_\cdot)$, and the second equality holds by the first part of the proof.
\end{proof}
In the next lemma we construct admissible controls that are equivalent in law under different reference probability systems.
 	
\begin{lemma}\label{Lemma:ChangeOfControls}
Let $\nu=(\Omega,\cF,\bP,\{\cF^t_s\},W)$  and $\tilde{\nu}=(\tilde{\Omega},\tilde{\cF},\tilde{\bP},\{\tilde{\cF}^t_s\},\tilde{W})$ be two distinct reference probability systems in $\cV_t$. Given $(\tau,\alpha,\xi)\in\AA^\nu_t$ (or in $\AA^\nu_{t,\bar z}$), there exists $(\tilde{\tau},\tilde{\alpha},\tilde{\xi})\in\AA^{\tilde\nu}_t$ (or in $\AA^{\tilde\nu}_{t,\bar z}$) that are $\{\tilde{\cF}^{t,0}_s\}$-predictable and such that
\begin{equation}\label{ControlsEqualinLaw}
\cL_\bP(\tau,\alpha,\xi,W)=\cL_{\tilde{\bP}}(\tilde{\tau},\tilde{\alpha},\tilde{\xi},\tilde{W}).
\end{equation}
\end{lemma}	

\begin{proof}
Let us first recall $\eta^\tau_s=\mathbf{1}_{\{s>\tau\}}$. Thanks to Lemma \ref{Lemma:AlphaAlpha0} and \ref{Lemma:Predict0Indisting} we can pass to the $\{\cF^{t,0}_s\}$-predictable representatives $(\eta^{\tau,0},\alpha^0,\xi^0)$ of $(\eta^\tau,\alpha,\xi)$. For simplicity, we denote $(\eta^{\tau,0},\alpha^0,\xi^0)=(\eta^{\tau},\alpha,\xi)$ with a slight abuse of notation.
	
By Lemma \ref{Lemma:CanonicalControl}, there exist a $\cP_{\Omega^*}/\cB(\bR^l)$-measurable function $\psi_\alpha:[t,T]\times\Omega^*\to\cK$, two $\cP_{\Omega^*}/\cB(\bR^d)$-measurable functions $\psi_\xi^\pm:[t,T]\times\Omega^*\to\bR^d$ and a $\cP_{\Omega^*}/\cB(\bN)$-measurable\footnote{Recall that $\cB(\bN)$ is the Borel sigma-algebra for the discrete topology.} function $\psi_\eta:[t,T]\times\Omega^*\to\{0,1\}$ such that, for all $(s,\omega)\in[t,T]\times\Omega$,
\begin{equation}\label{eq:psi1}
\alpha_s(\omega)=\psi_\alpha(s,W_\cdot(\omega)),\quad\xi^\pm_s(\omega)=\psi^\pm_\xi(s,W_\cdot(\omega)),\quad\eta^\tau_s(\omega)=\psi_{\eta^\tau}(s,W_\cdot(\omega)).
\end{equation}
Clearly the map $(s,\tilde{\omega})\mapsto(s,\tilde{W}_\cdot(\tilde{\omega}))$ is $\cP_{\tilde{\Omega}}/\cP_{\Omega^*}$-measurable as in Lemma \ref{Lemma:CanonicalControl}. Thus, the processes defined for all $(s,\tilde{\omega})\in[t,T]\times\tilde{\Omega}$ by
\begin{align}\label{eq:psi2}
\tilde{\alpha}_s(\omega):=\psi_\alpha(s,\tilde{W}_\cdot(\tilde{\omega})), \quad \hat{\xi}^\pm_s(\tilde{\omega}):=\psi^\pm_\xi(s,\tilde{W}_\cdot(\tilde{\omega})), \quad \hat{\eta}_s(\tilde{\omega}):=\psi_{\eta^\tau}(s,\tilde{W}_\cdot(\tilde{\omega})), 
\end{align}
are $\{\tilde{\cF}^{t,0}_s\}$-predictable. Hence, $\tilde{\alpha}\in\cA^{\tilde{\nu}}_t$ and it remains to construct $(\tilde\eta^\tau,\tilde\xi)\in\cE^{\tilde \nu}_t\times\cX^{\tilde \nu}_t$ (or in $\cE^{\tilde \nu}_t\times\cX^{\tilde \nu}_{t}(\bar z)$).
		
By left-continuity of $\xi$, we can define
\[
\psi^\pm_\xi(s-,W_\cdot(\omega)):=\lim_{u\uparrow s}\psi^\pm_\xi(u,W_\cdot(\omega))=\lim_{u\uparrow s}\xi^\pm_u(\omega)=\xi^\pm_{s-}(\omega)=\xi^\pm_{s}(\omega)=\psi^\pm_\xi(s,W_\cdot(\omega)).
\]
Let $D$ be the set of dyadic rationals in $[t,T]$, i.e., $D:=\{t_i^n, n\in\bN, i\in\{0,\ldots,2^n\} \}$ with $t^n_i=t+i/2^n(T-t)$. Notice that $D$ is countable and dense in $[t,T]$. Then, for every $s,t,u$ in $D$ with $s<t$, we have
\begin{align}\label{OmegaTildeknm}
1&=\bP\big(\xi^\pm_{s}\leq\xi^\pm_{t},\xi^\pm_{u}=\xi^\pm_{u-}\big)=\bP\big(\psi_\xi^\pm(s,W_\cdot)\leq\psi_\xi^\pm(t,W_\cdot),\psi_\xi^\pm(u,W_\cdot)=\psi_\xi^\pm(u-,W_\cdot)\big) \\
&=\tilde{\bP}\big(\psi_\xi^\pm(s,\tilde{W}_\cdot)\leq\psi_\xi^\pm(t,\tilde{W}_\cdot),\psi_\xi^\pm(u,\tilde{W}_\cdot)=\psi_\xi^\pm(u-,\tilde{W}_\cdot)\big)=\tilde{\bP}\big(\hat{\xi}^\pm_{s}\leq\hat{\xi}^\pm_{t},\hat{\xi}^\pm_{u}=\hat{\xi}^\pm_{u-}\big),\nonumber
\end{align}
where in the third equality we have used Lemma \ref{lem:lawpsi}. That is, for every $s,t,u\in D$ with $s<t$ the event $\tilde{\Omega}^{s,t}_u=\{\omega\in\tilde{\Omega}:\hat{\xi}^\pm_{s}(\omega)\leq\hat{\xi}^\pm_{t}(\omega),\,\hat{\xi}^\pm_{u}(\omega)=\hat{\xi}^\pm_{u-}(\omega)\}$ has full measure. Thus, letting $\hat{\Omega}:=\cap_{s,t,u\in D}\tilde{\Omega}^{s,t}_u$ we have $\tilde{\bP}(\hat{\Omega})=1$ and 
\[
\text{$\hat{\xi}^\pm_{s}(\omega)\leq\hat{\xi}^\pm_{t}(\omega)$ and $\hat{\xi}^\pm_{u}(\omega)=\hat{\xi}^\pm_{u-}(\omega)$ for every $s,t,u\in D$ with $s\leq u$,}
\] 
for every $\omega\in\hat{\Omega}$. Now, for every $n\in\bN$, we set
\[
\tilde{\xi}^{\pm,n}_s:=\hat{\xi}^\pm_t \ind_{\{t \}}(s)+\sum_{i=0}^{2^n-1}\hat{\xi}^\pm_{t^n_i}\ind_{(t^n_i,t^n_{i+1}]}(s), \quad s\in[t,T].
\]
Then, $\tilde{\xi}_s^{\pm,n}$ is $\{\tilde{\cF}^{t,0}_s \}$-predictable and, for every $\omega\in\hat{\Omega}$, we have $\lim_{n\to\infty}\tilde{\xi}_s^{\pm,n}(\omega)=\hat{\xi}_s^{\pm}(\omega)$ for $s\in D$. We define $\tilde{\xi}^\pm_s(\omega):=\liminf_{n\to\infty}\tilde{\xi}^{\pm,n}_s(\omega)$ for $(s,\omega)\in[t,T]\times\Omega$,
so that $\tilde{\xi}^\pm$ is $\{\tilde{\cF}^{t,0}_s \}$-predictable and
\begin{align}\label{TildeXi}
\tilde\bP(\tilde \xi^\pm_s=\hat \xi^\pm_s,\:s\in D)=1.
\end{align}
For each $\omega\in\hat \Omega$ the mapping $D\ni s\mapsto\hat \xi^\pm_s(\omega)$ is non-decreasing (componentwise) so that $\tilde \xi^\pm$ is $\tilde\bP$-a.s.\ non-decreasing in time, with $\tilde \xi^\pm_0=0$. Moreover, $\tilde\xi^{\pm,n}_s(\omega)\le \tilde\xi^{\pm,n+1}_s(\omega)$ for all $n\in\bN$ and all $(s,\omega)\in[t,T]\times\hat \Omega$, with the inequality understood componentwise. Thus, for any $(s,\omega)\in(t,T]\times\hat \Omega$ we have
\begin{align*}
\lim_{\eps\downarrow 0}\tilde\xi^\pm_{s-\eps}(\omega)=\lim_{\eps\downarrow 0}\lim_{n\uparrow\infty}\tilde\xi^{\pm,n}_{s-\eps}(\omega)=\lim_{n\uparrow\infty}\lim_{\eps\downarrow 0}\tilde\xi^{\pm,n}_{s-\eps}(\omega)=\lim_{n\uparrow\infty}\tilde\xi^{\pm,n}_{s}(\omega)=\tilde\xi^{\pm}_{s}(\omega),
\end{align*}
where the limits can be swapped by monotonicity (again componentwise). This proves that $\tilde\xi^\pm$ is $\bP$-a.s.\ left-continuous.

Setting $\tilde \xi=\tilde\xi^+-\tilde\xi^-$, it remains to prove integrability (or finite-fuel) property of $\tilde \xi$.
If $\xi\in\cX^\nu_t(\bar z)$, one simply adds the $\bP$-a.s.\ condition $V_{[t,T]}(\xi)=\sum_{i=1}^d\big(\xi^{i,+}_T+\xi^{i,-}_T\big)\le \bar z$ in \eqref{OmegaTildeknm} and thus obtains that $V_{[t,T]}(\tilde{\xi})=\sum_{i=1}^d\big(\tilde{\xi}^{i,+}_T+\tilde{\xi}^{i,-}_T\big)\le \bar z$, $\tilde{\bP}$-a.s. Hence, $\tilde\xi\in\cX^{\tilde\nu}_{t}(\bar z)$.
If instead $\xi\in\cX^\nu_t$, we must check that $\tilde{\bE}[|V_{[t,T]}(\tilde{\xi})|^p]<\infty$ where $\tilde{\bE}$ denotes the expectation under the measure $\tilde{\bP}$. This will follow once we prove \eqref{ControlsEqualinLaw}. 
By an analogous construction to the one for $\tilde{\xi}$, we also find a $\{\tilde{\cF}^{t,0}_s \}$-predictable process $\tilde{\eta}^\tau\in\cE^{\tilde{\nu}}_t$ from which we obtain the stopping time $\tilde \tau$ as in \eqref{eq:etatau}. 

Next we show that the equality in law \eqref{ControlsEqualinLaw} holds, 
by proving that all finite-dimensional distributions of $(\alpha,\xi,\eta^\tau,W)$ and $(\tilde\alpha,\tilde\xi,\tilde\eta^\tau,\tilde W)$ are the same (see, e.g., \cite[Proposition 2.2.]{kallenberg2006foundations}). Fix a finite sequence of times $\{s_k \}_{k=1}^n\subseteq [t,T]$ and sequences of vectors $\{\lambda^\alpha_k\}_{k=1}^n\subset \bR^l$, $\{\lambda^\xi_k\}_{k=1}^n\subset \bR^d$, $\{\lambda^\eta_k\}_{k=1}^n\subset \bR$ and $\{\gamma_k\}_{k=1}^n\subset \bR^{d'}$. Then, by dominated convergence and left-continuity of $\eta^\tau$ and $\xi$, we obtain
\begin{align*}
&\bE\Big[\exp\Big(i\sum_{k=1}^n(\langle\lambda^\alpha_k,\alpha_{s_k}\rangle+\langle\lambda^\xi_k,\xi_{s_k}\rangle+\langle\lambda^\eta_k,\eta_{s_k}\rangle+\langle\gamma_k,W_{s_k}\rangle)\Big)\Big]\\
&=\lim_{\substack{\{r^j_k\}\uparrow \{s_k\}\\ \{r^j_k\}\in D}}\bE\Big[\exp\Big(i\sum_{k=1}^n(\langle\lambda^\alpha_k,\alpha_{s_k}\rangle+\langle\lambda^\xi_k,\xi_{r^j_k}\rangle+\langle\lambda^\eta_k,\eta_{r^j_k}\rangle+\langle\gamma_k,W_{s_k}\rangle)\Big)\Big],
\end{align*}
where the limit is understood as a limit over sequences $\{r^j_k\}_{j\in\bN}$ such that $r^j_k\uparrow s_k$ for each $k$ and $r^j_k\in D$ for each $(j,k)$. By \eqref{TildeXi} we know that 
\[
\big(\tilde \xi_s(\omega),\tilde\eta^\tau_s(\omega)\big)=\big(\psi_\xi(s,\tilde W_\cdot(\omega)),\psi_{\eta^\tau}(s,\tilde W_\cdot(\omega))\big),\quad \text{for $(s,\omega)\in D\times\hat \Omega$}.
\]
Therefore by \eqref{eq:psi1}, \eqref{eq:psi2} and Lemma \ref{lem:lawpsi} we have
\begin{align*}
&\lim_{\substack{\{r^j_k\}\uparrow \{s_k\}\\ \{r^j_k\}\in D}}\bE\Big[\exp\Big(i\sum_{k=1}^n\big(\langle\lambda^\alpha_k,\alpha_{s_k}\rangle+\langle\lambda^\xi_k,\xi_{r^j_k}\rangle+\langle\lambda^\eta_k,\eta_{r^j_k}\rangle+\langle\gamma_k,W_{s_k}\rangle\big)\Big)\Big]\\
&=\lim_{\substack{\{r^j_k\}\uparrow \{s_k\}\\ \{r^j_k\}\in D}}\tilde\bE\Big[\exp\Big(i\sum_{k=1}^n\big(\langle\lambda^\alpha_k,\tilde\alpha_{s_k}\rangle+\langle\lambda^\xi_k,\tilde\xi_{r^j_k}\rangle+\langle\lambda^\eta_k,\tilde\eta_{r^j_k}\rangle+\langle\gamma_k,\tilde W_{s_k}\rangle\big)\Big)\Big]\\
&=\tilde\bE\Big[\exp\Big(i\sum_{k=1}^n\big(\langle\lambda^\alpha_k,\tilde\alpha_{s_k}\rangle+\langle\lambda^\xi_k,\tilde\xi_{s_k}\rangle+\langle\lambda^\eta_k,\tilde\eta_{s_k}\rangle+\langle\gamma_k,\tilde W_{s_k}\rangle\big)\Big)\Big],
\end{align*}
where the final equality is by dominated convergence and $\tilde\bP$-a.s.\ left-continuity of $(\tilde \xi,\tilde \eta^\tau)$. Then \eqref{ControlsEqualinLaw} holds. That also implies $\tilde{\bE}[|V_{[t,T]}(\tilde{\xi})|^p]=\bE[|V_{[t,T]}(\xi)|^p]$, so that $\tilde\xi\in\cX^{\tilde \nu}_t$ if $\xi\in\cX^{\nu}_t$.
\end{proof}

The equality in law under different reference probability systems extends also to the processes $X$ and $Z$ and the stopping time $\rho_\cO$, as illustrated in the next lemma, where we use the same notation as in Lemma \ref{Lemma:ChangeOfControls}. The proof of the lemma relies on a result by Kurtz \cite{kurtz2007yamada} about strong solutions of stochastic equations, that generalises the classical results by Yamada and Watanabe \cite[Ch.\ III]{ikeda1989stochastic}. Informally, \cite{kurtz2007yamada} states that if $Y$ is a stochastic input and $X$ is a stochastic output of an equation of the form
\begin{equation}\label{eq:KurtzEq}
\Gamma(X,Y)=0,
\end{equation}
then {\em pointwise uniqueness} (in the language of \cite{kurtz2007yamada}) implies $X=\Psi(Y)$ for some measurable function $\Psi$. Moreover, $\Psi$ is uniquely determined in the sense that if $(\tilde X,\tilde Y)$ is another pair solving \eqref{eq:KurtzEq} with $\tilde Y=Y$ in law, then $\tilde X=\Psi(\tilde Y)$. We will apply this result in the case that $X$ and $Y$ are c\`agl\`ad processes and, therefore, pointwise uniqueness coincides with {\em pathwise uniqueness}. 
\begin{lemma}\label{Coroll:EquivInLaw}
	Let $\nu,\tilde{\nu}\in\cV_t$, and take $(\tau,\alpha,\xi)\in\AA_t^\nu$ (or in $\AA_{t,\bar{z}}^\nu$) and $(\tilde{\tau},\tilde{\alpha},\tilde{\xi})\in\AA^{\tilde{\nu}}_t$ (or in $\AA_{t,\bar{z}}^{\tilde{\nu}}$) as in Lemma \ref{Lemma:ChangeOfControls}. Then, we have
	\begin{equation}\label{eq:corLaw}
	\cL_\bP\big(\tau,\alpha,\xi,W,X^{\nu\alpha,\xi},Z^{\nu\xi},\rho_{\cO}^{\nu;\alpha,\xi}\big)=\cL_{\tilde{\bP}}\big(\tilde{\tau},\tilde{\alpha},\tilde{\xi},\tilde W,X^{\tilde{\nu};\tilde{\alpha},\tilde{\xi}},Z^{\tilde{\nu};\tilde{\xi}},\rho_{\cO}^{\tilde{\nu};\tilde{\alpha},\tilde{\xi}}\big).
	\end{equation}
\end{lemma}

\begin{proof}
	For simplicity and in keeping with the proof of Lemma \ref{Lemma:ChangeOfControls}, we denote $(\eta^{\tau,0},\alpha^0,\xi^0,X^0)$ by $(\eta^\tau,\alpha,\xi,X)$, with a slight abuse of notation.
	
	By Assumption \ref{ass:Xind}, $X^{\alpha,\xi}$ is the pathwise unique solution of \eqref{StateProcess}, which can be written in the form of \eqref{eq:KurtzEq} with $Y=(\alpha,\xi,W)$. Then, by \cite[Corollary 2.8]{kurtz2007yamada}, $X^{\alpha,\xi}=\Psi_X(\alpha,\xi,W)$ for some measurable function $\Psi_X$. Moreover, $\cL_\bP(\alpha,\xi,W)=\cL_{\tilde{\bP}}(\tilde{\alpha},\tilde{\xi},\tilde{W})$ implies that $X^{\tilde{\alpha},\tilde{\xi}}=\Psi_X(\tilde{\alpha},\tilde{\xi},\tilde{W})$ solves \eqref{StateProcess} in the reference probability system $\tilde\nu$.
	
	Next, recall from the construction in Lemma \ref{Lemma:ChangeOfControls} that $\hat{\xi}_s=\psi_\xi(s,W_\cdot)$ and that $\hat \xi_s(\omega)=\tilde\xi_s(\omega)$ for $(s,\omega)\in D\times \hat \Omega$ with $D$ countable and dense in $[0,T]$ and $\tilde\bP(\hat \Omega)=1$. Then, since $\tilde \xi$ is also left-continuous $\tilde{\bP}$-a.s., we obtain
	\begin{align*}
	V_{[t,s]}(\tilde \xi)=&\sum_{i=1}^d\sup_{t\le r\le s}\big(\tilde\xi^{i,+}_r+\tilde\xi^{i,-}_r\big)=\sum_{i=1}^d\sup_{r\in[t,s]\cap D}\big(\tilde\xi^{i,+}_r+\tilde\xi^{i,-}_r\big)\\
	=&\sum_{i=1}^d\sup_{r\in[t,s]\cap D}\big(\hat \xi^{i,+}_r+\hat \xi^{i,-}_r\big)=:V_{[t,s]\cap D}(\hat \xi),\quad\text{for all $s\in[t,T]$, $\tilde\bP$-a.s.}
	\end{align*}
	Therefore, $Z^{\tilde\nu;\tilde\xi}_s=\psi_Z(s,\tilde W_\cdot)$ for some $\cP_{\Omega^*}/\cB(\bR_+)$-measurable function $\psi_Z$, because $\hat \xi_s=\psi_\xi(s,\tilde W_\cdot)$. 
	By left-continuity of the control $\xi$ in $\nu$, we also have $V_{[t,s]}(\xi)=V_{[t,s]\cap D}(\xi)$ and therefore it is clear that also $Z^{\nu;\xi}_s=\psi_Z(s,W_\cdot)$ for the same $\psi_Z$, because $\xi_s=\psi_\xi(s, W_\cdot)$.
	
	Recalling that $\cL_\bP(\tau,\alpha,\xi,W)=\cL_{\tilde{\bP}}(\tilde\tau,\tilde{\alpha},\tilde{\xi},\tilde{W})$, it is now clear that
	\begin{equation}\label{EqualLaw0}
	\cL_\bP\big(\tau,\alpha,\xi,W,X^{\nu;\alpha,\xi},Z^{\nu;\xi}\big)=\cL_{\tilde{\bP}}\big(\tilde{\tau},\tilde{\alpha},\tilde{\xi},\tilde W,X^{\tilde{\nu};\tilde{\alpha},\tilde{\xi}},Z^{\tilde{\nu};\tilde{\xi}}\big).
	\end{equation}
	Finally, to show that \eqref{eq:corLaw} holds, we set 
	\[
	\Sigma=(\eta^\tau,\alpha,\xi,W,X^{\nu;\alpha,\xi},Z^{\nu;\xi}) \quad \text{and} \quad \tilde{\Sigma}=(\tilde{\eta}^\tau,\tilde{\alpha},\tilde{\xi},\tilde W,X^{\tilde{\nu};\tilde{\alpha},\tilde{\xi}},Z^{\tilde{\nu};\tilde{\xi}})
	\]
	and we equivalently prove that the finite-dimensional distributions of $(\Sigma,\rho^{\alpha,\xi;\nu})$ and $(\tilde{\Sigma},\rho^{\tilde{\alpha},\tilde{\xi};\tilde{\nu}})$ coincide. 
	For every $t\leq s_1< s_2<\ldots<s_n\leq T$, $B\in\cB(\bR^N)$ with $N=n(2d+d'+l+2)$ and $u\in[t,T]$, we have
	\[
	\bP\big((\Sigma_{s_1},\ldots, \Sigma_{s_n})\in B, \rho^{\nu;\alpha,\xi}\geq u\big)=\tilde{\bP}\big((\tilde{\Sigma}_{s_1},\ldots, \tilde{\Sigma}_{s_n})\in B, \rho^{\tilde{\nu};\tilde{\alpha},\tilde{\xi}}\geq u\big).
	\]
	By definition \eqref{eq:rho} and left-continuity of $(X,Z)$, this is equivalent to showing that
	\begin{align*}
	\bP&\big((\Sigma_{s_1},\ldots \Sigma_{s_n})\in B, \{(X^{\nu;\alpha,\xi}_s,Z^{\nu;\xi}_s)\in\cO, \, \forall \: s\in[t,u) \}\big)\\
	&=\tilde{\bP}\big((\tilde{\Sigma}_{s_1},\ldots \tilde{\Sigma}_{s_n})\in B, \{(X^{\tilde{\nu};\tilde{\alpha},\tilde{\xi}}_s,Z^{\tilde{\nu};\tilde{\xi}}_s)\in\cO, \, \forall \: s\in[t,u) \}\big),
	\end{align*}
	which holds thanks to the equality in law \eqref{EqualLaw0} and the proof is complete.
\end{proof}

We are now ready to prove independence of the value function from the reference probability system, i.e., \eqref{IndepRefProbSyst}.
	
\begin{proposition}\label{Prop:IndepProbSyst}
Fix $(t,x)\in [0,T]\times \bR^d$ and let $\nu,\tilde{\nu}\in\cV_t$. For every $(\tau,\alpha,\xi)\in\AA^\nu_t$ there exists $(\tilde{\tau},\tilde{\alpha},\tilde{\xi})\in\AA^{\tilde{\nu}}_t$ such that
\begin{equation}\label{JequalJtilde}
J^\nu_{t,x}(\tau,\alpha,\xi)=J^{\tilde{\nu}}_{t,x}(\tilde{\tau},\tilde{\alpha},\tilde{\xi}).
\end{equation}
Consequently, for every $(t,x)\in[0,T]\times\bR^d$ and $\nu\in\cV_t$, we have that
\begin{equation}
w^\nu(t,x)=v(t,x).
\end{equation}
\end{proposition}
\begin{proof}
Fix an arbitrary treble $(\tau,\alpha,\xi)\in\AA^\nu_t$ and construct an admissible treble $(\tilde\tau,\tilde\alpha,\tilde\xi)\in\AA^{\tilde\nu}_t$ as in Lemma \ref{Lemma:ChangeOfControls} so that \eqref{ControlsEqualinLaw} holds. Then, Lemma \ref{Coroll:EquivInLaw} guarantees that 
\begin{equation}\label{EquivLaw}
\cL_\bP(\tau,\alpha,\xi,X^{\nu;\alpha,\xi},\rho^{\nu;\alpha,\xi}_\cO)=\cL_{\tilde{\bP}}(\tilde{\tau},\tilde{\alpha},\tilde{\xi},X^{\tilde{\nu};\tilde{\alpha},\tilde{\xi}},\rho^{\tilde{\nu};\tilde{\alpha},\tilde{\xi}}_\cO).
\end{equation}
Since we can write
\[
J_{t,x}^\nu(\tau,\alpha,\xi)=\bE\big[\varphi(\tau,\alpha,\xi,X^{\nu;\alpha,\xi},\rho^{\nu;\alpha,\xi}_\cO)\big]\quad \text{and} \quad J^{\tilde{\nu}}_{t,x}(\tilde{\tau},\tilde{\alpha},\tilde{\xi})=\tilde{\bE}\big[\varphi(\tilde{\tau},\tilde{\alpha},\tilde{\xi},X^{\tilde{\nu};\tilde{\alpha},\tilde{\xi}},\rho^{\tilde{\nu};\tilde{\alpha},\tilde{\xi}}_\cO)\big],
\]
for some measurable function $\varphi$, then by \eqref{EquivLaw} we obtain \eqref{JequalJtilde}.
This also yields that
\begin{equation}
J^\nu_{t,x}(\tau,\alpha,\xi)=J^{\tilde{\nu}}_{t,x}(\tilde{\tau},\tilde{\alpha},\tilde{\xi})\leq w^{\tilde{\nu}}(t,x), \quad \text{for all $(\tau,\alpha,\xi)\in\AA^\nu_t$}.
\end{equation}
Hence $w^\nu(t,x)\le w^{\tilde\nu}(t,x)$ and symmetrically we obtain $w^{\tilde{\nu}}(t,x)\leq w^\nu(t,x)$ so that $w^\nu(t,x)= w^{\tilde{\nu}}(t,x)$. Notice that the latter equality holds for any pair $\nu,\tilde\nu\in\cV_t$. Since the equality holds for any pair $\nu,\tilde\nu\in\cV_t$ then $v(t,x)=\sup_{\tilde{\nu}\in\cV_t}w^{\tilde{\nu}}(t,x)=w^\nu(t,x)$.
\end{proof}

By an identical argument of proof we also obtain the finite-fuel version of the above proposition. One only needs to recall the process $Z$ from \eqref{eq:Z} and notice that 
\[
\cL_\bP(\tau,\alpha,\xi,X^{\nu;\alpha,\xi}, Z^{\nu;\xi},\rho^{\nu;\alpha,\xi}_\cO)=\cL_{\tilde{\bP}}(\tilde{\tau},\tilde{\alpha},\tilde{\xi},X^{\tilde{\nu};\tilde{\alpha},\tilde \xi},Z^{\tilde \nu;\tilde \xi},\rho^{\tilde{\nu};\tilde{\alpha},\tilde{\xi}}_\cO),
\]
by Lemma \ref{Coroll:EquivInLaw}.

\begin{proposition}\label{prop:PS2}
Fix $(t,x,z)\in [0,T]\times \bR^d\times[0,\bar z]$ and let $\nu,\tilde{\nu}\in\cV_t$. For every $(\tau,\alpha,\xi)\in\AA^\nu_{t,\bar z-z}$ there exists $(\tilde{\tau},\tilde{\alpha},\tilde{\xi})\in\AA^{\tilde{\nu}}_{t,\bar z-z}$ such that
\begin{equation}\label{JequalJtilde2}
J^\nu_{t,x,z}(\tau,\alpha,\xi)=J^{\tilde{\nu}}_{t,x,z}(\tilde{\tau},\tilde{\alpha},\tilde{\xi}).
\end{equation}
Consequently, for every $(t,x,z)\in[0,T]\times\bR^d\times[0,\bar z]$ and $\nu\in\cV_t$, we have that
\begin{equation}
w^\nu(t,x,z)=v(t,x,z).
\end{equation}
\end{proposition}	
	
\section{Dynamic programming principle}\label{Sect:DPP}
	
Using regular conditional probabilities, in this section we prove the main results of the paper: Theorems \ref{thm:DPPdet} and \ref{thm:DPPdetff}, Proposition \ref{Prop:Msupermatingale} and Theorems \ref{thm:DPPst} and \ref{thm:DPPsttff}. In Appendix \ref{App:RegCondProb} we provide a detailed digression on regular conditional probabilities, for completeness, while here we only introduce the related notation.
	
Unless otherwise stated, we fix $(t,x,z)\in[0,T]\times\bR^d\times[0,\bar z]$ and a standard reference probability system $\nu=(\Omega,\cF,\bP,\{\cF^t_s\},W)\in\cV_t$. Then, we also fix $u\in(t,T)$ and denote by $\bP_\omega$ the regular conditional probability on $(\Omega,\cF)$ given $\cF^{t,0}_u$. That is, 
\[
\bP_\omega(A)=\bP(A|\cF^{t,0}_u)(\omega)\quad\text{for every $A\in\cF$ for $\bP$-a.e.\ $\omega$.}
\] 
The expectation with respect to $\bP_\omega$ is denoted by $\bE_\omega$ and the $\sigma$-algebra $\cF_\omega$ is the completion of $\cF^0$ with the $\bP_\omega$-null sets. To be precise, we should use $\bP^u_\omega$ instead of $\bP_\omega$ in order to keep track of the time $u$ with respect to which we evaluate the regular conditional probabilities. However, $u$ will be fixed throughout and so we can use a simpler notation.

Let $W^u_s:=W_s-W_u$ for $s\in[u,T]$ be the increments of $W$ after time $u$ and let $\cF^{u,0}_s:=\sigma(W^u_r: r\in[u,s])$ be the raw filtration generated by such increments. We denote by $\cF^{u,\omega}_s$ the augmentation of $\cF^{u,0}_s$ with the $\bP_\omega$-null sets. Now, for $\bP$-a.e.~$\omega\in\Omega$, we can define a standard reference probability system $\nu_\omega\in\cV_u$ as
\begin{align}\label{eq:nuomega}
\nu_\omega:=(\Omega,\cF_\omega,\bP_\omega,\{\cF^{u,\omega}_s\}_{s\in[u,T]}, W^u),
\end{align}
which will be frequently needed in the proofs below (it is indeed shown in Proposition \ref{Prop:StdRefProbSyst} that $\nu_\omega\in\cV_u$ and so, in particular, $W^u$ is a $\{\cF^{u,\omega}_s\}$-Brownian motion under $\bP_\omega$ for $\bP$-a.e~$\omega\in\Omega$).
	
In order to simplify notations in some proofs, for $s\in[t,T]$, on the event $\{s\le \tau\wedge\rho_\cO\}$, let us set
\begin{align}\label{DefY}
\Gamma_{s}(\rho_\cO,\tau,\alpha,\xi,X,Z):=&\int_s^{\tau\wedge\rho_\cO}\!\!f(r,X_r,Z_r,\alpha_r)\ud r\!-\!\int_{[s,\tau\wedge\rho_\cO)}\!\Big(\langle c_+(r,X_r),\ud\xi^+_r\rangle\!+\!\langle c_-(r,X_r),\ud\xi^-_r\rangle\Big)\nonumber\\
&+\!g_1(\rho_\cO,X_{\rho_\cO},Z_{\rho_\cO})\ind_{\{\rho_\cO\leq \tau \}}\!+\!g_2(\tau,X_\tau,Z_\tau)\ind_{\{\rho_\cO>\tau \}},
\end{align}
and	
\begin{align}\label{DefLambda}
\Lambda_{s}(\rho_\cO,\eta,\alpha,\xi,X,Z):=&\int_s^{\rho_\cO}\! (1\!-\!\eta_{r})f(r,X_r,Z_r,\alpha_r)\ud r \nonumber\\
&-\!\int_{[s,\rho_\cO)}\!(1\!-\!\eta_{r+})\langle c_+(r,X_r),\ud\xi^+_r\rangle\!-\!\int_{[s,\rho_\cO)}\!(1\!-\!\eta_{r+})\langle c_-(r,X_r),\ud\xi^-_r\rangle\\
&\!+\!g_1(\rho_\cO,X_{\rho_\cO},Z_{\rho_\cO})(1\!-\!\eta_{\rho_\cO})\!+\!\int_{[s,\rho_\cO)}g_2(r,X_r,Z_r)\ud\eta_r,\notag
\end{align}	
where $(X,Z)=(X^{\nu;\alpha,\xi},Z^{\nu;\xi})$ for a given couple of admissible controls $(\alpha,\xi)$ on $\nu$.
Then, from the same arguments as in \eqref{TauEtaEquivalence} we have
\begin{equation}\label{ZYequivalence}
\ind_{\{s\le \tau\wedge\rho_\cO\}}\Gamma_s(\rho_\cO,\tau,\alpha,\xi,X,Z)=\ind_{\{s\le \tau\wedge\rho_\cO\}}\Lambda_s(\rho_\cO,\eta^\tau,\alpha,\xi,X,Z),\quad\bP\text{-a.s.}
\end{equation} 	
and, setting $s=t$,
\begin{equation}\label{eq:JI}
J^\nu_{t,x,z}(\tau,\alpha,\xi)=\bE\big[\Gamma_t(\rho_\cO,\tau,\alpha,\xi,X^{\alpha,\xi},Z^{\xi})\big]=\bE\big[\Lambda_t(\rho_\cO,\eta^\tau,\alpha,\xi,X^{\alpha,\xi},Z^{\xi})\big]=I^\nu_{t,x,z}(\eta^\tau,\alpha,\xi).
\end{equation}
Clearly, removing the state variable $Z$ in the definitions of $\Gamma$ and $\Lambda$ we obtain analogous expressions for the infinite-fuel case. 	

\begin{remark}
In the proofs of this section we often use the expression ``\:$\bP$-a.e.\ $\omega$'' to indicate that a certain property $\Pi_i$ holds on a set $\Omega_i\in\cF$ with $\bP(\Omega_i)=1$. Clearly, the nature of the set $\Omega_i$ depends, in general, on the property $\Pi_i$ of interest. Since in our proofs we only consider a finite number of properties $\Pi_1,\Pi_2,\ldots,\Pi_n$ then the expression ``\:$\bP$-a.e.\ $\omega$'' refers to a universal set $\Omega':=\cap_{i=1}^n \Omega_i\in\cF$ with $\bP(\Omega')=1$ and such that all properties $\Pi_1,\Pi_2,\ldots,\Pi_n$ hold for all $\omega\in\Omega'$.
	\end{remark}

Here we prove Theorem \ref{thm:DPPdetff}. The proof of Theorem \ref{thm:DPPdet} is similar but easier as it involves one fewer state variable, so we omit it in order to avoid repetitions.

\begin{proof}[{\bf Proof of Theorem \ref{thm:DPPdetff}}] 
Recall that $(t,x,z)\in[0,T]\times\bR^d\times[0,\bar{z}]$, $u\in[t,T]$ and $\nu\in\cV_t$ are fixed. The proof is split into two steps.
		
\textit{Step 1}. (\textit{inequality} $\leq$). Given an admissible treble $(\tau,\alpha,\xi)\in\AA^\nu_{t,\bar z-z}$, set $X=X^{\nu;\alpha,\xi}$, $Z=Z^{\nu;\xi}$ and $\rho_\cO=\rho^{\nu;\alpha,\xi}_\cO$. By definition of $\Gamma$, we have
\begin{align}\label{Jdecomposed}
J_{t,x,z}^\nu(\tau,\alpha,\xi)=\bE\bigg[&\int_t^{u\wedge\tau\wedge\rho_\cO}f(s,X_s,Z_s,\alpha_s)\ud s-\int_{[t,u\wedge\tau\wedge\rho_\cO)}\langle c_+(s,X_s),\ud\xi^+_s\rangle\nonumber\\
&-\int_{[t,u\wedge\tau\wedge\rho_\cO)}\langle c_-(s,X_s),\ud\xi^-_s\rangle+g_1(\rho_\cO,X_{\rho_\cO},Z_{\rho_\cO})\ind_{\{\rho_\cO\leq\tau \}}\ind_{\{\tau\wedge\rho_\cO<u \}}\\
&+g_2(\tau,X_\tau,X_{\tau})\ind_{\{\rho_\cO>\tau \}}\ind_{\{\tau\wedge\rho_\cO<u \}}+\Gamma_u(\rho_\cO,\tau,\alpha,\xi,X,Z)\ind_{\{\tau\wedge\rho_\cO\geq u \}} \bigg].\nonumber
\end{align}
It remains to show that
\begin{equation}\label{eq:Gv}
\bE\Big[\Gamma_u(\rho_\cO,\tau,\alpha,\xi,X,Z)\ind_{\{\tau\wedge\rho_\cO\geq u \}}\Big]\leq \bE\Big[v(u,X_u,Z_u)\ind_{\{\tau\wedge\rho_\cO\geq u \}}\Big].
\end{equation}

Let us rewrite the left-hand side as 
\begin{align}\label{Part1Chain1}
\bE\Big[\Gamma_u(\rho_\cO,\tau,\alpha,\xi,X,Z)\ind_{\{\tau\wedge\rho_\cO\geq u \}}\Big]&=\bE\Big[\Lambda_u(\rho_\cO,\eta^\tau,\alpha,\xi,X,Z)\ind_{\{\tau\geq u \}}\ind_{\{\rho_\cO\geq u\}} \Big]\\
&=\bE\Big[\Lambda_u(\rho_\cO,\eta^\tau,\alpha,\xi,X,Z)(1-\eta^\tau_u)\ind_{\{\rho_\cO\geq u\}} \Big].\notag
\end{align}
By Lemma \ref{Lemma:AlphaAlpha0} and \ref{Lemma:Predict0Indisting}, we obtain a $\{\cF^{t,0}_s \}$-predictable control treble $(\eta^{\tau,0}\alpha^0,\xi^0)$ such that $\alpha=\alpha^0$, $\ud s\times\bP$-a.e.\ and $(\eta^{\tau,0},\xi^0)$ is indistinguishable from $(\eta^\tau,\xi)$. Then, by Lemma \ref{lemma:Xind}, $X=X^{\alpha,\xi}$ and $X^{\alpha^0,\xi^0}$ are indistinguishable. Moreover, $Z^\xi$ is indistinguishable from $Z^{\xi^0}=:Z^0$ and the latter is $\{\cF^{t,0}_s \}$-predictable as well. Since $X^{\alpha^0,\xi^0}$ is also left-continuous, Lemma \ref{Lemma:Predict0Indisting} guarantees that there exists a $\{\cF^{t,0}_s \}$-predictable process $X^0$ which is indistinguishable of $X^{\alpha^0,\xi^0}$. In summary, we have a pair of processes $(X^0,Z^0)$ that are $\{\cF^{t,0}_s \}$-predictable and indistinguishable from the original controlled pair $(X,Z)$. Thus, setting 
\begin{align}\label{eq:rho0}
\rho^0_\cO:=\inf\{s\geq t: (X^0_s,Z^0_s)\notin \cO \}\wedge T,
\end{align}
we also have $\bP(\rho_\cO=\rho^0_{\cO})=1$. 

For $s\in[u,T]$, let us denote $\eta^{0,u}_s=\ind_{\{s>u\}}\eta^{\tau,0}_s$ and $\xi^{0,u}_s=\xi^0_s-\xi^0_u$.
With this notation, we  obtain
\begin{align}\label{Part1Chain2}
\bE\Big[\Lambda_u(\rho_\cO,\eta^\tau,\alpha,\xi,X,Z)(1-\eta^\tau_u)\ind_{\{\rho_\cO\geq u\}} \Big]&=\bE\Big[\Lambda_u(\rho^0_\cO,\eta^{\tau,0},\alpha^0,\xi^0,X^0,Z^0)(1-\eta^{\tau,0}_u)\ind_{\{\rho^0_\cO\geq u\}} \Big]\nonumber\\
&=\bE\Big[\Lambda_u(\rho^0_\cO,\eta^{0,u},\alpha^0,\xi^{0,u},X^0,Z^0)(1-\eta^{\tau,0}_u)\ind_{\{\rho^0_\cO\geq u\}} \Big],
\end{align}
where the first equality is by equivalence of the processes under expectations and the second equality uses that $\Lambda_u$ only depends on the increments of $\xi$ after time $u$ and that $\eta^{\tau,0}_s=\eta^{0,u}_s$ for every $s\in[t,T]$ on the event $\{\tau\ge u\}$.
		
By indistinguishability of $X^{\alpha^0,\xi^0}$ and $X^0$ we have $\bP$-a.s., for all $s\in[u,T]$,
\begin{align}\label{eq:SDE0}
X^0_s&=X^{\alpha^0,\xi^0}_s\nonumber\\
&=X^{\alpha^0,\xi^0}_u+\int_u^s\mu(r,X^{\alpha^0,\xi^0}_r,\alpha^0_r)\ud r+\int_u^s\sigma(r,X^{\alpha^0,\xi^0}_r,\alpha^0_r)\ud W_r+\xi^0_s-\xi^0_u\\
&=X^0_u+\int_u^s\mu(r,X^0_r,\alpha^0_r)\ud r+\int_u^s\sigma(r,X^0_r,\alpha^0_r)\ud W^u_r+\xi^{0,u}_s.\notag
\end{align}
Since $(X^0_u,Z^0_u)$ is $\cF^{t,0}_u$-measurable, by a well-known property of regular conditional probabilities (see \eqref{ConstRegConditProb} in the Appendix), we have 
\[
\bP_\omega\Big(\big\{\omega'\in\Omega: \big(X^{0}_u(\omega'),Z^0_u(\omega')\big)=\big(X^{0}_u(\omega),Z^0_u(\omega)\big)\big\}\Big)=1,\quad\text{for $\bP$-a.e.\ $\omega$}.
\] 
It is also recalled in Lemma \ref{Lemma:AlmostSureSets} that if $\bP(\Omega_1)=1$ for some $\Omega_1\in\cF$, then $\Omega_1\in\cF_\omega$ and $\bP_\omega(\Omega_1)=1$ for a.e.\ $\omega\in \Omega$. Thus, taking $\Omega_1$ as the set where the SDE \eqref{eq:SDE0} holds, we have that $\bP_\omega$-a.s.,
for $\bP$-a.e.\ $\omega\in\Omega$
\begin{equation}\label{X^0SDE}
\begin{split}
&X^{0}_s = X^{0}_u(\omega)\!+\!\int_u^s\!\mu(r,X^0_r,\alpha^0_r)\ud r\!+\!\int_u^s\!\sigma(r,X^0_r,\alpha^0_r)\ud W^u_r\!+\!\xi^{0,u}_s,\\
&Z^0_s=Z^0_u(\omega)+V_{[u,T]}(\xi^{0,u}),
\end{split}
\end{equation}
for every $s\in[u,T]$. Likewise $\bP_\omega(\rho_\cO=\rho_\cO^0)=1$ for $\bP$-a.e.\ $\omega\in\Omega$. There are two subtle points related to \eqref{X^0SDE}: (i) a first one (raised in \cite{claisse2016pseudo}) is that the stochastic integral in \eqref{X^0SDE} is constructed with respect to the regular conditional probability $\bP_\omega$ and it is $\bP$-indistinguishable from the original one (see Lemma \ref{Lemma:StochIntRegCondProb} for details); (ii) a second one is that it is possible to check that the treble $(\eta^{0,u},\alpha^0,\xi^{0,u})$ belongs to the admissible class $\AA^{\nu_\omega}_{u,\bar z-Z^0(\omega)}$ for $\bP$-a.e.\ $\omega\in\Omega$ (see Proposition \ref{Prop:StdRefProbSyst}). Finally, $X^0$ is $\{\cF^{t,0}_s \}$-adapted and therefore it is $\{\cF^{u,\omega}_{s} \}$-adapted for $\bP$-a.e.~$\omega$ (see Lemma \ref{Lemma:FiltrationsInclusion}).

In conclusion, up to $\bP_\omega$-indistinguishability, the process $(X^0,Z^0)$ is the unique solution of \eqref{X^0SDE} in the reference probability system $\nu_\omega$, for $\bP$-a.e.\ $\omega\in\Omega$. Consistently with the notation introduced in Section \ref{sec:problem} around the SDE \eqref{StateProcess} and the process $Z$ in \eqref{eq:Z}, we should say that for $\bP$-a.e.\ $\omega$ the process $(X^0_s,Z^0_s)_{s\in[u,T]}$ is indistinguishable from the pair $X^{u,X^{0}_u(\omega);\nu_\omega;\alpha^0,\xi^{0,u}}$ and $Z^{u,Z^{0}_u(\omega);\nu_\omega;\xi^{0,u}}$ but we avoid such heavy notation as no confusion shall arise.

Notice that, on the event $\{\rho^0_\cO\ge u\}$, we have  
\[
\rho^{0}_\cO=\inf\{s\ge u: (X^0_s,Z^0_s)\notin \cO\}\wedge T,
\]
so that $\rho^0_\cO$ is equal to the first time {\em after} time $u$ when the process $(X^0_s,Z^0_s)_{s\in[u,T]}$ leaves the set $\cO$. Thus, $\rho^0_\cO$ defines a stopping time in the reference probability system $\nu_\omega$. This fact will be used in the next group of equations, without further mention.

Then, continuing from \eqref{Part1Chain2}, by tower property we obtain
\begin{align}\label{Part1Chain3}
&\bE\Big[\Lambda_u(\rho^0_\cO,\eta^{0,u},\alpha^0,\xi^{0,u},X^0,Z^0)(1-\eta^{\tau,0}_u)\ind_{\{\rho^0_\cO\geq u\}} \Big]\nonumber\\
&=\bE\Big[\bE\Big[\Lambda_u(\rho^0_\cO,\eta^{0,u},\alpha^0,\xi^{0,u},X^0,Z^0)\Big|\cF^{t,0}_u \Big](\omega)(1-\eta^{\tau,0}_u(\omega))\ind_{\{\rho^0_\cO\geq u\}}(\omega)\Big]\nonumber\\
&=\bE\Big[\bE_\omega\Big[\Lambda_u(\rho^0_\cO,\eta^{0,u},\alpha^0,\xi^{0,u},X^0,Z^0) \Big](1-\eta^{\tau,0}_u(\omega))\ind_{\{\rho^0_\cO\geq u\}}(\omega)\Big]\\
&=\bE\Big[I^{\nu_\omega}_{u,X^0_u(\omega),Z^0_u(\omega)}(\eta^{0,u},\alpha^0,\xi^{0,u})(1-\eta^{\tau,0}_u(\omega))\ind_{\{\rho^0_\cO\geq u\}}(\omega)\Big]\nonumber\\
&\leq\bE\Big[v(u,X^0_u(\omega),Z^0_u(\omega))(1-\eta^{\tau,0}_u(\omega))\ind_{\{\rho^0_\cO\geq u\}}(\omega)\Big]=\bE\Big[v(u,X_u,Z_u)\ind_{\{\tau\wedge\rho\geq u\}}\Big],\notag
\end{align}
where the second equality is by property \eqref{RegConditProbProp} of the measure $\bP_\omega$ (and Assumption \ref{ass:obj0}), the third one uses the definition of 
$I^{\nu_\omega}$ (see \eqref{TauEtaEquivalence}) and the fact that the processes $(\eta^{0,u},\alpha^0,\xi^{0,u},X^0,Z^0)$ are well-defined in the reference probability system $\nu_\omega$ for $\bP$-a.e.\ $\omega$ (see Proposition \ref{Prop:StdRefProbSyst}); the inequality follows from admissibility of the treble $(\eta^{0,u},\alpha^0,\xi^{0,u})\in\AA^{\nu_\omega}_{u,\bar z-Z^0_u(\omega)}$ for $\bP$-a.e.~$\omega$ and Proposition \ref{Prop:IndepProbSyst}; in the final expression we recall that $(X^0_u,Z^0_u,\eta^{\tau,0},\rho^0_\cO)$ and $(X_u,Z_u,\eta^{\tau},\rho_\cO)$ are equal $\bP$-a.s. 

Combining \eqref{Part1Chain1}, \eqref{Part1Chain2} and \eqref{Part1Chain3}, we obtain \eqref{eq:Gv} which, together with \eqref{Jdecomposed}, gives
\begin{align*}
&J_{t,x,z}^\nu(\tau,\alpha,\xi)\\
&\leq\bE\bigg[\int_t^{u\wedge\tau\wedge\rho_\cO}\!f(s,X_s,Z_s,\alpha_s)\ud s\!-\!\int_{[t,u\wedge\tau\wedge\rho_\cO)}\!\langle c_+(s,X_s),\ud\xi^+_s\rangle\!-\!\int_{[t,u\wedge\tau\wedge\rho_\cO)}\!\langle c_-(s,X_s),\ud\xi^-_s\rangle\nonumber\\
&\quad\qquad +g_1(\rho_\cO,X_{\rho_\cO},Z_{\rho_\cO})\ind_{\{\rho_\cO\leq\tau\}\cap\{ \rho_\cO<u \}}+g_2(\tau,X_\tau,Z_{\tau})\ind_{\{\tau< u\wedge\rho_\cO \}}+v(u,X_u,Z_u)\ind_{\{\tau\wedge\rho_\cO\geq u \}} \bigg].
\end{align*}
Taking the supremum over all admissible trebles we obtain the first inequality
\begin{align*}
v(t,x,z)\leq\sup_{(\tau,\alpha,\xi)\in\AA^\nu_{t,\bar z-z}} \bE\bigg[&\int_t^{u\wedge\tau\wedge\rho_\cO} f(s,X_s,Z_s,\alpha_s)\ud s-\int_{[t,u\wedge\tau\wedge\rho_\cO)}\langle c_+(s,X_s),\ud\xi^+_s\rangle\nonumber \\
&- \int_{[t,u\wedge\tau\wedge\rho_\cO)}\langle c_-(s,X_s),\ud\xi^-_s\rangle+g_1(\rho_\cO,X_{\rho_\cO},Z_{\rho_\cO})\ind_{\{\rho_\cO\leq\tau\}\cap\{ \rho_\cO<u \}}\nonumber\\
&+g_2(\tau,X_\tau,Z_\tau)\ind_{\{\tau< u\wedge\rho_\cO \}}+v(u,X_u,Z_u)\ind_{\{\tau\wedge\rho_\cO\geq u\}} \bigg].
\end{align*}
		
\textit{Step 2}. (\textit{inequality} $\geq$). At the start of the proof we fixed $(t,x,z)$, $u\in[t,T]$ and $\nu\in\cV_t$. Let now  $(\tau,\alpha,\xi)\in\AA^{\nu}_{t,\bar z-z}$ be arbitrary but fixed too. The idea is to construct another admissible treble $(\hat{\tau},\hat{\alpha},\hat{\xi})\in\AA^\nu_{t,\bar z-z}$ which coincides with $(\tau,\alpha,\xi)$ up to time $u$ and whose restriction to $[u,T]$ is $\delta$-optimal with respect to $v(u,X^{\alpha,\xi}_u,Z^\xi_u)$ for some $\delta>0$.

Let us denote $B^k=\{x\in\bR^d\, :\, |x|\le k,\,z\in[0,\bar z]\}$. By Assumption \ref{ass:obj} and \eqref{eq:vc2}, for any $\eps>0$ we can pick $\delta_1>0$ such that if $(x_1,z_1),(x_2,z_2)\in B^1$, with $z_1\le z_2$ and $|(x_1,z_1)-(x_2,z_2)|<\delta_1$, then
\begin{equation}\label{Iterating1}
|J^{\nu'}_{u,x_1,z_1}(\tau',\alpha',\xi')-J^{\nu'}_{u,x_2,z_2}(\tau',\alpha',[\xi']^{z_2})|+|v(u,x_1,z_1)-v(u,x_2,z_2)|<\eps,
\end{equation}
for all $(\tau',\alpha',\xi')\in\AA^{\nu'}_{u,\bar z-z_1}$ and any $\nu'\in\cV_u$.
Since $\bR^d\times[0,\bar z]$ is separable, we can choose a partition $\{B^1_n \}_{n\in\bN}$ of $B^1$ into countable disjoint Borel sets with $\mathrm{diam}(B^1_n)<\delta_1$ for every $n\in\bN$, where $\mathrm{diam}(B):=\sup \{|(x,z)-(x',z')|: (x,z),(x',z')\in B\}$. Then, \eqref{Iterating1} holds for any $(x_1,z_1),(x_2,z_2)\in B^1_n$ with $z_1\le z_2$ and $n\in\bN$. Similarly, we can choose $\delta_2>0$ such that if $(x_1,z_1),(x_2,z_2)\in B^2$, $z_1\le z_2$ and $|(x_1,z_1)-(x_2,z_2)|<\delta_2$, then \eqref{Iterating1} holds too. Thus, we can also choose a partition $\{B^2_n \}_{n\in\bN}$ of $B^2\setminus B^1$ into countable disjoint Borel sets with $\mathrm{diam}(B^2_n)<\delta_2$ for every $n\in\bN$ so that \eqref{Iterating1} holds for all $(x_1,z_1),(x_2,z_2)\in B^2_n$. Iterating this argument, we find a partition $\{B^k_n \}_{k,n\in\bN}$ of $\bR^d\times[0,\bar z]$ into countable disjoint Borel sets such that for arbitrary $n,k\in\bN$, given $(x_1,z_1),(x_2,z_2)\in B^k_n$ with $z_1\le z_2$, the condition \eqref{Iterating1} holds. With a slight abuse of notation we relabel the partition $\{ B^k_n \}_{k,n\in\bN}$ simply by $\{B_n \}_{n\in\bN}$.

Let us fix a reference probability system $\bar{\nu}=(\bar{\Omega},\bar{\cF},\bar{\bP},\{\bar{\cF}^u_{s}\},\bar{W})\in\cV_u$, which will be used as an auxiliary system when dealing with regular conditional probabilities.
Let $n,m\in\bN$ and take $(x_n,z_n)\in B_n$. For convenience, we pick $z_n$ so that $z_n\ge z$ for all $(x,z)\in B_n$. That is, with no loss of generality, the partition of $[0,\bar{z}]$ is of the form $\{\{0\},(0,a],(a,b],\ldots \}$. By Proposition \ref{Prop:IndepProbSyst}, there exist $(\tau^{n;(m)},\alpha^{n;(m)},\xi^{n;(m)})\in\AA^{\bar{\nu}}_{u,\bar z-z_n}$ such that
\begin{equation}\label{1/mOptimalControls}
J^{\bar{\nu}}_{u,x_n,z_n}(\tau^{n;(m)},\alpha^{n;(m)},\xi^{n;(m)})\geq v(u,x_n,z_n)-\tfrac{1}{m}.
\end{equation}
For the moment $m$ is fixed and we drop the superscript for notational convenience. Hence, we write $(\tau^{n},\alpha^{n},\xi^{n})$ instead of $(\tau^{n;(m)},\alpha^{n;(m)},\xi^{n;(m)})$. 

For future reference we set $\eta^n_s=\eta^{\tau^n}_s=\ind_{\{s>\tau^n\}}\in\cE^{\bar \nu}_u$ and we recall that there exists $\{\bar \cF^{u,0}_s\}$-predictable representatives $(\eta^{n,0},\alpha^{n,0},\xi^{n,0})$ such that $\alpha^n=\alpha^{n,0}$, $\ud s\times \bar\bP$-a.s.\ and $(\eta^{n,0},\xi^{n,0})$ is $\bar \bP$-indistinguishable from $(\eta^{n},\xi^{n})$. Then, by Lemma \ref{Lemma:CanonicalControl}, we have $\cP_{\Omega^*}$-measurable maps $\psi_{\alpha^n}$, $\psi^\pm_{\xi^n}$ and $\psi_{\eta^n}$ such that 
\[
\alpha^{n,0}_s(\omega)=\psi_{\alpha^n}\big(s, \bar W_\cdot(\omega)\big),\quad \xi^{n,0;\pm}_s(\omega)=\psi^\pm_{\xi^n}\big(s, \bar W_\cdot(\omega)\big),\quad\eta^{n,0}_s(\omega)=\psi_{\eta^n}\big(s, \bar W_\cdot(\omega)\big),\quad s\in[u,T].
\]
These will be used later. It is clear that by equivalence of the controls we have (recall \eqref{EtaTauEquivalence})
\begin{equation}\label{Ichain1}
I^{\bar{\nu}}_{u,x_n,z_n}(\eta^n,\alpha^n,\xi^n)=I^{\bar{\nu}}_{u,x_n,z_n}(\eta^{n,0},\alpha^{n,0},\xi^{n,0}).
\end{equation}

As in Lemma \ref{Lemma:ChangeOfControls} we construct a treble of $\{\cF^{u,0}_s\}$-predictable, admissible controls $(\tilde\eta^n,\tilde\alpha^n,\tilde\xi^n)\in\AA^{\nu_\omega}_{u,\bar z-z_n}$ such that
\begin{equation}\label{TildeXi^n=Xi^n}
\cL_{\bar{\bP}}(\eta^{n,0},\alpha^{n,0},\xi^{n,0},\bar{W})=\cL_{\bP_\omega}(\tilde{\eta}^n,\tilde{\alpha}^n,\tilde{\xi}^{n},W^u), \quad \bP\text{-a.e.~}\omega.
\end{equation}
It is worth noticing that effectively the treble $(\tilde\eta^n,\tilde\alpha^n,\tilde\xi^n)$ is independent of the specific $\omega\in\Omega$ that determines the reference probability system $\nu_\omega$. That happens because $W^u_\cdot=W_\cdot-W_u$ is a Brownian motion on every $\nu_\omega$, for $\bP$-a.e.~$\omega$, and the mappings $\psi_{\alpha^n}$, $\psi_{\eta^n}$ and $\psi_{\xi^n}$ do not depend on $\omega$. Moreover, the set $\tilde{\Omega}^n\subseteq \Omega$ where $\tilde\eta^n$ and $\tilde{\xi}^n$ satisfy all the admissibility conditions (left-continuity, bounded variation, $(\tilde\eta^n_u,\tilde\xi^n_u)=(0,0)$) belongs, by construction, to $\cF^0$ and $\bP_\omega(\tilde{\Omega}^n)=1$ for $\bP$-a.e.~$\omega$ (see the proof of Lemma \ref{Lemma:CanonicalControl}). Therefore, we also have $\bP(\tilde{\Omega}^n)=1$ (see Lemma \ref{Lemma:AlmostSureSets}) and setting $\tilde{\Omega}:=\cap_n \tilde{\Omega}^n\in\cF^0$ we conclude that the admissibility conditions of $(\tilde \eta^n,\tilde \xi^n)$ are satisfied for every $n\in\bN$ for all $\omega\in\tilde{\Omega}$ with $\bP(\tilde{\Omega})=1$.

By \eqref{TildeXi^n=Xi^n} and Lemma \ref{Coroll:EquivInLaw}, we have 
\begin{align*}
\cL_{\bar{\bP}}(&\tau^{n,0},\alpha^{n,0},\xi^{n,0},X^{u,x_n;\bar \nu;\alpha^{n,0},\xi^{n,0}},Z^{u,z_n;\bar \nu;\xi^{n,0}},\rho_\cO^{u,x_n,z_n;\bar \nu;\alpha^{n,0},\xi^{n,0}})\\
&=\cL_{\bP_{\omega}}(\tilde{\tau}^n,\tilde{\alpha}^n,\tilde{\xi}^{n},X^{u,x_n;\nu_\omega;\tilde{\alpha}^n,\tilde{\xi}^n},Z^{u,z_n;\nu_\omega;\tilde \xi^n},\rho_\cO^{u,x_n,z_n;\nu_\omega;\tilde{\alpha}^n,\tilde{\xi}^n}), \quad \text{for $\bP$-a.e.\ $\omega\in\Omega$},
\end{align*}
with $\tau^{n,0}=\inf\{s\ge u:\eta^{n,0}_s=1\}\wedge T$ and $\tilde \tau^{n}=\inf\{s\ge u:\tilde\eta^n_s=1\}\wedge T$.
Thanks to the above and \eqref{Ichain1}, for each $n\in\bN$ and $\bP$-a.e.\ $\omega$, we have 
\begin{align}\label{Omega_0equalityInLaw}
J^{\bar \nu}_{u,x_n,z_n}(\tau^n,\alpha^n,\xi^n)= I^{\bar{\nu}}_{u,x_n,z_n}(\eta^{n},\alpha^{n},\xi^{n})=I^{\nu_\omega}_{u,x_n,z_n}(\tilde{\eta}^n,\tilde{\alpha}^n,\tilde{\xi}^n)=J^{\nu_\omega}_{u,x_n,z_n}(\tilde\tau^n,\tilde\alpha^n,\tilde\xi^n),
\end{align}
which we will need later on to conclude our proof. Notice that the $\bP$-null set where the above fails may depend on $(x_n,z_n)$.

Recall the initial admissible treble $(\tau,\alpha,\xi)\in\AA^\nu_{t,\bar z-z}$ and denote by $(\eta^{\tau,0},\alpha^0,\xi^0)$ its $\{\cF^{t,0}_s\}$-predictable representative. Set $(X,Z)=(X^{\nu;\alpha^0,\xi^0},Z^{\nu;\xi^0})$ and denote by $(X^0,Z^0)$ the $\{\cF^{t,0}_s\}$-predictable processes indistinguishable of $(X,Z)$. Define the events $O_n:=\{\omega:(X^0_u,Z^0_u)\in B_n\}$. Since $\{B_n\}_{n\in\bN}$ forms a partition of $\bR^d\times[0,\bar z]$, the events $\{O_n\}_{n\in\bN}$ form a partition of $\Omega$. Then, letting $U_k:=\cup_{n=1}^k O_n$ for $k\in\bN$ we have $U_k\uparrow \Omega$ as $k\to\infty$. 

For every $k\in\bN$ and $s\in[t,T]$, define
\begin{align}
\hat{\eta}^k_s&:=\eta^{\tau,0}_s\ind_{[t,u]}(s)+\Big[\eta^{\tau,0}_u+(1-\eta^{\tau,0}_u)\Big(\sum_{n=1}^k\tilde{\eta}^n_s\ind_{O_n}+\ind_{U_k^c}\Big)\Big]\ind_{(u,T]}(s), \label{HatEta}\\
\hat{\alpha}^k_s&:=\alpha^0_s\ind_{[t,u]}(s)+\bigg[\sum_{n=1}^k\tilde{\alpha}^n_s\ind_{O_n}+\alpha^0_s\ind_{U_k^c}\bigg]\ind_{(u,T]}(s), \label{HatAlpha}
\end{align}
\begin{align}
\hat{\xi}^{k,\pm}_s&:=\xi^{\pm,0}_s\ind_{[t,u]}(s)+\bigg[\sum_{n=1}^k\Big(\tilde{\xi}^{n,\pm}_s+\xi^{\pm,0}_u\Big)\ind_{O_n}+\xi^{\pm,0}_s\ind_{U_k^c}\bigg]\ind_{(u,T]}(s),
\end{align}
where notice that by construction $\tilde{\xi}^{n,\pm}_{\,\cdot}\in\cX^{\nu_\omega}_u(\bar z-z_n)\subseteq\cX^{\nu_\omega}_u(\bar z-Z^0_u(\omega))$ for $\bP$-a.e.\ $\omega\in O_n$ because $Z^{0}_u(\omega)\le z_n$ for $\bP$-a.e.\ $\omega\in O_n$ (recall that $z_n$ was chosen so that $z_n\ge z$ for all $(x,z)\in B_n$). Finally, we set
\begin{align}
\hat{\xi}^k&:=\hat{\xi}^{k,+}-\hat{\xi}^{k,-},\label{Zeta}
\end{align}
and notice also that $\hat \xi^k\in\cX^\nu_t(\bar z-z)$. Then, defining $\hat \tau^k:=\inf\{s\ge t: \hat \eta^k_s=1\}\wedge T$ (cf.\ \eqref{eq:etatau}) we have $(\hat \tau^k,\hat \alpha^k,\hat \xi^k)\in\AA^\nu_{t,\bar z -z}$ and, in particular, $(\hat \eta^k,\hat \alpha^k, \hat \xi^k)$ are $\{\cF^{t,0}_s\}$-adapted by construction.

Clearly,
\[
O_n\subseteq\big\{\omega\in\Omega:(\hat \alpha^k_s,\hat \xi^{k,\pm}_s)=(\tilde \alpha^n_s,\xi^{\pm,0}_u+\tilde{\xi}^{n,\pm}_s),\:\:s\in(u,T]\big\}
\]
and, moreover, we have
\begin{equation}\label{eq:POn}
\bP_\omega(O_n)=\bE_\omega[\ind_{O_n}]=\bE[\ind_{O_n}|\cF^{t,0}_u](\omega)=\bE[\ind_{O_n}|\cF^t_u](\omega)=\ind_{O_n}(\omega)=1, \quad \bP\text{-a.e.~}\omega\in O_n,
\end{equation}
where the third equality is by Lemma \ref{Lemma:ExpectRegConditProb}. The equation above confirms the intuitively obvious fact that, $\bP_\omega$-a.s.\ for $\bP$-a.e.\ $\omega\in O_n$, $(\hat \alpha^k_s,\hat \xi^{k,\pm}_s)=(\tilde \alpha^n_s,\xi^\pm_u+\tilde{\xi}^{n,\pm}_s)$ for all $s\in(u,T]$.

Let us denote by $X^{k,0}$ the $\{\cF^{t,0}_s\}$-predictable process which is $\bP$-indistinguishable from the controlled dynamics $X^{\nu;\hat{\alpha}^k,\hat{\xi}^k}$. As in \eqref{X^0SDE}, we have $\bP_\omega$-a.s.~for $\bP$-a.e.\ $\omega\in\Omega$
\begin{align*}
X^{k,0}_s=X^{k,0}_u(\omega)+\int_u^s\mu(r,X^{k,0}_r,\hat \alpha^k_r)\ud r+\int_u^s\sigma(r,X^{k,0}_r,\hat \alpha^{k}_r)\ud W^u_r+\hat \xi^{k}_s-\hat \xi^{k}_u,\:\:\text{for $s\in[u,T]$.}\notag
\end{align*}
Again, we remark that $X^{k,0}$ defines a process on $\nu_\omega$ that is $\bP_\omega$-indistinguishable from the solution of the controlled SDE on $\nu_\omega$ denoted by $X^{u,X^{k,0}(\omega);\nu_\omega;\hat \alpha^{k},\hat \xi^k}$. We avoid further notation and simply identify the two processes on $\nu_\omega$. Likewise, let 
$Z^{k,0}_s=z+V_{[t,s]}(\hat \xi^k)$ for $s\in[t,T]$ so that $Z^{k,0}_s=Z^0_u(\omega)+V_{[u,s]}(\hat \xi^k)$, $\bP_\omega$-a.s.\ for $\bP$-a.e.\ $\omega$ for every $s\in[u,T]$.

Recall the exit time $\rho_\cO=\rho^{\nu;\alpha^0,\xi^0}_\cO=\inf\{s\ge t: (X_s,Z_s)\notin\cO\}\wedge T$ of the process $(X,Z)$ from $\cO$ after time $t$ (or equivalently for the process $(X^0,Z^0)$) and let $\rho^k_\cO=\rho^{\nu;\hat \alpha^k,\hat \xi^k}_\cO$ be the analogous exit time for the process $(X^{k,0},Z^{k,0})$. Since
\[
\bP\big((\eta^{\tau,0}_s,\alpha^0_s,\xi^0_s)=(\hat \eta^{k}_s,\hat \alpha^k_s,\hat \xi^k_s),\: s\in[t,u]\big)=1,
\]
then
\begin{align}\label{eq:X0Xk}
\begin{aligned}
&\bP\big((X_s,Z_s)=(X^{0}_s,Z^{0}_s)=(X^{k,0}_s,Z^{k,0}_s),\:s\in[t,u]\big)=1,\\
&\bP(\rho_\cO\wedge u=\rho^k_\cO\wedge u)=1\quad\text{and}\quad \bP(\tau\wedge u=\hat\tau^k\wedge u)=1.
\end{aligned}
\end{align}
Moreover, on the event $\{\rho_\cO\ge u\}$, we have 
\[
\rho^k_\cO=\inf\big\{s\ge u\,:\, (X^{k,0}_s,Z^{k,0}_s)\notin\cO\big\}\wedge T,
\]
i.e., the exit time is {\em after} time $u$,
and $\rho^k_\cO$ is a stopping time in the reference probability system $\nu_\omega$.

Since $(\hat{\tau}^k,\hat{\alpha}^k,\hat{\xi}^k)\in\AA^\nu_{t,\bar z-z}$ we have
\begin{align*}
v(t,x,z)
&\ge \bE\bigg[\int_t^{u\wedge\hat{\tau}^k\wedge\rho^k_\cO}f(s,X^{k,0}_s,Z^{k,0}_s,\hat{\alpha}^k_s)\ud s-\int_{[t,u\wedge\hat{\tau}^k\wedge\rho^k_\cO)}\langle c_+(s,X^{k,0}_s),\ud\hat{\xi}^{k,+}_s\rangle\nonumber\\
&\hspace{30pt}-\int_{[t,u\wedge\hat{\tau}^k\wedge\rho^k_\cO)}\langle c_-(s,X^{k,0}_s),\ud\hat{\xi}^{k,-}_s\rangle+g_1\big(\rho^k_\cO,X^{k,0}_{\rho^k_\cO},Z^{k,0}_{\rho^k_\cO}\big)\ind_{\{\rho^k_\cO\leq\hat{\tau}^k \}}\ind_{\{\hat{\tau}^k\wedge\rho^k_\cO<u \}}\\
&\hspace{30pt}+g_2\big(\hat{\tau}^k,X^{k,0}_{\hat{\tau}^k}, Z^{k,0}_{\hat \tau^k}\big)\ind_{\{\rho^k_\cO>\hat{\tau}^k \}}\ind_{\{\hat{\tau}^k\wedge\rho^k_\cO<u \}}+\Gamma_u\big(\rho^k_\cO,\hat{\tau}^k,\hat{\alpha}^k,\hat{\xi}^k,X^{k,0},Z^{k,0}\big)\ind_{\{\hat{\tau}^k\wedge\rho^k_\cO\geq u \}} \bigg].\notag
\end{align*}
Thanks to \eqref{eq:X0Xk}, the expression simplifies to
\begin{align}\label{vGeqRunning+Terminal}
v(t,x,z)&\geq\bE\bigg[\int_t^{u\wedge\tau\wedge\rho_\cO}f(s,X_s,Z_s,\alpha_s)\ud s-\int_{[t,u\wedge\tau\wedge\rho_\cO)}\langle c_+(s,X_s),\ud\xi^{+}_s\rangle\nonumber\\
&\hspace{30pt}-\int_{[t,u\wedge\tau\wedge\rho_\cO)}\langle c_-(s,X_s),\ud\xi^{-}_s\rangle+g_1\big(\rho_\cO,X_{\rho_\cO},Z_{\rho_\cO}\big)\ind_{\{\rho_\cO\leq\tau \}\cap\{\rho_\cO<u \}}\\
&\hspace{30pt}+g_2\big(\tau,X_{\tau},Z_\tau\big)\ind_{\{\tau<\rho_\cO\wedge u \}}+\Gamma_u\big(\rho^k_\cO,\hat{\tau}^k,\hat{\alpha}^k,\hat{\xi}^k,X^{k,0},Z^{k,0}\big)\ind_{\{\tau\wedge\rho_\cO\geq u \}}\bigg],\notag 
\end{align}
where we also used that the equality between stopping times in \eqref{eq:X0Xk} implies $\{\hat \tau^k\wedge\rho^k_\cO\geq u\}=\{\tau\wedge\rho_\cO\geq u \}$.

Thus, recalling \eqref{ZYequivalence} and the event $U_k$, for the last term in \eqref{vGeqRunning+Terminal} we have
\begin{align}\label{Zsplit}
&\bE\Big[\Gamma_u(\rho^k_\cO,\hat{\tau}^k,\hat{\alpha}^k,\hat{\xi}^k,X^{k,0},Z^{k,0})\ind_{\{\tau\wedge\rho_\cO\geq u \}}\Big]\\
&=\bE\Big[\Gamma_u(\rho^k_\cO,\hat{\tau}^k,\hat{\alpha}^k,\hat{\xi}^k,X^{k,0},Z^{k,0})\ind_{\{\tau\wedge\rho_\cO\geq u \}}(\ind_{U_k}+\ind_{U^c_k})\Big]\notag\\
&=\bE\Big[\Lambda_u\big(\rho^k_\cO,\hat{\eta}^k,\hat{\alpha}^k,\hat{\xi}^k,X^{k,0},Z^{k,0}\big)\ind_{\{\tau\wedge\rho_\cO\geq u \}}(\ind_{U_k}+\ind_{U^c_k})\Big].\notag
\end{align}
On the event $U_k$ we partition over the sets $\{O_n\}_{n\le k}$ and use that on each $O_n$ the control treble $(\hat \eta^k_s,\hat \alpha^k_s,\hat \xi^k_s)_{s\in(u,T]}$ becomes $(\tilde \eta^n_s,\tilde \alpha^n_s,\xi^{0}_u+\tilde \xi^{n}_s)_{s\in(u,T]}$. Then, we obtain 
\begin{align*}
\bE\Big[&\Lambda_u(\rho^k_\cO,\hat{\eta}^k,\hat{\alpha}^k,\hat{\xi}^k,X^{k,0},Z^{k,0})\ind_{\{\tau\wedge\rho_\cO\geq u \}}\ind_{U_k}\Big]\\
&=\sum_{n=1}^k\bE\Big[\Lambda_u(\rho^k_\cO,\tilde{\eta}^n,\tilde{\alpha}^n,\xi^0_u+\tilde{\xi}^n,X^{k,0},Z^{k,0})\ind_{\{\tau\wedge\rho_\cO\geq u \}\cap O_n}\Big]\nonumber\\
&=\sum_{n=1}^k\bE\Big[\Lambda_u(\rho^k_\cO,\tilde{\eta}^n,\tilde{\alpha}^n,\tilde{\xi}^n,X^{k,0},Z^{k,0})\ind_{\{\tau\wedge\rho_\cO\geq u \}\cap O_n}\Big],\nonumber
\end{align*}
where in the final equality we use that the explicit dependence of the objective function $\Lambda_u$ on the control $\xi$ is only via its increments after time $u$ (while of course also depending implicitly on $\xi$ via the processes $(X^{k,0},Z^{k,0})$). Since $(\tilde{\eta}^n,\tilde{\alpha}^n,\tilde{\xi}^n,X^{k,0},Z^{k,0})$ defines a process on $\nu_\omega$, by the tower property and Assumption \ref{ass:obj0}, we arrive at  
\begin{align}\label{eq:DPP0}
&\bE\Big[\Gamma_u(\rho^k_\cO,\hat{\tau}^k,\hat{\alpha}^k,\hat{\xi}^k,X^{k,0},Z^{k,0})\ind_{\{\tau\wedge\rho_\cO\geq u \}}\ind_{U_k}\Big]\\
&=\sum_{n=1}^k\bE\Big[\bE\Big[\Lambda_u(\rho^k_\cO,\tilde{\eta}^n,\tilde{\alpha}^n,\tilde{\xi}^n,X^{k,0},Z^{k,0})\Big|\cF^t_u\Big](\omega)\ind_{\{\tau\wedge\rho_\cO\geq u \}\cap O_n}(\omega)\Big]\nonumber\\
&=\sum_{n=1}^k\bE\Big[\bE_\omega\Big[\Lambda_u(\rho^{k}_\cO,\tilde{\eta}^n,\tilde{\alpha}^n,\tilde{\xi}^n,X^{k,0},Z^{k,0})\Big]\ind_{\{\tau\wedge\rho_\cO\geq u \}\cap O_n}(\omega)\Big],\notag
\end{align}
because $\bE_\omega[Y]=\bE[Y|\cF^{t,0}_u](\omega)=\bE[Y|\cF^t_u](\omega)$ for $\cF$-measurable $Y$ with $\bE[(Y)^-]<\infty$ (Lemma \ref{Lemma:ExpectRegConditProb}).
For $\bP$-a.e.\ $\omega\in\{\tau\wedge\rho_\cO\geq u \}\cap O_n$ the process $(X^{k,0},Z^{k,0})$ is $\bP_\omega$-indistinguishable from $(X^{\omega,n},Z^{\omega,n})$ by \eqref{eq:POn}, where $X^{\omega,n}=X^{u,X^{0}_u(\omega);\nu_\omega;\tilde \alpha^n,\tilde\xi^n}$ and $Z^{\omega,n}=Z^{u,Z^0_u(\omega);\nu_\omega;\tilde \xi^n}$. Likewise, $\rho^k_\cO$ is $\bP_\omega$-a.s.\ equal to $\rho^{\omega,n}_\cO=\inf\{s\ge u: (X^{\omega,n}_s,Z^{\omega,n}_s)\notin\cO\}\wedge T$. Then, recalling \eqref{eq:JI}, we have for $\bP$-a.e.\ $\omega$
\begin{align}\label{eq:DPP1}
\bE_\omega&\Big[\Lambda_u(\rho^{k}_\cO,\tilde{\eta}^n,\tilde{\alpha}^n,\tilde{\xi}^n,X^{k,0},Z^{k,0})\Big]\ind_{\{\tau\wedge\rho_\cO\geq u \}\cap O_n}(\omega)\\
&=I^{\nu_\omega}_{u,X^{0}_u(\omega), Z^{0}_u(\omega)}\big(\tilde{\eta}^n,\tilde{\alpha}^n,\tilde{\xi}^n\big)\ind_{\{\tau\wedge\rho_\cO\geq u \}\cap O_n}(\omega)\nonumber\\
&=J^{\nu_\omega}_{u,X^{0}_u(\omega), Z^{0}_u(\omega)}(\tilde{\tau}^n,\tilde{\alpha}^n,\tilde{\xi}^n)\ind_{\{\tau\wedge\rho_\cO\geq u \}\cap O_n}(\omega),\notag
\end{align}
where $\tilde \tau^n=\inf\{s\ge u: \tilde \eta^n_s=1\}\wedge T$. By construction of $\{B_n\}_{n\in\bN}$ (see \eqref{Iterating1}) and since $(X^0_u(\omega),Z^0_u(\omega))\in O_n$, $\bP_\omega$-a.s.\ for $\bP$-a.e.\ $\omega\in\Omega$, we have 
\begin{align}\label{eq:DPP2}
J^{\nu_\omega}_{u,X^{0}_u(\omega), Z^{0}_u(\omega)}(\tilde{\tau}^n,\tilde{\alpha}^n,\tilde{\xi}^n)\ind_{O_n}(\omega)\ge \Big(J^{\nu_\omega}_{u,x_n, z_n}(\tilde{\tau}^n,\tilde{\alpha}^n,\tilde{\xi}^n)-\eps\Big)\ind_{O_n}(\omega), 
\end{align}
upon also recalling that $Z^0_u(\omega)\le z_n$ for $\bP$-a.e.\ $\omega\in O_n$.

Plugging \eqref{eq:DPP1} and \eqref{eq:DPP2} back into \eqref{eq:DPP0} and using \eqref{Omega_0equalityInLaw}, we obtain
\begin{align}\label{eq:DPP3}
&\bE\Big[\Gamma_u(\rho^k_\cO,\hat{\tau}^k,\hat{\alpha}^k,\hat{\xi}^k,X^{k,0},Z^{k,0})\ind_{\{\tau\wedge\rho_\cO\geq u \}}\ind_{U_k}\Big]\\
&\ge \sum_{n=1}^k\bE\Big[\Big(J^{\nu_\omega}_{u,x_n,z_n}(\tilde \tau^n,\tilde \alpha^n,\tilde \xi^n)-\eps\Big)\ind_{\{\tau\wedge\rho_\cO\ge u\}\cap O_n}\Big]\notag\\
&= 
\sum_{n=1}^k\bE\Big[\Big(J^{\bar \nu}_{u,x_n,z_n}( \tau^n, \alpha^n, \xi^n)-\eps\Big)\ind_{\{\tau\wedge\rho_\cO\ge u\}\cap O_n}\Big].\notag
\end{align}
Now recall that the treble $(\tau^n,\alpha^n,\xi^n)\in\AA^{\bar \nu}_{u,\bar z-z_n}$ is $\tfrac{1}{m}$-optimal. Thus,
\begin{align}\label{eq:DPP4}
&\bE\Big[\Big(J^{\bar \nu}_{u,x_n,z_n}( \tau^n, \alpha^n, \xi^n)-\eps\Big)\ind_{\{\tau\wedge\rho_\cO\ge u\}\cap O_n}\Big]\\
&\ge \bE\Big[\Big(v(u,x_n,z_n)-\eps-\tfrac{1}{m}\Big)\ind_{\{\tau\wedge\rho_\cO\ge u\}\cap O_n}\Big]\nonumber\\
&\ge \bE\Big[\Big(v(u,X_u,Z_u)-2\eps-\tfrac{1}{m}\Big)\ind_{\{\tau\wedge\rho_\cO\ge u\}\cap O_n}\Big],\notag
\end{align}
where the second inequality uses once again \eqref{Iterating1} and that $(X_u(\omega),Z_u(\omega))\in B_n$ for $\omega\in O_n$. Combining \eqref{eq:DPP3} and \eqref{eq:DPP4} and summing over the indicator functions, we arrive at 
\begin{align}\label{eq:DPP5}
&\bE\Big[\Gamma_u(\rho^k_\cO,\hat{\tau}^k,\hat{\alpha}^k,\hat{\xi}^k,X^{k,0},Z^{k,0})\ind_{\{\tau\wedge\rho_\cO\geq u \}}\ind_{U_k}\Big]\\
&\ge\! \bE\Big[\Big(v(u,X_u,Z_u)\!-\!2\eps\!-\!\tfrac{1}{m}\Big)\ind_{\{\tau\wedge\rho_\cO\ge u\}}\ind_{U_k}\Big].\notag
\end{align}

On the event $U^c_k$ we have $\hat\eta^k_{u+}=1$, which corresponds to $\hat \tau^k=u$. Then, in \eqref{Zsplit},
\begin{align}\label{eq:DPP6}
&\bE\Big[\Gamma_u(\rho^k_\cO,\hat \tau^k, \hat \alpha^k, \hat \xi^k, X^{k,0},Z^{k,0})\ind_{\{\tau\wedge\rho_\cO\ge u\}}\ind_{U^c_k} \Big]\\
&=\bE\Big[\Big(g_1(u,X^{k,0}_u,Z^{k,0}_u)\ind_{\{\rho^k_\cO=u\}}+g_2(u,X^{k,0}_u,Z^{k,0}_u)\ind_{\{\rho^k_\cO>u\}}\Big)\ind_{\{\tau\wedge\rho_\cO\ge u\}}\ind_{U^c_k} \Big]\notag\\
&=\bE\Big[\Big(g_1(u,X^{k}_u,Z^{k}_u)\ind_{\{\rho_\cO=u\}}+g_2(u,X^{k}_u,Z^{k}_u)\ind_{\{\rho_\cO>u\}}\Big)\ind_{\{\tau\wedge\rho_\cO\ge u\}}\ind_{U^c_k} \Big],\notag
\end{align}
where the final equality is by \eqref{eq:X0Xk}. Putting together \eqref{eq:DPP5}--\eqref{eq:DPP6} and \eqref{vGeqRunning+Terminal} we arrive at 
\begin{align}
v(t,x,z)&\geq\bE\bigg[\int_t^{u\wedge\tau\wedge\rho_\cO}f(s,X_s,Z_s,\alpha_s)\ud s-\int_{[t,u\wedge\tau\wedge\rho_\cO)}\langle c_+(s,X_s),\ud\xi^{+}_s\rangle\nonumber\\
&\hspace{30pt}-\int_{[t,u\wedge\tau\wedge\rho_\cO)}\langle c_-(s,X_s),\ud\xi^{-}_s\rangle+g_1(\rho_\cO,X_{\rho_\cO},Z_{\rho_\cO})\ind_{\{\rho_\cO\leq\tau \}\cap\{\rho_\cO<u \}}\\
&\hspace{30pt}+g_2(\tau,X_{\tau},Z_\tau)\ind_{\{\tau<\rho_\cO\wedge u \}}+\Big(v(u,X_u,Z_u)-2\eps-\tfrac{1}{m}\Big)\ind_{\{\tau\wedge\rho_\cO\geq u \}}\ind_{U_k} \nonumber\\
&\hspace{30pt}+\Big(g_1(u,X_{u},Z_{u})\ind_{\{\rho_\cO= u \}}+g_2(u,X_u,Z_u)\ind_{\{\rho_\cO> u \}}\Big)\ind_{\{\tau\geq u \}\cap U^c_k}\bigg].\notag
\end{align}

Letting $k\to\infty$ we have $U_k\uparrow\Omega$ and $U_k^c\downarrow\varnothing$. Moreover, $v\ge g_2$ by definition and both $g_1$ and $g_2$ are bounded from below by Assumption \ref{ass:obj0}. Then, by Fatou's lemma, we obtain
\begin{align}\label{SecondInequality}
v(t,x,z)&\geq\bE\bigg[\int_t^{u\wedge\tau\wedge\rho_\cO}f(s,X_s,Z_s,\alpha_s)\ud s-\int_{[t,u\wedge\tau\wedge\rho_\cO)}\langle c_+(s,X_s),\ud\xi^{+}_s\rangle\nonumber\\
&\hspace{30pt}-\int_{[t,u\wedge\tau\wedge\rho_\cO)}\langle c_-(s,X_s),\ud\xi^{-}_s\rangle+g_1(\rho_\cO,X_{\rho_\cO},Z_{\rho_\cO})\ind_{\{\rho_\cO\leq\tau \}\cap\{\rho_\cO<u \}}\\
&\hspace{30pt}+g_2(\tau,X_{\tau},Z_\tau)\ind_{\{\tau<\rho_\cO\wedge u \}}+v(u,X_u,Z_u)\ind_{\{\tau\wedge\rho_\cO\geq u \}}\bigg]-\big(\tfrac{1}{m}+2\eps\big)\bP(\tau\wedge\rho_\cO\geq u)\notag
\end{align}
Letting $m\to\infty$ and $\eps\to 0$, then taking the supremum over $(\tau,\alpha,\xi)\in\cT^\nu_t\times\cA^\nu_t\times\cX^\nu_t$ in \eqref{SecondInequality} we obtain our second inequality. Recalling the result in part 1 of this proof we conclude. 
\end{proof}
\begin{remark}
The proof of Theorem \ref{thm:DPPdet} follows by identical arguments but dropping the dependence on the state process $Z$ (i.e., the state dynamics is only given by the couple $(t,X)$). The only note of caution is that instead of the finite-fuel condition, in the infinite-fuel problem one must check the admissibility condition $\bE\big[|V_{[t,T]}(\hat{\xi}^{k})|^p\big]<\infty$. That is not hard because
\[
V_{[t,T]}(\hat{\xi}^{k})=V_{[t,u]}(\xi^0)+\sum_{n=1}^k V_{[u,T]}(\tilde \xi^n)\ind_{O_n}+V_{[u,T]}(\xi^0)\ind_{U^c_k}.
\]
Clearly,
\[
\bE\big[|V_{[t,u]}(\xi^0)+V_{[u,T]}(\xi^0)\ind_{U^c_k}|^p\big]\le \bE\big[|V_{[t,T]}(\xi^0)\big|^p\big]<\infty
\] 
because $\xi^0\in\cX^\nu_{t}$. Moreover, recalling \eqref{TildeXi^n=Xi^n} we also obtain, for all $n\le k$,
\begin{align*}
\bE\big[|V_{[u,T]}(\tilde \xi^n_s)|^p\big]&=\bE\big[\bE[|V_{[u,T]}(\tilde \xi^n_s)|^p\big|\cF^{t,0}_u](\omega)\big]\nonumber\\
&=\bE\big[\bE_\omega[|V_{[u,T]}(\tilde \xi^n_s)|^p]\big]=\bE\big[\bar{\bE}[|V_{[u,T]}(\xi^{n,0}_s)|^p]\big]=\bar{\bE}\big[|V_{[u,T]}(\xi^{n}_s)|^p\big]<\infty,\nonumber
\end{align*}
where the final equality is by $\bar \bP$-indistinguishability of $\xi^n$ and $\xi^{n,0}$ and the inequality is by $\xi^n\in\cX^{\bar{\nu}}_u$.
\end{remark}

	\begin{remark}\label{Rmk:PseudoMarkov}
		In the proof of Theorem \ref{thm:DPPdet} we use the equivalence between weak formulation and strong formulation of the stochastic control problem (see Proposition \ref{Prop:IndepProbSyst}). For example, when we pass to the reference probability system $\nu_\omega$ at time $u\in[t,T]$ we need the weak formulation to argue that the inequality in \eqref{Part1Chain3} holds. Similar technical steps in other parts of the proof use the same argument.
	\end{remark}

Our next goal is to extend the dynamic programming principle to stopping times $\sigma\in(t,T)$. In the process we also prove Proposition \ref{Prop:Msupermatingale}, which turns out to be a useful tool in the proof of Theorems \ref{thm:DPPst} and \ref{thm:DPPsttff}. 

\begin{proof}[{\bf Proof of Proposition \ref{Prop:Msupermatingale}}] We perform the full proof only in the case of finite fuel. The analogue for the infinite-fuel problem follows by the same arguments. Following \cite[Ch.~II.1]{revuz2013continuous} we need to show that $\bE[(M^{\nu;\alpha,\xi;}_{s\wedge\tau})^-]<\infty$ for all $s\in[t,T]$ and $\bE[M^{\nu;\alpha,\xi}_{s\wedge\tau}|\cF^{t}_u]\le M^{\nu;\alpha,\xi}_{u\wedge\tau}$ for all $t\le u\le s\le T$. The integrability condition is immediate from Assumption \ref{ass:obj0}. Let us now prove the supermartingale inequality.
	
Fix $(t,x,z)$ and an arbitrary admissible treble $(\tau,\alpha,\xi)\in\AA^\nu_{t,\bar z-z}$ on a given reference probability system $\nu\in\cV_t$. Setting $X=X^{\nu;\alpha,\xi}$, $Z=Z^{\nu;\xi}$ and $\rho_\cO:=\rho_\cO^{\nu;\alpha,\xi}$ and arguing as in the proof of Theorem \ref{thm:DPPdetff} we obtain an $\{\cF^{t,0}_s\}$-predictable treble $(\eta^{\tau,0},\alpha^0,\xi^0)\in\AA^\nu_{t,\bar z-z}$ and a process $(X^0,Z^0)$ that are $\bP$-indistinguishable from the original ones (notice that actually $\alpha=\alpha^0$ only $\bP\times\ud s$-a.e.). Then, given an arbitrary $u\in[t,T]$, for $\bP$-a.e.\ $\omega\in\Omega$ we introduce the reference probability system $\nu_\omega\in\cV_u$ defined in \eqref{eq:nuomega}. In this system, we have 
\[
\bP_\omega\Big(\big\{\omega'\in\Omega: (X^{0}_u(\omega'),Z^0_u(\omega'))=(X^{0}_u(\omega),Z^0_u(\omega))\big\}\Big)=1.
\]
Moreover, defining $\eta^{0,u}_s=\ind_{\{s>u\}}\eta^{\tau,0}_s$ and $\xi^{0,u}_s=\xi^0_s-\xi^0_u$, the process $(X^0,Z^0)$ follows the dynamics \eqref{X^0SDE} in the reference probability system $\nu_\omega$, i.e., more precisely, it is $\bP_\omega$-indistinguishable of the pair of processes $X^{u,X^0_u(\omega);\nu_\omega;\alpha^0,\xi^{0,u}}$ and $Z^{u,Z^0_u(\omega);\nu_\omega;\xi^{0,u}}$ defined as in \eqref{StateProcess} and \eqref{eq:Z} but under $\nu_\omega$. From indistinguishibility of $(X,Z)$ and $(X^0,Z^0)$, it follows that $\bP(\rho_\cO=\rho^0_\cO)=1$, with $\rho^0_\cO$ as in \eqref{eq:rho0}, and therefore $\bP_\omega(\rho_\cO=\rho^0_\cO)=1$ for $\bP$-a.e.\ $\omega\in\Omega$ (Lemma \ref{Lemma:AlmostSureSets}). In particular, below we will use that on the event $\{\rho_\cO\ge u\}$ we have $\rho^0_\cO=\inf\{s\ge u: (X^0_s,Z^0_s)\notin\cO \}\wedge T$, i.e., $\rho^0_\cO$ is defined {\em after} time $u$ and therefore it is a non-trivial stopping time in $\nu_\omega$. In the same spirit, on the event $\{\tau\ge u\}$ we have $\tau=\inf\{s\ge u: \eta^{0,u}_s=1\}\wedge T$ so that $\tau$ is a non-trivial stopping time in $\nu_\omega$.

Let $\cN$ be such that $\nu_{\omega}\in\cV_u$ for every $\omega\in\Omega\setminus\cN$ with $\bP(\cN)=0$ and fix $\bar \omega\in\Omega\setminus\cN$. Then, by the dynamic programming principle equation (Theorem \ref{thm:DPPdetff}) applied in the reference probability system $\nu_{\bar \omega}\in\cV_u$ with $(X^0,Z^0)$ as described above, we have 
\begin{align}\label{v_uGeq}
&\ind_{\{\tau(\bar \omega)\wedge\rho_\cO(\bar \omega)\ge u\}}v(u,X^0_u(\bar \omega),Z^0_u(\bar \omega))\notag\\
&\geq \ind_{\{\tau(\bar \omega)\wedge\rho_\cO(\bar \omega)\ge u\}}\bE_{\bar \omega}\bigg[\int_u^{s\wedge\tau\wedge\rho^0_\cO} f(r,X^0_r,Z^0_r,\alpha^0_r)\ud r-\int_{[u,s\wedge\tau\wedge\rho^0_\cO)}\langle c_+(r,X^0_r),\ud\xi^{0,u;+}_r\rangle\\
&\qquad-\int_{[u,s\wedge\tau\wedge\rho^0_\cO)}\langle c_-(r,X^0_r),\ud\xi^{0,u;-}_r\rangle+g_1(\rho^0_\cO,X^0_{\rho^0_\cO},Z^0_{\rho^0_\cO})\ind_{\{\rho^0_\cO\leq\tau\}\cap\{ \rho^0_\cO<s \}}\nonumber\\
&\qquad+g_2(\tau,X^0_{\tau},Z^0_{\tau})\ind_{\{\tau< s\wedge\rho^0_\cO\}}+v(s,X^0_s,Z^0_s)\ind_{\{\tau\wedge\rho^0_\cO\geq s\}}\bigg]=:\ind_{\{\tau(\bar \omega)\wedge\rho_\cO(\bar \omega)\ge u\}}\Psi_{u,s}(\bar \omega),\notag
\end{align}
for every $s\in[u,T]$, as $(\tau,\alpha^0,\xi^{0,u})\in\AA^{\nu_{\bar{\omega}}}_{u,\bar z-Z^0_u(\bar{\omega})}$ (see Proposition \ref{Prop:StdRefProbSyst}). Since the above expression holds for arbitrary $\bar \omega\in\Omega\setminus\cN$, plugging it into our definition of $M^{\nu;\alpha,\xi}$ (see \eqref{eq:M}) we obtain $\bP$-a.s. 
\begin{align}\label{M_uGeq2}
M^{\nu;\alpha,\xi}_{u\wedge\tau}&\geq\bigg(\int_t^{u\wedge\tau\wedge\rho_\cO} f(s,X_s,Z_s,\alpha_s)\ud s-\int_{[t,u\wedge\tau\wedge\rho_\cO)}\!\Big(\!\langle c_+(s,X_s),\ud\xi^{+}_s\rangle\!+\!\langle c_-(s,X_s),\ud\xi^{-}_s\rangle\Big)\nonumber\\
&\hspace{30pt}+g_1(\rho_\cO,X_{\rho_\cO},Z_{\rho_\cO})\ind_{\{\rho_\cO\leq\tau\}\cap\{ \rho_\cO<u \}}+g_2(\tau,X_\tau,Z_\tau)\ind_{\{\tau< u\wedge\rho_\cO \}}\bigg)+\ind_{\{\tau\wedge\rho_\cO\ge u\}}\Psi_{u,s}.
\end{align}
Recalling that $\bE_\omega[Y]=\bE[Y|\cF^{t,0}_u](\omega)=\bE[Y|\cF^t_u](\omega)$ for any $\cF$-measurable $Y$ such that $\bE[(Y)^-]<\infty$ (Lemma \ref{Lemma:ExpectRegConditProb}), we obtain by Assumption \ref{ass:obj0}, for $\bP$-a.e.~$\omega$,
\begin{align}\label{eq:Psi}
\Psi_{u,s}(\omega)&=\bE\bigg[\int_u^{s\wedge\tau\wedge\rho_\cO} f(r,X_r,Z_r,\alpha_r)\ud r-\int_{[u,s\wedge\tau\wedge\rho_\cO)}\langle c_+(r,X_r),\ud\xi^{+}_r\rangle\nonumber\\
&\hspace{30pt}-\int_{[u,s\wedge\tau\wedge\rho_\cO)}\langle c_-(r,X_r),\ud\xi^{-}_r\rangle+g_1(\rho_\cO,X_{\rho_\cO},Z_{\rho_\cO})\ind_{\{\rho_\cO\leq\tau\}\cap\{ \rho_\cO<s \}}\\
&\hspace{30pt}+g_2(\tau,X_{\tau},Z_{\tau})\ind_{\{\tau< s\wedge\rho_\cO\}}+v(s,X_s,Z_s)\ind_{\{\tau\wedge\rho_\cO\geq s\}}\bigg|\cF^t_u\bigg](\omega),\notag
\end{align}
where we also used the $\bP$-equivalence of $(\tau,\alpha,\xi,\rho_\cO, X,Z)$ and their $\{\cF^{t,0}_s\}$-predictable counterpart $(\tau^0,\alpha^0,\xi^0,\rho^0_\cO,X^0,Z^0)$ and that $\ud \xi^{0,u;\pm}_r=\ud \xi^{0;\pm}_r$ for $r\in[u,T]$, $\bP$-a.s.\ by construction.
Combining \eqref{M_uGeq2} and \eqref{eq:Psi}, we arrive at $M^{\nu;\alpha,\xi}_{u\wedge\tau}\ge \bE[M^{\nu;\alpha,\xi}_{s\wedge\tau}|\cF^t_u]$, $\bP$-a.s., as needed.

Now assume that the treble $(\tau^*,\alpha^*,\xi^*)\in\AA_{t,\bar{z}-z}^\nu$ is optimal, i.e.,
\begin{align}\label{eq:optimal}
v(t,x,z)=J^\nu_{t,x,z}(\tau^*,\alpha^*,\xi^*).
\end{align}
Then, repeating step 1 in the proof of Theorem \ref{thm:DPPdetff} with the treble $(\tau,\alpha,\xi)=(\tau^*,\alpha^*,\xi^*)$, we obtain
\begin{align*}
&J^\nu_{t,x,z}(\tau^*,\alpha^*,\xi^*)\\
&\le \bE\bigg[\!\int_t^{u\wedge\tau^*\wedge\rho^*_\cO}\!\!\! f(s,X^*_s,Z^*_s,\alpha^*_s)\ud s\!-\!\int_{[t,u\wedge\tau^*\wedge\rho^*_\cO)}\!\Big(\langle c_+(s,X^*_s),\ud\xi^{*,+}_s\rangle\!+\!\langle c_-(s,X^*_s),\ud\xi^{*,-}_s\rangle\Big)\nonumber\\
&\quad\!+\!g_1(\rho^*_\cO,X^*_{\rho^*_\cO},Z^*_{\rho^*_\cO})\ind_{\{\rho^*_\cO\leq\tau^*\}\cap\{ \rho^*_\cO<u \}}\!+\!g_2(\tau^*,X^*_{\tau^*},Z^*_{\tau^*})\ind_{\{\tau^*< u\wedge\rho^*_\cO \}}\!+\!v(u,X^*_u,Z^*_u)\ind_{\{\tau^*\wedge\rho^*_\cO\geq u\}} \bigg],\nonumber
\end{align*}
where we have denoted $(X^*,Z^*):=(X^{\nu;\alpha^*,\xi^*}, Z^{\nu;\alpha^*,\xi^*})$ and $\rho^*_\cO:=\rho^{\nu;\alpha^*,\xi^*}_\cO$. 
Using the inequality above and \eqref{eq:optimal} it is not difficult to prove that
\[
v(u,X^*_u(\omega),Z^*_u(\omega))=J^{\nu_\omega}_{u,X^*_u(\omega),Z^*_u(\omega)}(\tau^*,\alpha^*,\xi^*), \quad \text{for $\bP$-a.e.\ $\omega\in\{\tau^*\wedge\rho_\cO\geq u \}$}.
\]
Then, the inequality in \eqref{v_uGeq} becomes an equality upon replacing $(\tau,\alpha^0,\xi^0)$ therein with $(\tau^*,\alpha^*,\xi^*)$. Thus, $M^{\nu;\alpha^*,\xi^*}_{s\wedge\tau^*}$ is a martingale. Notice that $\bE[|M^{\nu;\alpha^*,\xi^*}_{s\wedge\tau^*}|]<\infty$ holds thanks to   \eqref{MartingaleIntegrability}.
\end{proof}

We can now prove the dynamic programming principle for stopping times (Theorems \ref{thm:DPPst} and \ref{thm:DPPsttff}). Once again we only provide the full proof for the finite-fuel problem as the one for the infinite-fuel case follows the same arguments but with one fewer state variable.
\begin{proof}[{\bf Proof of Theorem \ref{thm:DPPsttff}}]
The proof is split into two steps and it uses a standard approximation argument for stopping times, combined with Theorem \ref{thm:DPPdetff}.
Fix $(t,x,z)\in[0,T]\times\bR^d\times[0,\bar z]$. For every $n\in\bN$ let $\Pi_n:=\{t^i_n \}_{i=0}^{2^n}$ be the $n$-th dyadic partition of $[t,T]$, i.e., $t^i_n:=t+i/2^n(T-t)$. \medskip
		
\textit{Step 1.} (\textit{inequality} $\leq$). 	Let $(\tau,\alpha,\xi)\in\AA^\nu_{t,\bar z-z}$ and let $(X,Z)=(X^{\nu;\alpha,\xi},Z^{\nu;\alpha,\xi})$. Let $\sigma\in\cT^\nu_t$ be arbitrary but fixed. Since $\sigma$ is a stopping time with respect to the augmented Brownian filtration $\{\cF^t_s\}$, then it is $\{\cF^t_s\}$-predictable (see, e.g., \cite[Ch.\ V, Cor.\ 3.3]{revuz2013continuous}). Thus, there exists a sequence $(\sigma_n)_{n\in\bN}$ of $\{\cF^{t}_s\}$-predictable stopping times such that $\sigma_n\uparrow \sigma$ as $n\to\infty$ on the event $\{\sigma>t\}$ and $\sigma_n$ takes values in $\Pi_n$ (see, e.g., \cite[Ch.\ IV, Thm.\ 77]{dellacherie1978probabilities}). We set $A^i_n:=\{\sigma_n=t^i_n \}\in\cF^t_{t^i_n}$.

Recalling $\Gamma_s$ from \eqref{DefY} we have, for every $n\in\bN$,
\begin{align}\label{Jdecomposed2}
J_{t,x,z}^\nu(\tau,\alpha,\xi)&=\bE\bigg[\int_t^{\sigma_n\wedge\tau\wedge\rho_\cO}f(s,X_s,Z_s,\alpha_s)\ud s-\int_{[t,\sigma_n\wedge\tau\wedge\rho_\cO)}\langle c_+(s,X_s),\ud\xi^+_s\rangle\nonumber\\
&\hspace{30pt}-\int_{[t,\sigma_n\wedge\tau\wedge\rho_\cO)}\langle c_-(s,X_s),\ud\xi^-_s\rangle+g_1(\rho_\cO,X_{\rho_\cO},Z_{\rho_\cO})\ind_{\{\rho_\cO\leq\tau \}}\ind_{\{\tau\wedge\rho_\cO<\sigma_n \}}\nonumber\\
&\hspace{30pt}+g_2(\tau,X_\tau,Z_\tau)\ind_{\{\rho_\cO>\tau \}}\ind_{\{\tau\wedge\rho_\cO<\sigma_n \}}+\Gamma_{\sigma_n}(\rho_\cO,\tau,\alpha,\xi,X,Z)\ind_{\{\tau\wedge\rho_\cO\geq \sigma_n \}} \bigg]\nonumber\\
&=\sum_{i=0}^{2^n}\bE\bigg[\bigg(\int_t^{t^i_n\wedge\tau\wedge\rho_\cO}f(s,X_s,Z_s,\alpha_s)\ud s-\int_{[t,t^i_n\wedge\tau\wedge\rho_\cO)}\langle c_+(s,X_s),\ud\xi^+_s\rangle\\
&\hspace{56pt}-\int_{[t,t^i_n\wedge\tau\wedge\rho_\cO)}\langle c_-(s,X_s),\ud\xi^-_s\rangle+g_1(\rho_\cO,X_{\rho_\cO},Z_{\rho_\cO})\ind_{\{\rho_\cO\leq\tau \}\cap\{\rho_\cO<t^i_n \}}\nonumber\\
&\hspace{56pt}+g_2(\tau,X_\tau,Z_\tau)\ind_{\{\tau<t^i_n\wedge\rho_\cO \}}+\Gamma_{t^i_n}(\rho_\cO,\tau,\alpha,\xi,X,Z)\ind_{\{\tau\wedge\rho_\cO\geq t^i_n \}}\bigg)\ind_{A^i_n} \bigg].\nonumber
\end{align}
Now, for every $n\in\bN$ and every $i=0,1,\ldots,2^n$, we apply the same argument as in the dynamic programming principle for deterministic times \eqref{Part1Chain1}-\eqref{Part1Chain3} with $u=t^i_n$ and using the regular conditional probability $\bP^{t^i_n}_\omega(A)=\bP(A|\cF^{t,0}_{t^i_n})=\bP(A|\cF^{t}_{t^i_n})$ (also notice that $A^i_n\in\cF^t_{t^i_n}$). Thus, we obtain
\begin{equation}
\bE\Big[\Gamma_{t^i_n}(\rho_\cO,\tau,\alpha,\xi,X,Z)\ind_{\{\tau\wedge\rho_\cO\geq t^i_n\}}\ind_{A^i_n}\Big]\leq \bE\Big[v(t^i_n,X_{t^i_n},Z_{t^i_n})\ind_{\{\tau\wedge\rho\geq t^i_n\}}\ind_{A^i_n}\Big].
\end{equation}
Plugging this inequality into \eqref{Jdecomposed2} and summing over the events $A^i_n$, $i=0,\ldots 2^n$, we have 
\begin{align}
J_{t,x,z}^\nu(\tau,\alpha,\xi)&\leq\bE\bigg[\int_t^{\sigma_n\wedge\tau\wedge\rho_\cO}f(s,X_s,Z_s,\alpha_s)\ud s-\int_{[t,\sigma_n\wedge\tau\wedge\rho_\cO)}\langle c_+(s,X_s),\ud\xi^+_s\rangle\nonumber\\
&\hspace{30pt}-\int_{[t,\sigma_n\wedge\tau\wedge\rho_\cO)}\langle c_-(s,X_s),\ud\xi^-_s\rangle+g_1(\rho_\cO,X_{\rho_\cO},Z_{\rho_\cO})\ind_{\{\rho_\cO\leq\tau \}\cap\{\rho_\cO<\sigma_n \}}\\
&\hspace{30pt}+g_2(\tau,X_\tau,Z_\tau)\ind_{\{\tau<\sigma_n\wedge\rho_\cO \}}+v(\sigma_n,X_{\sigma_n},Z_{\sigma_n})\ind_{\{\tau\wedge\rho_\cO\geq \sigma_n\}} \bigg].\nonumber
\end{align}
Letting $n\to\infty$, by Assumption \ref{ass:DCT}, we obtain
\begin{align*}
J_{t,x,z}^\nu(\tau,\alpha,\xi)&\leq\bE\bigg[\int_t^{\sigma\wedge\tau\wedge\rho_\cO}f(s,X_s,Z_s,\alpha_s)\ud s-\int_{[t,\sigma\wedge\tau\wedge\rho_\cO)}\langle c_+(s,X_s),\ud\xi^+_s\rangle\nonumber\\
&\qquad-\int_{[t,\sigma\wedge\tau\wedge\rho_\cO)}\langle c_-(s,X_s),\ud\xi^-_s\rangle+g_1(\rho_\cO,X_{\rho_\cO},Z_{\rho_\cO})\ind_{\{\rho_\cO\leq\tau \}\cap\{\rho_\cO<\sigma \}}\nonumber\\
&\qquad+g_2(\tau,X_\tau,Z_\tau)\ind_{\{\tau<\sigma\wedge\rho_\cO \}}+v(\sigma,X_{\sigma},Z_{\sigma})\ind_{\{\tau\wedge\rho_\cO\geq \sigma\}} \bigg].
\end{align*}
Taking the supremum over $(\tau,\alpha,\xi)\in\AA^\nu_{t,\bar z-z}$ we conclude this step in the proof.
\medskip

\textit{Step 2.} (\textit{inequality} $\geq$). Again, let $(\tau,\alpha,\xi)\in\AA^\nu_{t,\bar z-z}$ and let $(X,Z)=(X^{\nu\alpha,\xi},Z^{\nu;\alpha,\xi})$. The process $M^{\nu;\alpha,\xi}_{u\wedge\tau}$ defined in \eqref{eq:M}
is a supermartingale on $\nu$ by Proposition \ref{Prop:Msupermatingale}, with $M^{\nu;\alpha,\xi}_{t\wedge\tau}=M^{\nu;\alpha,\xi}_{t}=v(t,x,z)$. 
For every $n\in\bN$, let $\sigma_n\in\cT^\nu_t$ be as in step 1 above. Since $\sigma_n$ takes values in a discrete set we can apply optional sampling theorem in discrete time to the left-continuous supermartingale $M^{\nu;\alpha,\xi}_{u\wedge\tau}$ and obtain
$v(t,x,z)=M^{\nu;\alpha,\xi}_t\ge\bE\big[M^{\nu;\alpha,\xi}_{\sigma_n\wedge\tau}\big]$, for all $n\in\bN$, since $\cF^t_t=\cF^t_{t^0_n}$ is trivial.
That is, for every $n\in\bN$,
\begin{align*}
v(t,x,z)&\geq\bE\bigg[\int_t^{\sigma_n\wedge\tau\wedge\rho_\cO}\!\! f(s,X_s,Z_s,\alpha_s)\ud s-\int_{[t,\sigma_n\wedge\tau\wedge\rho_\cO)}\langle c_+(s,X_s),\ud\xi^+_s\rangle\nonumber\\
&\qquad\:\:- \int_{[t,\sigma_n\wedge\tau\wedge\rho_\cO)}\langle c_-(s,X_s),\ud\xi^-_s\rangle
+g_1(\rho_\cO,X_{\rho_\cO},Z_{\rho_\cO})\ind_{\{\rho_\cO\leq\tau\}\cap\{ \rho_\cO<\sigma_n \}}\\
&\qquad\:\:+g_2(\tau,X_\tau,Z_\tau)\ind_{\{\tau< \sigma_n\wedge\rho_\cO \}}
+v(\sigma_n,X_{\sigma_n},Z_{\sigma_n})\ind_{\{\tau\wedge\rho_\cO\geq \sigma_n\}}\bigg].
\end{align*}
Letting $n\to\infty$, by Assumption \ref{ass:DCT}, we obtain
\begin{align*}
v(t,x,z)&\geq\bE\bigg[\int_t^{\sigma\wedge\tau\wedge\rho_\cO}f(s,X_s,Z_s,\alpha_s)\ud s-\int_{[t,\sigma\wedge\tau\wedge\rho_\cO)}\langle c_+(s,X_s),\ud\xi^+_s\rangle\nonumber\\
&\hspace{30pt}-\int_{[t,\sigma\wedge\tau\wedge\rho_\cO)}\langle c_-(s,X_s),\ud\xi^-_s\rangle+g_1(\rho_\cO,X_{\rho_\cO},Z_{\rho_\cO})\ind_{\{\rho_\cO\leq\tau \}\cap\{\rho_\cO<\sigma \}}\nonumber\\
&\hspace{30pt}+g_2(\tau,X_\tau,Z_\tau)\ind_{\{\tau<\sigma\wedge\rho_\cO \}}+v(\sigma,X_{\sigma},Z_\sigma)\ind_{\{\tau\wedge\rho_\cO\geq \sigma\}} \bigg].
\end{align*}
Combined with step 1 above, we conclude.
\end{proof}
	
\begin{remark}
Notice that for the controls $\xi\in\cX^\nu_t$, the condition $\bE[|V_{[t,T]}(\xi)|^p]<\infty$ is not needed to prove the dynamic programming principle (if anything it makes some arguments lengthier). Therefore, our study could have been developed without assuming any integrability of the admissible singular controls. It is also worth noticing that existence and uniqueness of the solution to the SDE \eqref{StateProcess} is guaranteed by, e.g., \cite[Theorem 1]{doleans1976existence} without assuming integrability on the singular controls. 

It is however useful to have developed our arguments with such additional constraint, in order to allow for wider applicability of our results. Indeed, a-priori it is not obvious that adding constraints to the set of admissible controls is without consequence for the DPP {\em(\cite[Remark 2.15]{fabbri2017stochastic})}. Moreover, the condition $\bE[|V_{[t,T]}(\xi)|^p]<\infty$ is useful to prove continuity of the value function {\em(see, e.g., \cite[Ch.~III, Sec.~9]{milazzo2021})}.
\end{remark}

	\appendix
	\section{Regular conditional probabilities}\label{App:RegCondProb}
	
In this appendix, we gather well-known results on regular conditional probabilities most of which can be sourced, for instance, from \cite[Ch.~I, Sec.~III]{ikeda1989stochastic} and \cite[Ch.~II, Sec.~III]{fabbri2017stochastic}. For completeness, we also provide most proofs for the convenience of an unfamiliar reader. Regular conditional probabilities allow us to construct reference probability systems starting at future times $u\in[t,T]$ that preserve the history of the problem up to time $u$ (see Proposition \ref{Prop:StdRefProbSyst}). These systems provide a pseudo-Markovian structure (as remarked in \cite{claisse2016pseudo}) which is essential in order to obtain the dynamic programming principle for our (in principle non-Markovian) problem.  

\subsection{Background material}
	
Let $(\Omega,\cH,\bQ)$ be a probability space and $\cG\subseteq\cH$ be a sub $\sigma$-algebra. A function $p:\Omega\times\cH\to[0,1]$ is called regular conditional probability (given $\cG$) if
\begin{enumerate}[(i)]
\item for every $\omega\in\Omega$, the function $\bQ_\omega:=p(\omega,\cdot)$ is a probability measure on $(\Omega,\cH)$;
\item for every $A\in\cH$, the function $p(\cdot,A)$ is $\cG$-measurable;
\item for every $A\in\cH$, we have $\bQ_\omega(A)=\bE^\bQ[\ind_A|\cG](\omega)=:\bQ(A|\cG)$ for $\bQ$-a.e.~$\omega$.
\end{enumerate}

By recalling that $(X)^\pm$ denote the positive/negative part of X, if $X$ is $\cH$-measurable and either $\bE^\bQ[(X)^-]<\infty$ or $\bE^\bQ[(X)^+]<\infty$, it follows that
\begin{equation}\label{RegConditProbProp}
\bE^\bQ[X|\cG](\omega)=\bE^{\bQ_\omega}[X]:=\int_\Omega X(\omega')\ud\bQ_\omega(\omega'), \quad \bQ\text{-a.e.} \: \omega.
\end{equation}
Indeed, $\bE^\bQ[X]$ is well-defined if $\min\{\bE^\bQ[(X)^-],\bE^\bQ[(X)^+]\}<\infty$ (see, e.g., \cite[p.~25, eq.~(2)]{rudin}) and conditional expectations can be constructed in the standard way with the aid of monotone convergence theorem.

Existence of regular conditional probabilities is non-trivial but the next result is well-known. Let $(\Omega,\cF,\bP)$ be a complete probability space, with $(\Omega,\cF^0)$ a standard measurable space in the sense of item {\bf (i)} in Section \ref{sec:notation} and $\cF$ the $\bP$-completion of $\cF^0$. Then, given a sub-$\sigma$-algebra $\cG\subseteq\cF^0$, there exists a regular conditional probability $p:\Omega\times\cF^0\to[0,1]$ which is $\bP$-a.s.\ unique. That is, if $p':\Omega\times\cF^0\to[0,1]$ is another regular conditional probability, given $\cG$, then there exists $N\in\cG$ with $\bP(N)=0$ such that if $\omega\in\Omega\setminus N$ then $p(\omega,A)=p'(\omega,A)$ for every $A\in\cF^0$. Moreover, if $\Gamma$ is a Polish space and $\gamma:\Omega\to\Gamma$ is $\cG/\cB(\Gamma)$-measurable then
	\begin{equation}\label{ConstRegConditProb}
	\bP_\omega(\{\omega' \in\Omega: \gamma(\omega')=\gamma(\omega)\})=1, \quad \bP\text{-a.e.~}\omega,
	\end{equation}
	i.e., $\gamma$ is $\bP_\omega$-a.s.~ equal to the constant $\gamma(\omega)$ for $\bP$-a.e.~ $\omega$. For a proof of these results see, e.g., Theorem 3.1, Theorem 3.2 and the subsequent Corollary of Chapter I in \cite{ikeda1989stochastic}.
	
\subsection{A collection of technical results}
Let $\nu=(\Omega,\cF,\bP,\{\cF^t_s\},W)\in\cV_t$ be a fixed standard reference probability system. Recall that $\cF$ is the completion of $\cF^0$ with the $\bP$-null sets and that $(\Omega,\cF^0)$ is a standard measurable space.

Let us fix $u\in[t,T]$. We denote by $\bP_\omega$ the regular conditional probability given $\cF^{t,0}_u$. That is, $\bP_\omega(A)=\bP(A|\cF^{t,0}_u)(\omega)$ for every $A\in\cF$ for $\bP$-a.e.\ $\omega$. We also denote by $\cF_\omega$ the completion of $\cF^0$ with the $\bP_\omega$-null sets. Thus,
\begin{equation}
\cF_\omega=\{A\cup N_1: A\in\cF^0, N_1\subseteq N\in\cF^0, \bP_\omega(N)=0 \}.
\end{equation}
For completeness we should use the notation $\bP^u_\omega(\,\cdot\,)=\bP(\,\cdot\,|\cF^{t,0}_u)(\omega)$, but since $u$ is fixed we prefer a simpler notation and drop the superscript $u$.

We define $W^u_s:=W_s-W_u$ for $s\in[u,T]$, $\cF^{u,0}_s:=\sigma(W^u_r: r\in[u,s])$ and we denote by $\cF^{u,\omega}_s$ the augmentation of $\cF^{u,0}_s$ with the $\bP_\omega$-null sets. For the sake of simplicity, we also denote by $\bE$ the expectation with respect to $\bP$ and by $\bE_\omega$ the expectation with respect to $\bP_\omega$.
	
	In this framework, we have the following results.
	
	\begin{lemma}\label{Lemma:AlmostSureSets}
		Let $A\in\cF$, then $A\in\cF_\omega$ for $\bP$-a.e.\ $\omega$. Moreover, if $\bP(A)=1$, then $\bP_\omega(A)=1$ for $\bP$-a.e.~ $\omega$. Conversely, let $A\in\cF^0$ be such that $\bP_\omega(A)=1$ for $\bP$-a.e.\ $\omega$, then $\bP(A)=1$.
	\end{lemma}
	\begin{proof}
		Since $A\in\cF$, then $A=A_0\cup N_0$ where $A_0\in\cF^0$ and $N_0\subseteq N\in\cF^0$ with $\bP(N)=0$. Therefore,
		\begin{equation}\label{AlmostSureSets}
		0=\bP(N)=\bE[\bE[\ind_{N}|\cF^{t,0}_u](\omega)]=\bE[\bP_\omega(N)].
		\end{equation}
		Thus, $\bP_\omega(N)=0$ for $\bP$-a.e.\ $\omega$ and so $N_0\in\cF_\omega$ for $\bP$-a.e.\ $\omega$. Hence, $A\in\cF_\omega$ for $\bP$-a.e.\ $\omega$. By a similar argument as in \eqref{AlmostSureSets}, we obtain the rest of the proof.
	\end{proof}
	
\begin{lemma}\label{Lemma:ExpectRegConditProb}
Let $Y:\Omega\to\bR^d$ be a $\cF$-measurable random variable such that either $\bE[(Y)^-]<\infty$ or $\bE[(Y)^+]<\infty$. Then, for $\bP$-a.e.~$\omega$, we have that $Y$ is $\cF_\omega$-measurable with $\min\{\bE_\omega[(Y)^-],\bE_\omega[(Y)^+]\}<\infty$ and $\bE_\omega[Y]=\bE[Y|\cF^t_u](\omega)$.
\end{lemma}
\begin{proof}
Let us only consider the case $\bE[(Y)^-]<\infty$, as the other case is analogous. Since $Y$ is $\cF$-measurable, then (see, e.g., \cite[Lemma 1.25]{kallenberg2006foundations}) there exists $Y^0$ which is $\cF^0$-measurable and such that $Y=Y^0$, $\bP$-a.s. Thus, by Lemma \ref{Lemma:AlmostSureSets}, we also have that $Y=Y^0$, $\bP_\omega$-a.s.\ for $\bP$-a.e.\ $\omega$. That is, for $\bP$-a.e.\ $\omega$, the set $F:=\{Y= Y^0 \}\in\cF_\omega$ and $\bP_\omega(F)=1$. Let us also denote $N:=\Omega\setminus F=F^c$ so that $N=\{Y\neq Y^0\}$ and $\bP_\omega(N)=0$, for $\bP$-a.e.\ $\omega$. 

Let $B\in\cB(\bR^d)$ and notice that $Y^{-1}(B)=\{\omega'\in\Omega: Y(\omega')\in B\}$ can be decomposed as $\big[Y^{-1}(B)\cap F\big]\cup\big[Y^{-1}(B)\cap N\big]$. Since $\cF_\omega$ is $\bP_\omega$-complete, then $Y^{-1}(B)\cap N\in\cF_\omega$. Moreover,
\[
Y^{-1}(B)\cap F=\{\omega'\in\Omega: Y^0(\omega')\in B\}\cap F\in\cF_\omega
\]
since $Y^0$ is $\cF^0$-measurable. Thus, $Y^{-1}(B)\in\cF_\omega$ for $\bP$-a.e.\ $\omega$, i.e., $Y$ is $\cF_\omega$-measurable.

Thanks to $\bE[(Y)^-]<\infty$, for $\bP$-a.e.~$\omega$, we have that $\bE_\omega[(Y)^-]=\bE[(Y)^-|\cF^{t,0}_u](\omega)<\infty$ . Finally, if $A\in\cF^{t}_u$ then $A=A_1\cup N_1$ with $A_1\in\cF^{t,0}_u$ and $\bP(N_1)=0$. Thus,
		\begin{equation}\label{eq:A.8}
		\bE[\bE[Y|\cF^{t,0}_u]\ind_A]=\bE[\bE[Y|\cF^{t,0}_u]\ind_{A_1\cup N_1}]=\bE[\bE[Y|\cF^{t,0}_u]\ind_{A_1}]=\bE[Y\ind_{A_1}]=\bE[Y\ind_A].
		\end{equation}
Hence, for $\bP$-a.e.~$\omega$, $\bE[Y|\cF^t_u](\omega)=\bE[Y|\cF^{t,0}_u](\omega)=\bE_{\omega}[Y]$, as needed.
\end{proof}

Recall that for $\bP$-a.e.\ $\omega\in\Omega$, the filtration $\{\cF^{u,\omega}_s\}_{s\in[u,T]}$ is defined as the $\bP_\omega$-augmentation of $\{\cF^{u,0}_s\}_{s\in[u,T]}$. In the next lemma we show that, for $\bP$-a.e.\ $\omega\in\Omega$, it actually coincides with the $\bP_\omega$-augmentation of $\{\cF^{t,0}_s\}_{s\in[t,T]}$. This conveys the intuitive idea that by adding all the $\bP_\omega$-null sets to $\{\cF^{u,0}_s\}_{s\in[u,T]}$, we recover the full Brownian filtration started at time $t$.
	
\begin{lemma}[{\cite[Lemma 2.26]{fabbri2017stochastic}}]\label{Lemma:FiltrationsInclusion}
For $\bP$-a.e.~ $\omega$,
$\cF^{u,\omega}_{s}$ coincides with $\cF^{t,0}_s$ augmented by the $\bP_\omega$-null sets of $\cF^0$.
\end{lemma}
\begin{proof}
Let $D:=\{t_n \}_{n\in\bN}$ be a dense subset of $[t,T]$. As $W$ has continuous paths, we have that $\cF^{t,0}_s=\sigma(W_{r}: r\in D\cap[t,s])$ for any $s\in[t,T]$ and $\cF^{u,0}_s=\sigma(W_{r}-W_u: r\in D\cap[u,s])$ for any $s\in[u,T]$. Clearly $\cF^{u,0}_s\subseteq \cF^{t,0}_s$ for any $s\in[u,T]$. Therefore, for $\bP$-a.e.\ $\omega$, it must be $\cF^{u,\omega}_s\subseteq \cF^{t,\omega}_s$ for all $s\in[u,T]$, where $\cF^{t,\omega}_s$ is the $\sigma$-algebra $\cF^{t,0}_s$ augmented by the $\bP_\omega$-null sets of $\cF^0$.
To complete our proof we next show that, for $\bP$-a.e.\ $\omega$, $\cF^{t,0}_s\subseteq \cF^{u,\omega}_s$ for every $s\in[u,T]$.
Let $s_m\in D\cap[u,T]$ and $t_{m,n}\in D\cap[t,s_m]$. 
By \eqref{ConstRegConditProb}, for each couple $(m,n)$ there is $\Omega_{m,n}\subset\Omega$ such that $\bP(\Omega_{m,n})=1$ and for every $\omega\in\Omega_{m,n}$ it holds $\bP_\omega (F^{m,n}_\omega)=1$, where
\[
F^{m,n}_\omega:=\{\omega'\in\Omega : W_{u\wedge t_{m,n}}(\omega')= W_{u\wedge t_{m,n}}(\omega)\}.
\]
We denote $N^{m,n}_\omega:=\Omega\setminus F^{m,n}_\omega=(F^{m,n}_\omega)^c$ and also recall that $\cF^{t,0}_T\subseteq\cF^0$. Therefore, $F^{m,n}_\omega\in\cF^0$ for each $\omega\in \Omega_{m,n}$ and, in particular, $F^{m,n}_\omega\in \cF^{u,\omega}_s$ for all $s\in[u,T]$ and all $\omega\in\Omega_{m,n}$.
		
Fix $\omega\in\Omega_{m,n}$. For any $B\in\cB(\bR^{d'})$, we have
\[
W^{-1}_{t_{m,n}}(B)=\big[W^{-1}_{t_{m,n}}(B)\cap F^{m,n}_\omega\big]\cup\big[W^{-1}_{t_{m,n}}(B)\cap N^{m,n}_\omega\big].
\]
Since $\cF^{u,\omega}_{s_m}$ contains all the $\bP_\omega$-null sets, then $W^{-1}_{t_{m,n}}(B)\cap N^{m,n}_\omega\in\cF^{u,\omega}_{s_m}$. For $t_{m,n}\le u$, we have that $W^{-1}_{t_{m,n}}(B)\cap F^{m,n}_\omega$ is either equal to $\Omega$ or $\varnothing$, depending on whether $W_{t_{m,n}}(\omega)\in B$ or not. Thus, $W^{-1}_{t_{m,n}}(B)\cap F^{m,n}_\omega\in\cF^{u,\omega}_{s_m}$ in that case.	For $u<t_{m,n}\leq s_m$, we have $W_{t_{m,n}}=W^u_{t_{m,n}}+W_u$. Then, 
\begin{align*}
W_{t_{m,n}}^{-1}(B)\cap F^{m,n}_\omega=\big\{\omega'\in\Omega: W_u(\omega')=W_u(\omega)\}\cap\{\omega'\in\Omega:W^u_{t_{m,n}}(\omega')\in \big(B-W_u(\omega)\big)\},
\end{align*}
where $B-x:=\{y-x\in\bR^{d'}, y\in B\}$. Both sets on the right-hand side of the above equation are $\cF^{u,\omega}_{s_m}$-measurable, so $W_{t_{m,n}}^{-1}(B)\in\cF^{u,\omega}_{s_m}$ also in this case.

So far we have proven that for any given $s_m\in D\cap [u,T]$ and $t_{m,n}\in D\cap[t,s_m]$, and for all $\omega\in \Omega_{m,n}$, we have $W^{-1}_{t_{m,n}}(B)\in \cF^{u,\omega}_{s_m}$ for any $B\in\cB(\bR^{d'})$.
Let $\bar{\Omega}:=\cap_{m,n\in\bN} \Omega_{m,n}$. Thus, $\bP(\bar{\Omega})=1$ and for every $\omega\in\bar{\Omega}$ we have that $W_r^{-1}(B)\in\cF^{u,\omega}_s$ for every $s\in D\cap[u,T]$, $r\in D\cap[t,s]$ and $B\in\cB(\bR^{d'})$. Therefore, for $\bP$-a.e.~$\omega$, we obtain
\begin{equation*}
\cF^{t,0}_s\subseteq\cF^{u,\omega}_s, \quad \text{for all $s\in D\cap[u,T]$}.
\end{equation*}
If $s\in [u,T]\setminus D$, then by left-continuity of $\{\cF^{t,0}_s \}$, we have
\begin{equation*}
\cF^{t,0}_s=\sigma\big(\cup_{r\in D\cap[t,s)} \cF^{t,0}_r\big)\subseteq \cF^{u,\omega}_s \quad \text{for all $\omega\in\bar{\Omega}$}.
\end{equation*}
		
Thus, we have shown that, for $\bP$-a.e.~$\omega$, we have that $\cF^{t,0}_s\subseteq\cF^{u,\omega}_{s}$ for every $s\in[u,T]$ as needed to conclude. 	
\end{proof}

In the next proposition, for $\bP$-a.e.\ $\omega\in \Omega$, we generate reference probability systems $\nu_\omega$ starting at time $u\in[t,T]$, i.e., $\nu_\omega\in\cV_u$ and then we construct admissible controls on $\nu_\omega$ starting from ones in the standard reference probability system $\nu\in\cV_t$. The probability measure on the reference probability systems $\nu_\omega$ is the regular conditional probability $\bP_\omega$.
To obtain these results, it is convenient to recall the notation from Section \ref{sec:ind}. Recall also that given any admissible treble $(\tau,\alpha,\xi)$, either in $\AA^\nu_{t}$ or $\AA^\nu_{t,\bar z -z}$, and letting $\eta^\tau$ be the increasing process associated to $\tau$, it is possible to construct a $\{\cF^{t,0}_s\}$-predictable treble $(\eta^{\tau,0},\alpha^0,\xi^0)$ which is also admissible (see Lemma \ref{Lemma:AlphaAlpha0} and \ref{Lemma:Predict0Indisting}).

\begin{proposition}\label{Prop:StdRefProbSyst}
Fix $\nu\in\cV_t$. Let $(\tau,\alpha,\xi)\in\AA^\nu_t$ and denote by $(\eta^{\tau,0},\alpha^0,\xi^0)$ the associated $\{\cF^{t,0}_s\}$-predictable admissible treble. Set $\xi^{0,u}_s:=\xi^0_s-\xi^0_u$, $\eta^{0,u}_s:=\ind_{\{s>u\}}\eta^{\tau,0}_s$ for $s\in[u,T]$ and let  $\alpha^{0,u}$ be the restriction of $\alpha^0$ to $[u,T]$. Finally, let
\[
\nu_\omega:=(\Omega,\cF_\omega,\bP_\omega,\{\cF^{u,\omega}_s\},W^u).
\]
Then, for $\bP$-a.e.\ $\omega\in\Omega$, we have $\nu_\omega\in\cV_u$ (i.e., $\nu_\omega$ is a legitimate reference probability system staring at time $u$) and $(\eta^{0,u},\alpha^{0,u},\xi^{0,u})\in\AA^{\nu_\omega}_u$. The same result holds if we replace $\AA^\nu_t$ and $\AA^{\nu_\omega}_u$ with $\AA^\nu_{t,\bar z-z}$ and $\AA^{\nu_\omega}_{u,\bar z-Z^{\xi,0}_u(\omega)}$, respectively, where $Z^{\xi,0}_u:=z+V_{[t,u]}(\xi^0)$.
\end{proposition}
	
\begin{proof}
We first want to prove that $\nu_\omega\in\cV_u$ for $\bP$-a.e.~ $\omega$. By construction, $(\Omega,\cF_\omega,\bP_\omega)$ is a complete probability space. It remains to show that, for $\bP$-a.e.~$\omega$, the process $W^u$ is a Brownian motion on $(\Omega,\cF_\omega,\bP_\omega)$ such that $\bP_\omega(W^u_u=0)=1$. 

Obviously, $W_u^u=0$ and $W^u$ has continuous paths. To prove it is a Brownian motion,
we must show that (see, e.g., \cite[Proposition 1.2.7]{morimoto2010stochastic}), for $\bP$-a.e.~$\omega$,
\begin{equation}\label{W^uBrownMot}
\bE_\omega[e^{i\langle y,W^u_{s}-W^u_{r}\rangle}|\cF^{u,\omega}_{r}]=\e^{-\frac{|y|^2}{2}(s-r)}, \quad \text{$\bP_\omega$-a.s.},
\end{equation}
for all $u\leq r<s\leq T$ and $y\in\bR^{d'}$.
		
Since, for $\bP$-a.e.~$\omega$, $\{\cF^{u,\omega}_{r}\}$ is the augmentation of $\{\cF^{t,0}_{r}\}$ with the $\bP_\omega$-null sets (see Lemma \ref{Lemma:FiltrationsInclusion}), then an argument as in \eqref{eq:A.8} guarantees that, for $\bP$-a.e.~$\omega$,
\begin{equation}
\bE_\omega[e^{i\langle y,W^u_{s}-W^u_r\rangle}|\cF^{u,\omega}_{r}]=\bE_\omega[e^{i\langle y,W^u_{s}-W^u_{r}\rangle)}|\cF^{t,0}_{r}], \quad\text{$\bP_\omega$-a.s.}, 
\end{equation}
for all $u\leq r<s\leq T$ and $y\in\bR^{d'}$. Then, \eqref{W^uBrownMot} is equivalent to proving that, for $\bP$-a.e.~$\omega$,
\begin{equation}\label{W^uBrownMotEquiv}
\bE_\omega[e^{-\frac{|y|^2}{2}(s-r)}\ind_A]=\bE_\omega[e^{i\langle y,W^u_{s}-W^u_{r}\rangle}\ind_A], \quad\text{for all $u\leq r<s\leq T$, $y\in\bR^{d'}$ and $A\in\cF^{t,0}_{r}$}.
\end{equation}
		
Now let $r\in [u,T]\cap\bQ$, $D:=[t,r]\cap\bQ$ and $(B_k)_{k}$ be a countable generator of $\cB(\bR^{d'})$. Since $W$ is continuous, we have that $\cF^{t,0}_r=\sigma(\cG)$ where $\cG=\{W_q^{-1}(B_k), \: q\in D, \: k\in\bN \}$ is a countable generator. Let $\cH:=\pi(\cG)$ be the $\pi$-system generated by $\cG$, then $\cH$ is a countable $\pi$-system generating $\cF^{t,0}_r$. Thus, for $s\in(r,T]\cap\bQ$, $y\in\bQ^{d'}$ and $A\in\cH$, we have
\begin{align}\label{eq:BMomega}
\bE_\omega[e^{-\frac{|y|^2}{2}(s-r)}\ind_A]&=\bE_\omega[\bE[e^{i\langle y,W_{s}-W_{r}\rangle}|\cF^t_{r}]\ind_A] \nonumber \\
&=\bE_\omega[\bE[e^{i\langle y,W^u_{s}-W^u_{r}\rangle}\ind_A|\cF^t_{r}]] \nonumber \\
&=\bE[\bE[e^{i\langle y,W^u_{s}-W^u_{r}\rangle}\ind_A|\cF^t_{r}]|\cF^{t,0}_u](\omega) \nonumber \\
&=\bE_\omega[e^{i\langle y,W^u_{s}-W^u_{r}\rangle}\ind_A], \quad \bP\text{-a.e.~} \omega,
\end{align}
where in the first line we have used that $W$ is a Brownian motion on $(\Omega,\cF,\bP)$, in the third line we have used the regular conditional probability property \eqref{RegConditProbProp} and in the fourth line we have used the tower property. Thus, we have just shown that for every $r,s,y$ and $A$ (taken in their respective countable sets) there exists $\Omega^A_{r,s,y}\subseteq\Omega$ with $\bP(\Omega^A_{r,s,y})=1$ such that \eqref{eq:BMomega} holds for every $\omega\in\Omega^A_{r,s,y}$. Then, take
\[
\bar \Omega= \bigcap_{\substack{u\le r< s\le T \\ r,s\in\bQ}}\bigcap_{y\in\bQ^{d'}}\bigcap_{A\in\cH}\Omega^A_{r,s,y},
\]
so that $\bP(\bar{\Omega})=1$ and \eqref{eq:BMomega} holds for every $\omega\in\bar{\Omega}$ and every $u\leq r< s\leq T$, $r,s\in\bQ$, $y\in\bQ^{d'}$, $A\in\cH$. Since $\cH$ is a countable $\pi$-system generating $\cF^{t,0}_r$, then (see, e.g., \cite[Theorem 1.1]{baldi2017stochastic}) we obtain that \eqref{eq:BMomega} holds in fact for every $A\in\cF^{t,0}_r$. Namely, for $\bP$-a.e.~$\omega$, we have that 
\begin{equation}\label{W^uBrownMotQ}
\bE_\omega[e^{-\frac{|y|^2}{2}(s-r)}\ind_A]=\bE_\omega[e^{i\langle y,W^u_{s}-W^u_{r}\rangle}\ind_A],
\end{equation}
for every $u\leq r< s\leq T$, $r,s\in\bQ$, $y\in\bQ^{d'}$ and $A\in\cF^{t,0}_r$.
For arbitrary $r\in[u,T]$, $s\in[r,T]$, $y\in\bR^{d'}$, we select sequences $(r_n)_n\subseteq [u,T]\cap\bQ$, $(s_n)_n\subseteq [r,T]\cap\bQ$ and $(y_n)_n\subseteq\bQ^{d'}$ such that $r_n\to r$, $s_n\to s$ and $y_n\to y$ as $n\to\infty$. Then, for all $\omega\in\bar \Omega$, using dominated convergence we obtain
\begin{align*}
\bE_\omega\big[e^{-\frac{|y|^2}{2}(s-r)}\ind_A\big]=\lim_{n\to\infty}\bE_\omega\big[e^{-\frac{|y_n|^2}{2}(s_n-r_n)}\ind_A\big]=\lim_{n\to\infty}\bE_\omega\big[e^{i\langle y_n,W^u_{s_n}-W^u_{r_n}\rangle}\ind_A\big]=\bE_\omega\big[e^{i\langle y,W^u_{s}-W^u_{r}\rangle}\ind_A\big], 
\end{align*}
for all $A\in\cF^{t,0}_r$, where the second equality is by \eqref{W^uBrownMotQ}. This concludes the proof of \eqref{W^uBrownMot} and therefore $\nu_\omega\in\cV_u$ is a reference probability system.
		
It remains to prove that $(\eta^{0,u},\alpha^{0,u},\xi^{0,u})$ is an admissible treble. Since, for $\bP$-a.e.~$\omega$, $\cF^{t,0}_s\subseteq\cF^{u,\omega}_{s}$ for every $s\in[u,T]$ (see Lemma \ref{Lemma:FiltrationsInclusion}), then $\alpha^0$ is $\{\cF^{u,\omega}_{s} \}$-progressively measurable and $\eta^{0,u}$ and $\xi^{0,u}$ are $\{\cF^{u,\omega}_{s} \}$-adapted. By construction $\eta^{0,u}$ is left-continuous and non-decreasing with $\eta^{0,u}\in\{0,1\}$ and $\eta^{0,u}_u=0$, $\bP$-a.s. By Lemma \ref{Lemma:AlmostSureSets}, these properties hold $\bP_\omega$-a.s.~for $\bP$-a.e.~$\omega$ and so $\eta^{0,u}\in\cE^{\nu_\omega}_u$. Similarly, we obtain that $\xi^{0,u}$ is left-continuous, of bounded variation and with $\xi^{0,u}_u=0$, $\bP_\omega$-a.s.~for $\bP$-a.e.~$\omega$. 

If $\xi\in\cX^\nu_t(\bar z - z)$, then by definition $V_{[t,T]}(\xi^{0})\le \bar z-z$, $\bP$-a.s. Recall the process $Z^\xi_s=z+V_{[t,s]}(\xi)$ from \eqref{eq:Z} and let $Z^{\xi,0}_s=z+V_{[t,s]}(\xi^0)$ be its $\{\cF^{t,0}_s\}$-predictable counterpart. Then, using \eqref{ConstRegConditProb} and Lemma \ref{Lemma:AlmostSureSets}, it is immediate to check that $V_{[u,T]}(\xi^{0,u})\le \bar z-Z^{\xi,0}_u(\omega)$, $\bP_\omega$-a.s.\ for $\bP$-a.e.\ $\omega$. Thus, $\xi^{0,u}\in\cX^{\nu_\omega}_u(\bar z-Z^{\xi,0}_u(\omega))$ for $\bP$-a.e.\ $\omega\in\Omega$. If instead $\xi\in\cX^\nu_t$, then
\begin{align*}
\bE\big[\bE_\omega\big[|V_{[u,T]}(\xi^{0,u})|^p\big]\big]=\bE\big[\bE\big[|(V_{[u,T]}(\xi^{0,u})|^p\big|\cF^{t,0}_u\big](\omega)\big]\le \bE\big[|V_{[t,T]}(\xi^{0})|^p\Big]<\infty.
\end{align*}
Therefore,
\[
\bE_\omega\big[|V_{[u,T]}(\xi^{0,u})|^p\big]<\infty,\quad\text{for $\bP$-a.e.\ $\omega\in\Omega$},
\]
which guarantees $\xi^{0,u}\in\cX^{\nu_\omega}_u$ for $\bP$-a.e.\ $\omega$.

We then conclude that the treble $(\eta^{0,u},\alpha^{0,u},\xi^{0,u})$ is admissible in $\nu_\omega$ for $\bP$-a.e.\ $\omega$.
\end{proof}

The final lemma in this appendix is essentially a repetition of \cite[Lemma 11]{claisse2016pseudo}. It is shown that the stochastic integral with respect to the underlying probability $\bP$ is indistinguishable of the stochastic integral with respect to the regular conditional probability $\bP_\omega$ for $\bP$-a.e.~$\omega$. 
Let $\{H_s\}_{s\in[t,T]}$ be an adapted process on $\nu=(\Omega,\cF,\bP,\{\cF^t_s\},W)$ which is $\bP$-a.s.\ square-integrable in $[t,T]$ with respect to the Lebesgue measure. Then,
\begin{align}\label{eq:Pint}
\bP\left(\int_t^T H^2_s \ud s<\infty\right)=1
\end{align}
implies, by Lemma \ref{Lemma:AlmostSureSets}, that
\begin{align*}
\bP_\omega\left(\int_t^T H^2_s \ud s<\infty\right)=1,\:\:\text{for $\bP$-a.e.\ $\omega$}.
\end{align*}
For such a process, we denote by
	\begin{equation}
	\int^\bP_{[u,s]} H_r\ud W^u_r
	\end{equation}
	the stochastic integral of $H$ against $W^u$ when constructed with respect to the probability $\bP$ and, similarly, by
	\begin{equation}
	\int^{\bP_\omega}_{[u,s]} H_r\ud W^u_r
	\end{equation}
	the stochastic integral of $H$ against $W^u$ when constructed with respect to the regular conditional probability $\bP_\omega$. 
	
\begin{lemma}[{\cite[Lemma 11]{claisse2016pseudo}}]\label{Lemma:StochIntRegCondProb}
Let $H$ be an $\{\cF^t_s \}$-adapted process such that \eqref{eq:Pint} holds.
Then, $\bP_\omega$-a.s.\ for $\bP$-a.e.\ $\omega$, we have
\begin{equation}\label{StochIntRegCondProb}
\int^\bP_{[u,s]}H_r \ud W^u_r=\int^{\bP_\omega}_{[u,s]}H_r\ud W^u_r, \quad \text{for all $s\in[u,T]$}.
\end{equation}
\end{lemma}
	\begin{proof}
		Notice that since the stochastic integrals in \eqref{StochIntRegCondProb} are continuous, it is sufficient to prove that they are a modifications. So we now want to prove that, for every $s\in[u,T]$, 
		\begin{equation}\label{StochIntRegCondProb1}
		\int^\bP_{[u,s]}H_r \ud W^u_r=\int^{\bP_\omega}_{[u,s]}H_r\ud W^u_r, \quad \text{$\bP_\omega${-a.s.,} for $\bP${-a.e.}\ $\omega\in\Omega$},
		\end{equation}
		By a standard localisation procedure, we can assume that
		\begin{equation}
		\bE\bigg[\int_u^s H^2_r \ud r\bigg]<\infty,
		\end{equation}
		and so it is sufficient to show that
		\begin{equation}\label{StochIntRegCondProb2}
		\bE_\omega\bigg[\bigg|\int^\bP_{[u,s]}H_r \ud W^u_r-\int^{\bP_\omega}_{[u,s]}H_r\ud W^u_r\bigg|^2 \bigg]=0, \quad \text{$\bP$\text{-a.e.}\ $\omega$, for all $s\in[u,T]$}.
		\end{equation} 
		Notice that in \eqref{StochIntRegCondProb2} since $|\int^\bP_{[u,s]}H_r \ud W^u_r|^2\in L^1(\Omega,\cF,\bP)$ then, by Lemma \ref{Lemma:ExpectRegConditProb}, we also have $|\int^\bP_{[u,s]}H_r \ud W^u_r|^2\in L^1(\Omega,\cF_\omega,\bP_\omega)$ for $\bP$-a.e.~$\omega$. 
		
		Let $(H^{(n)})_{n\in\bN}$ be a sequence of simple processes such that
		\begin{equation}\label{ApproxSimpleProcP}
		\lim_{n\to\infty} \bE\bigg[\int_u^s |H_r-H^{(n)}_r|^2\ud r\bigg]=0.
		\end{equation}
		Then, by Lemma \ref{Lemma:ExpectRegConditProb}, we also have that
		\begin{align}
		\lim_{n\to\infty}\bE\bigg[\bE_\omega\bigg[\int_u^s|H_r-H^{(n)}_r|^2\ud r\bigg]\bigg]&=\lim_{n\to\infty}\bE\bigg[\bE\bigg[\int_u^s|H_r-H^{(n)}_r|^2\ud r\bigg|\cF^t_u\bigg](\omega)\bigg]\\
		&=\lim_{n\to\infty}\bE\bigg[\int_u^s|H_r-H^{(n)}_r|^2\ud r\bigg]=0.\notag
		\end{align}
		Thus, up to a subsequence which we still denote by $(H^{(n)})_{n\in\bN}$, we have 
		\begin{equation}\label{ApproxSimpleProcPomega}
		\lim_{n\to\infty} \bE_\omega\bigg[\int_u^s |H_r-H^{(n)}_r|^2\ud r\bigg]=0, \quad \bP\text{-a.e. }\omega.
		\end{equation}
By Ito's isometry with respect to the regular conditional probability $\bP_\omega$ and \eqref{ApproxSimpleProcPomega}, we have that
		\begin{align}\label{eq:isom}
		\lim_{n\to\infty}\bE_\omega\bigg[\bigg|\int^{\bP_\omega}_{[u,s]}(H^{(n)}_r-H_r)\ud W^u_r\bigg|^2\bigg]&=\lim_{n\to\infty} \bE_\omega\bigg[\int_u^s |H^{(n)}_r-H_r|^2\ud r\bigg]=0, \quad \bP\text{-a.e. }\omega.
		\end{align}

Similarly, by tower property (with Lemma \ref{Lemma:ExpectRegConditProb}) and It\^o's isometry we have
		\begin{align*}
		&\lim_{n\to\infty}\bE\bigg[\bE_\omega\bigg[\bigg|\int^{\bP}_{[u,s]}(H^{(n)}_r-H_r)\ud W^u_r\bigg|^2\bigg]\bigg]=\lim_{n\to\infty}\bE\bigg[\bE\bigg[\bigg|\int^{\bP}_{[u,s]}(H^{(n)}_r-H_r)\ud W^u_r\bigg|^2\bigg|\cF^t_u\bigg](\omega)\bigg]\\
		&=\lim_{n\to\infty}\bE\bigg[\bigg|\int^{\bP}_{[u,s]}(H^{(n)}_r-H_r)\ud W^u_r\bigg|^2\bigg]=\lim_{n\to\infty} \bE\bigg[\int_u^s |H^{(n)}_r-H_r|^2\ud r\bigg]=0,\notag
		\end{align*}
		where in the final equality we have used \eqref{ApproxSimpleProcP}. Again, up to a new subsequence which we still denote by $(H^{(n)})_{n\in\bN}$, we have
		\begin{equation}\label{ApproxNewSubseq}
		\lim_{n\to\infty}\bE_\omega\bigg[\bigg|\int^{\bP}_{[u,s]}(H^{(n)}_r-H_r)\ud W^u_r\bigg|^2\bigg]=0, \quad \bP\text{-a.e. }\omega.
		\end{equation}
		
		Now notice that
		\begin{align*}
		\int^\bP_{[u,s]}H_r \ud W^u_r-\int^{\bP_\omega}_{[u,s]}H_r\ud W^u_r=&\int^\bP_{[u,s]}H^{(n)}_r \ud W^u_r-\int^{\bP_\omega}_{[u,s]}H^{(n)}_r\ud W^u_r\nonumber\\	
		&+\int^{\bP_\omega}_{[u,s]}(H^{(n)}_r-H_r)\ud W^u_r-\int^{\bP}_{[u,s]}(H^{(n)}_r-H_r)\ud W^u_r,
		\end{align*}
		where the first difference on the right-hand side is null as stochastic integrals for simple processes are defined pathwise. Hence, for $\bP$-a.e.~$\omega$,
		\begin{align}\label{AddSubtractH^n}
		\bE_\omega\bigg[\bigg|\int^\bP_{[u,s]}H_r \ud W^u_r-\int^{\bP_\omega}_{[u,s]}H_r\ud W^u_r\bigg|^2 \bigg]&\leq 2\bE_\omega\bigg[\bigg|\int^{\bP_\omega}_{[u,s]}(H^{(n)}_r-H_r)\ud W^u_r\bigg|^2\bigg]\\
		&\hspace{10pt}+2\bE_\omega\bigg[\bigg|\int^{\bP}_{[u,s]}(H^{(n)}_r-H_r)\ud W^u_r\bigg|^2\bigg]\notag
		\end{align}
		Thus, by letting $n\uparrow\infty$ in \eqref{AddSubtractH^n} and using \eqref{eq:isom} and \eqref{ApproxNewSubseq}, we obtain the desired result \eqref{StochIntRegCondProb2}.
	\end{proof}
	
	\bibliography{bibfile}{}
	\bibliographystyle{abbrv} 
	
\end{document}